\documentclass[12pt,a4paper,reqno]{amsart}
\usepackage{amssymb}
\usepackage{amscd}
\usepackage{enumerate}
\usepackage{graphicx}
\usepackage{siunitx}
\usepackage{tikz-cd}
\usepackage{bm}
\numberwithin{equation}{section}

\usepackage{mathtools}
\usepackage[tableposition=top]{caption}
\usepackage{booktabs,dcolumn}

\DeclareFontFamily{OT1}{rsfs}{}
\DeclareFontShape{OT1}{rsfs}{n}{it}{<-> rsfs10}{}
\DeclareMathAlphabet{\mathscr}{OT1}{rsfs}{n}{it}

\theoremstyle{plain}

\newtheorem{theorem}{Theorem}[section]
\newtheorem{proposition}[theorem]{Proposition}
\newtheorem{lemma}[theorem]{Lemma}
\newtheorem{corollary}[theorem]{Corollary}
\newtheorem{conjecture}[theorem]{Conjecture}
\newtheorem{heuristic}[theorem]{Heuristic}

\theoremstyle{definition}

\newtheorem{definition}[theorem]{Definition}
\newtheorem{remark}[theorem]{Remark}

\renewcommand\P{\mathbb{P}}
\newcommand\E{\mathbb{E}}

\newcommand\R{\mathbb{R}}
\newcommand\Z{\mathbb{Z}}
\newcommand\N{\mathbb{N}}

\renewcommand\a{\mathbf{a}}
\renewcommand\b{\mathbf{b}}
\renewcommand\j{\mathbf{j}}
\renewcommand\k{\mathbf{k}}
\renewcommand\v{\mathbf{v}}
\renewcommand\t{\mathbf{t}}
\renewcommand\r{\mathbf{r}}
\renewcommand\l{\mathbf{l}}
\newcommand\X{\mathbf{X}}

\newcommand\Y{\mathbf{Y}}

\newcommand\Unif{\mathbf{Unif}}
\newcommand\Geom{\mathbf{Geom}}
\newcommand\Log{\mathbf{Log}}
\newcommand\Pascal{\mathbf{Pascal}}
\newcommand\Syrac{\mathbf{Syrac}}
\newcommand\Hold{\mathbf{Hold}}
\newcommand\C{\mathbb{C}}
\newcommand\Q{\mathbb{Q}}
\newcommand\TV{{\operatorname{TV}}}
\newcommand\Aff{{\operatorname{Aff}}}
\newcommand\Pass{{\operatorname{Pass}}}
\newcommand\dist{{\operatorname{dist}}}
\newcommand\Col{{\operatorname{Col}}}
\newcommand\Syr{{\operatorname{Syr}}}
\newcommand\Osc{{\operatorname{Osc}}}
\newcommand\eps{\varepsilon}

\renewcommand{\mod}{\bmod}

\begin{document}

\title[Collatz orbits attain almost bounded values]{Almost all orbits of the Collatz map attain almost bounded values}

\author{Terence Tao}
\address{UCLA Department of Mathematics, Los Angeles, CA 90095-1555.}
\email{tao@math.ucla.edu}



\begin{abstract}  Define the \emph{Collatz map} $\Col \colon \N+1 \to \N+1$ on the positive integers $\N+1 = \{1,2,3,\dots\}$ by setting $\Col(N)$ equal to $3N+1$ when $N$ is odd and $N/2$ when $N$ is even, and let $\Col_{\min}(N) \coloneqq \inf_{n \in \N} \Col^n(N)$ denote the minimal element of the Collatz orbit $N, \Col(N), \Col^2(N), \dots$. The infamous \emph{Collatz conjecture} asserts that $\Col_{\min}(N)=1$ for all $N \in \N+1$.  Previously, it was shown by Korec that for any $\theta > \frac{\log 3}{\log 4} \approx 0.7924$, one has $\Col_{\min}(N) \leq N^\theta$ for almost all $N \in \N+1$ (in the sense of natural density).  In this paper we show that for \emph{any} function $f \colon \N+1 \to \R$ with $\lim_{N \to \infty} f(N)=+\infty$, one has $\Col_{\min}(N) \leq f(N)$ for almost all $N \in \N+1$ (in the sense of logarithmic density).  Our proof proceeds by establishing a stabilisation property for a certain first passage random variable associated with the Collatz iteration (or more precisely, the closely related Syracuse iteration), which in turn follows from estimation of the characteristic function of a certain skew random walk on a $3$-adic cyclic group $\Z/3^n\Z$ at high frequencies.  This estimation is achieved by studying how a certain two-dimensional renewal process interacts with a union of triangles associated to a given frequency.
\end{abstract}

\maketitle


\section{Introduction}

\subsection{Statement of main result}

Let $\N \coloneqq \{0,1,2,\dots\}$ denote the natural numbers, so that $\N+1 = \{1,2,3,\dots\}$ are the positive integers.  The \emph{Collatz map} $\Col \colon \N+1 \to \N+1$ is defined by setting $\Col(N) \coloneqq 3N+1$ when $N$ is odd and $\Col(N) \coloneqq N/2$ when $N$ is even.  For any $N \in \N+1$, let $\Col_{\min}(N) \coloneqq \min \Col^\N(N) = \inf_{n \in \N} \Col^n(N)$ denote the minimal element of the Collatz orbit $\Col^\N(N) \coloneqq \{ N, \Col(N), \Col^2(N), \dots\}$.    We have the infamous \emph{Collatz conjecture} (also known as the $3x+1$ conjecture):

\begin{conjecture}[Collatz conjecture]\label{collatz}  We have $\Col_{\min}(N)=1$ for all $N \in \N+1$.
\end{conjecture}

We refer the reader to \cite{lag}, \cite{chamber} for extensive surveys and historical discussion of this conjecture.

While the full resolution of Conjecture \ref{collatz} remains well beyond reach of current methods, some partial results are known.
Numerical computation has verified $\Col_{\min}(N)=1$ for all $N \leq 5.78 \times 10^{18}$ \cite{olive}, for all $N \leq 10^{20}$ \cite{eric}, and most recently for all $N \leq 2^{68} \approx 2.95 \times 10^{20}$ \cite{barina}, while Krasikov and Lagarias \cite{kl} showed that
$$ \# \{ N \in \N+1 \cap [1,x]: \Col_{\min}(N) = 1 \} \gg x^{0.84}$$
for all sufficiently large $x$, where $\# E$ denotes the cardinality of a finite set $E$, and our conventions for asymptotic notation are set out in Section \ref{notation-sec}.  In this paper we will focus on a different type of partial result, in which one establishes upper bounds on the minimal orbit value $\Col_{\min}(N)$ for ``almost all'' $N \in \N+1$.  For technical reasons, the notion of ``almost all'' that we will use here is based on logarithmic density, which has better approximate multiplicative invariance properties than the more familiar notion of natural density (see \cite{tao:chowla} for a related phenomenon in a more number-theoretic context).  Due to the highly probabilistic nature of the arguments in this paper, we will define logarithmic density using the language of probability theory.

\begin{definition}[Almost all]  Given a finite non-empty subset $R$ of $\N+1$, we define\footnote{In this paper all random variables will be denoted by boldface symbols, to distinguish them from purely deterministic quantities that will be denoted by non-boldface symbols.  When it is only the distribution of the random variable that is important, we will use multi-character boldface symbols such as $\Log$, $\Unif$, or $\Geom$ to denote the random variable, but when the dependence or independence properties of the random variable are also relevant, we shall usually use single-character boldface symbols such as $\a$ or $\j$ instead.}  $\Log(R)$ to be a random variable taking values in $R$ with the logarithmically uniform distribution
$$ \P( \Log(R) \in A ) = \frac{\sum_{N \in A \cap R} \frac{1}{N}}{\sum_{N \in R} \frac{1}{N}}$$
for all $A \subset \N+1$.  The \emph{logarithmic density} of a set $A \subset \N+1$ is then defined to be $\lim_{x \to \infty} \P( \Log(\N+1 \cap [1,x]) \in A )$, provided that the limit exists.  We say that a property $P(N)$ holds for \emph{almost all} $N \in \N+1$ if $P(N)$ holds for $N$ in a subset of $\N+1$ of logarithmic density $1$, or equivalently if
$$ \lim_{x \to \infty} \P( P( \Log( \N+1 \cap [1,x] ) ) ) = 1.$$
\end{definition}

In Terras \cite{terras} (and independently Everett \cite{everett}) it was shown that $\Col_{\min}(N) < N$ for almost all $N$.  This was improved by Allouche \cite{allouche} to $\Col_{\min}(N) < N^\theta$ for almost all $N$, and any fixed constant $\theta>\frac{3}{2} - \frac{\log 3}{\log 2} \approx 0.869$; the range of $\theta$ was later extended to $\theta > \frac{\log 3}{\log 4} \approx 0.7924$ by Korec \cite{korec}.  (Indeed, in these results one can use natural density instead of logarithmic density to define ``almost all''.)  It is tempting to try to iterate these results to lower the value of $\theta$ further.  However, one runs into the difficulty that the uniform (or logarithmic) measure does not enjoy any invariance properties with respect to the Collatz map: in particular, even if it is true that $\Col_{\min}(N) < x^\theta$ for almost all $N \in [1,x]$, and $\Col_{\min}(N') \leq x^{\theta^2}$ for almost all $N' \in [1, x^\theta]$, the two claims cannot be immediately concatenated to imply that $\Col_{\min}(N) \leq x^{\theta^2}$ for almost all $N \in [1,x]$, since the Collatz iteration may send almost all of $[1,x]$ into a very sparse subset of $[1,x^\theta]$, and in particular into the exceptional set of the latter claim $\Col_{\min}(N') \leq x^{\theta^2}$.

Nevertheless, in this paper we show that it is possible to locate an alternate probability measure (or more precisely, a family of probability measures) on the natural numbers with enough invariance properties that an iterative argument does become fruitful.  More precisely, the main result of this paper is the following improvement of these ``almost all'' results.

\begin{theorem}[Almost all Collatz orbits attain almost bounded values]\label{main}  Let $f \colon \N + 1 \to \R$ be any function with $\lim_{N \to\infty} f(N) = +\infty$.  Then one has $\Col_{\min}(N) < f(N)$ for almost all $N \in \N+1$ (in the sense of logarithmic density).
\end{theorem}

Thus for instance one has $\Col_{\min}(N) < \log\log\log\log N$ for almost all $N$.

\begin{remark}  One could ask whether it is possible to sharpen the conclusion of Theorem \ref{main} further, to assert that there is an absolute constant $C_0$ such that $\Col_{\min}(N) \leq C_0$ for almost all $N \in \N+1$.  However this question is likely to be almost as hard to settle as the full Collatz conjecture, and out of reach of the methods of this paper.  Indeed, suppose for any given $C_0$ that there existed an orbit $\Col^\N(N_0) = \{N_0, \Col(N_0), \Col^2(N_0),\dots\}$ that never dropped below $C_0$ (this is the case if there are infinitely many periodic orbits, or if there is at least one unbounded orbit).  Then probabilistic heuristics (such as \eqref{snn} below) suggest that for a positive density set of $N \in \N+1$, the orbit $\Col^\N(N) = \{N, \Col(N), \Col^2(N), \dots\}$ should encounter one of the elements $\Col^n(N_0)$ of the orbit of $N_0$ before going below $C_0$, and then the orbit of $N$ will never dip below $C_0$.  However, Theorem \ref{main} is easily seen\footnote{Indeed, if the latter assertion failed, then there exists a $\delta$ such that the set $\{ N \in \N+1: \Col_{\min}(N) \leq C\}$ has lower logarithmic density less than $1-\delta$ for every $C$.  A routine diagonalisation argument then shows that there exists a function $f$ growing to infinity such that $\{ N \in \N+1: \Col_{\min}(N) \leq f(N)\}$ has lower logarithmic density at most $1-\delta$, contradicting Theorem \ref{main}.} to be equivalent to the assertion that for any $\delta>0$, there exists a constant $C_\delta$ such that $\Col_{\min}(N) \leq C_\delta$ for all $N$ in a subset of $\N+1$ of lower logarithmic density (in which the limit in the definition of logarithmic density is replaced by the limit inferior) at least $1-\delta$; in fact (see Theorem \ref{main-alt}) our arguments give a constant of the form $C_\delta \ll \exp(\delta^{-O(1)})$, and it may be possible to refine the subset so that the logarithmic density (as opposed to merely the lower logarithmic density) exists and is at least $1-\delta$.  In particular\footnote{We thank Ben Green for this observation.}, it is possible in principle that a sufficiently explicit version of the arguments here, when combined with numerical verification of the Collatz conjecture, can be used to show that the Collatz conjecture holds for a set of $N$ of positive logarithmic density.  Also, it is plausible that some refinement of the arguments below will allow one to replace logarithmic density by natural density in the definition of ``almost all''.
\end{remark}

\subsection{Syracuse formulation}

We now discuss the methods of proof of Theorem \ref{main}.  It is convenient to replace the Collatz map $\Col \colon \N+1 \to \N+1$ with a slightly more tractable acceleration $N \mapsto \Col^{f(N)}(N)$ of that map.  One common instance of such an acceleration in the literature is the map $\Col_2 \colon \N+1 \to \N+1$, defined by setting $\Col_2(N) \coloneqq \Col^2(N) = \frac{3N+1}{2}$ when $N$ is odd and $\Col_2(N) \coloneqq \frac{N}{2}$ when $N$ is even.  Each iterate of the map $\Col_2$ performs exactly one division by $2$, and for this reason $\Col_2$ is a particularly convenient choice of map when performing ``$2$-adic'' analysis of the Collatz iteration.  It is easy to see that $\Col_{\min}(N) = (\Col_2)_{\min}(N)$ for all $N \in \N+1$, so all the results in this paper concerning $\Col$ may be equivalently reformulated using $\Col_2$.  The triple iterate $\Col^3$ was also recently proposed as an acceleration in \cite{carletti}.  However, the methods in this paper will rely instead on ``$3$-adic'' analysis, and it will be preferable to use an acceleration of the Collatz map (first appearing to the author's knowledge in \cite{crandall}) which performs exactly one multiplication by $3$ per iteration.  More precisely, let $2\N+1 = \{1,3,5,\dots\}$ denote the odd natural numbers, and define the \emph{Syracuse map} $\Syr \colon 2\N+1 \to 2\N+1$ (OEIS A075677) to be the largest odd number dividing $3N+1$; thus for instance
$$ \Syr(1)=1; \quad \Syr(3) = 5; \quad \Syr(5) = 1; \quad \Syr(7) = 11.$$
Equivalently, one can write
\begin{equation}\label{S-def}
 \Syr(N) = \Col^{\nu_2(3N+1)+1}(N) = \Aff_{\nu_2(3N+1)}(N)
\end{equation}
where for each positive integer $a \in \N+1$, $\Aff_a\colon \R \to \R$ denotes the affine map
$$ \Aff_a(x) \coloneqq \frac{3x+1}{2^a}$$
and for each integer $M$ and each prime $p$, the \emph{$p$-valuation} $\nu_p(M)$ of $M$ is defined as the largest natural number $a$ such that $p^a$ divides $M$ (with the convention $\nu_p(0) = +\infty$).  (Note that $\nu_2(3N+1)$ is always a positive integer when $N$ is odd.)  For any $N \in 2\N+1$, let $\Syr_{\min}(N) \coloneqq \min \Syr^\N(N)$ be the minimal element of the Syracuse orbit
$$\Syr^\N(N) \coloneqq \{ N, \Syr(N), \Syr^2(N), \dots\}.$$
This Syracuse orbit $\Syr^\N(N)$ is nothing more than the odd elements of the corresponding Collatz orbit $\Col^\N(N)$, and from this observation it is easy to verify the identity
\begin{equation}\label{c-ident}
 \Col_{\min}(N) = \Syr_{\min}( N / 2^{\nu_2(N)} )
\end{equation}
for any $N \in \N+1$.  Thus, the Collatz conjecture can be equivalently rephrased as

\begin{conjecture}[Collatz conjecture, Syracuse formulation]\label{collatz-syr}  We have $\Syr_{\min}(N)=1$ for all $N \in 2\N+1$.
\end{conjecture}

We may similarly reformulate Theorem \ref{main} in terms of the Syracuse map.  We say that a property $P(N)$ holds for \emph{almost all} $N \in 2\N+1$  if
$$ \lim_{x \to \infty} \P( P(\Log( 2\N+1 \cap [1,x] ) )  )= 1,$$
or equivalently if $P(N)$ holds for a set of odd natural numbers of logarithmic density $1/2$.  Theorem \ref{main} is then equivalent to

\begin{theorem}[Almost all Syracuse orbits attain almost bounded values]\label{main-syr}  Let $f\colon 2\N + 1 \to \R$ be a function with $\lim_{N \to\infty} f(N) = +\infty$.  Then one has $\Syr_{\min}(N) < f(N)$ for almost all $N \in 2\N+1$.
\end{theorem}

Indeed, if Theorem \ref{main-syr} holds and $f\colon \N +1 \to \R$ is such that $\lim_{N \to \infty} f(N) = +\infty$, then from \eqref{c-ident} we see that for any $a \in \N$, the set of $N \in \N+1$ with $\nu_2(N) = a$ and $\Col_{\min}( N ) = \Syr_{\min}(N/2^a) < f(N)$ has logarithmic density $2^{-a}$.  Summing over any finite range $0 \leq a \leq a_0$ we obtain a set of logarithmic density $1 - 2^{-a_0}$ on which the claim $\Col_{\min}(N) < f(N)$ holds, and on sending $a_0$ to infinity one obtains Theorem \ref{main}.  The converse implication (which we will not need) is also straightforward and left to the reader.

The iterates $\Syr^n$ of the Syracuse map can be described explicitly as follows.  For any finite tuple $\vec a = (a_1,\dots,a_n) \in (\N+1)^n$ of positive integers, we define the composition $\Aff_{\vec a} = \Aff_{a_1,\dots,a_n}\colon \R \to \R$ to be the affine map
$$ \Aff_{a_1,\dots,a_n}(x) \coloneqq \Aff_{a_n}( \Aff_{a_{n-1}}( \dots (\Aff_{a_1}(x)) \dots )).$$
A brief calculation shows that
\begin{equation}\label{a-form}
\Aff_{a_1,\dots,a_n}(x) = 3^n 2^{-|\vec a|} x + F_n(\vec a)
\end{equation}
where the \emph{size} $|\vec a|$ of a tuple $\vec a$ is defined as
\begin{equation}\label{length}
 |\vec a| \coloneqq a_1 + \dots + a_n,
\end{equation}
and we define the \emph{$n$-Syracuse offset map} $F_n\colon (\N+1)^n \to \Z[\frac{1}{2}]$ to be the function
\begin{equation}\label{fn-def}
\begin{split}
 F_n(\vec a) &\coloneqq \sum_{m=1}^n 3^{n-m} 2^{-a_{[m,n]}} \\
&= 3^{n-1} 2^{-a_{[1,n]}} + 3^{n-2} 2^{-a_{[2,n]}} + \dots + 3^1 2^{-a_{[n-1,n]}} + 2^{-a_n},
\end{split}
\end{equation}
where we adopt the summation notation
\begin{equation}\label{sum}
a_{[j,k]} \coloneqq \sum_{i=j}^k a_i
\end{equation}
for any $1 \leq j \leq k \leq n$, thus for instance $|\vec a| = a_{[1,n]}$.  The $n$-Syracuse offset map $F_n$ takes values in the ring $\Z[\frac{1}{2}] \coloneqq \{ \frac{M}{2^a}: M \in\Z, a \in \N \}$ formed by adjoining $\frac{1}{2}$ to the integers.

By iterating \eqref{S-def} and then using \eqref{a-form}, we conclude that
\begin{equation}\label{s-iter}
\Syr^n(N) = \Aff_{\vec a^{(n)}(N)}(N) = 3^n 2^{-|\vec a^{(n)}(N)|} N + F_n(\vec a^{(n)}(N))
\end{equation}
for any $N \in 2\N+1$ and $n \in \N$, where we define \emph{$n$-Syracuse valuation} $\vec a^{(n)}(N) \in (\N+1)^n$ of $N$ to be the tuple
\begin{equation}\label{van}
 \vec a^{(n)}(N) \coloneqq \left(\nu_2(3N+1), \nu_2(3\Syr(N)+1), \dots, \nu_2(3\Syr^{n-1}(N)+1)\right).
\end{equation}
This tuple is referred to as the \emph{$n$-path} of $N$ in \cite{ks}.

The identity \eqref{s-iter} asserts that $\Syr^n(N)$ is the image of $N$ under a certain affine map $\Aff_{\vec a^{(n)}(N)}$ that is determined by the $n$-Syracuse valuation $\vec a^{(n)}(N)$ of $N$.  This suggests that in order to understand the behaviour of the iterates $\Syr^n(N)$ of a typical large number $N$, one needs to understand the behaviour of $n$-Syracuse valuation $\vec a^{(n)}(N)$, as well as the $n$-Syracuse offset map $F_n$.  For the former, we can gain heuristic insight by observing that for a positive integer $a$, the set of odd natural numbers $N \in 2\N+1$ with $\nu_2(3N+1)=a$ has (logarithmic) relative density $2^{-a}$.  To model this probabilistically, we introduce the following probability distribution:

\begin{definition}[Geometric random variable]\label{geom}  If $\mu > 1$, we use $\Geom(\mu)$ to denote a geometric random variable of mean $\mu$, that is to say $\Geom(\mu)$ takes values in $\N+1$ with
$$ \P( \Geom(\mu) = a ) = \frac{1}{\mu} \left( \frac{\mu-1}{\mu} \right)^{a-1}$$
for all $a \in \N+1$.  We use $\Geom(\mu)^n$ to denote a tuple of $n$ independent, identically distributed (or \emph{iid} for short) copies of $\Geom(\mu)$, and use $\X \equiv \Y$ to denote the assertion that two random variables $\X,\Y$ have the same distribution.  Thus for instance
$$ \P( \a = a ) = 2^{-a}$$
whenever $\a \equiv \Geom(2)$ and $a \in \N+1$, and more generally
$$ \P( \vec \a = \vec a ) = 2^{-|\vec a|}$$
whenever $\vec \a \equiv \Geom(2)^n$ and $\vec a \in (\N+1)^n$ for some $n \in \N$.
\end{definition}

In this paper, the only geometric random variables we will actually use are $\Geom(2)$ and $\Geom(4)$.

We will then be guided by the following heuristic:

\begin{heuristic}[Valuation heuristic]\label{Val}  If $N$ is a ``typical'' large odd natural number, and $n$ is much smaller than $\log N$, then the $n$-Syracuse valuation $\vec a^{(n)}(N)$ behaves like $\Geom(2)^n$.
\end{heuristic}

We can make this heuristic precise as follows.  Given two random variables $\X,\Y$ taking values in the same discrete space $R$, we define the \emph{total variation} $d_\TV(\X,\Y)$ between the two variables to be the total variation of the difference in the probability measures, thus
\begin{equation}\label{tv-1}
 d_\TV(\X,\Y) \coloneqq \sum_{r \in R} |\P( \X = r ) - \P( \Y = r )|.
\end{equation}
Note that
\begin{equation}\label{tv-2}
 \sup_{E \subset R} |\P(\X \in E) - \P(\Y \in E)| \leq d_\TV(\X,\Y) \leq 2 \sup_{E \subset R} |\P(\X \in E) - \P(\Y \in E)|.
\end{equation}
For any finite non-empty set $R$, let $\Unif(R)$ denote a uniformly distributed random variable on $R$.  Then we have the following result, proven in Section \ref{rach-sec}:

\begin{proposition}[Distribution of $n$-Syracuse valuation]\label{rach}  Let $n \in \N$, and let $\mathbf{N}$ be a random variable taking values in $2\N+1$.  Suppose there exist an absolute constant $c_0 > 0$ and some natural number $n' \geq (2+c_0) n$ such that $\mathbf{N} \mod 2^{n'}$ is approximately uniformly distributed in the odd residue classes $(2\Z+1)/2^{n'}\Z$ of $\Z/2^\ell\Z$, in the sense that
\begin{equation}\label{boy}
 d_\TV( \mathbf{N} \mod 2^{n'}, \Unif((2\Z+1)/2^{n'}\Z) ) \ll 2^{-n'}.
\end{equation}
Then
\begin{equation}\label{girl}
 d_\TV( \vec a^{(n)}(\mathbf{N}), \Geom(2)^n ) \ll 2^{-c_1 n}
\end{equation}
for some absolute constant $c_1>0$ (depending on $c_0$).  The implied constants in the asymptotic notation are also permitted to depend on $c_0$.
\end{proposition}

Informally, this proposition asserts that Heuristic \ref{Val} is justified whenever $N$ is expected to be uniformly distributed modulo $2^{n'}$ for some $n'$ slightly larger than $2n$.  The hypothesis \eqref{boy} is somewhat stronger than what is actually needed for the conclusion \eqref{girl} to hold, but this formulation of the implication will suffice for our applications.  We will apply this proposition in Section \ref{main-sec}, not to the original logarithmic distribution $\Log(2\N+1 \cap [1,x])$ (which has too heavy a tail near $1$ for the hypothesis \eqref{boy} to apply), but to the variant $\Log( 2\N+1 \cap [y,y^\alpha])$ for some large $y$ and some $\alpha>1$ close to $1$.

\begin{remark}\label{cat}  Another standard way in the literature to justify Heuristic \ref{Val} is to consider the Syracuse dynamics on the $2$-adic integers $\Z_2 \coloneqq \varprojlim_m \Z/2^m\Z$, or more precisely on the odd $2$-adics $2\Z_2+1$.  As the $2$-valuation $\nu_2$ remains well defined on (almost all of) $\Z_2$, one can extend the Syracuse map $\Syr$ to a map on $2\Z_2+1$.  As is well known (see e.g., \cite{lag}), the Haar probability measure on $2\Z_2+1$ is preserved by this map, and if $\mathbf{Haar}(2\Z_2+1)$ is a random element of $2\Z_2+1$ drawn using this measure, then it is not difficult (basically using the $2$-adic analogue of Lemma \ref{iter-lem} below) to show that the random variables $\nu_2( 3\Syr^{j}(\mathbf{Haar}(2\Z_2+1)) + 1)$ for $j \in \N$ are iid copies of $\Geom(2)$. However, we will not use this $2$-adic formalism in this paper.
\end{remark}

In practice, the offset $F_n(\vec a)$ is fairly small (in an Archimedean sense) when $n$ is not too large; indeed, from \eqref{fn-def} we have
\begin{equation}\label{fn-bound}
0 \leq F_n(\vec a) \leq 3^n 2^{-a_n} \leq 3^n
\end{equation}
for any $n \in \N$ and $\vec a \in (\N+1)^n$.  For large $N$, we then conclude from \eqref{s-iter} that we have the heuristic approximation
$$ \Syr^n(N) \approx 3^n 2^{-|\vec a^{(n)}(N)|} N $$
and hence by Heuristic \ref{Val} we expect $\Syr^n(N)$ to behave statistically like
\begin{equation}\label{snusnu}
 \Syr^n(N) \approx 3^n 2^{-|\Geom(2)^n|} N = N \exp( n \log 3 - |\Geom(2)^n| \log 2 )
\end{equation}
if $n$ is much smaller than $\log N$.
One can view the sequence $n \mapsto n \log 3 - |\Geom(2)^n| \log 2$ as a simple random walk on $\R$ with negative drift $\log 3 - 2 \log 2 = \log \frac{3}{4}$.  From the law of large numbers we expect to have
\begin{equation}\label{lln}
|\Geom(2)^n| \approx 2n
\end{equation}
most of the time, thus we are led to the heuristic prediction
\begin{equation}\label{snn}
 \Syr^n(N) \approx (3/4)^n N
\end{equation}
for typical $N$; indeed, from the central limit theorem or the Chernoff bound we in fact expect the refinement
\begin{equation}\label{snn-2}
 \Syr^n(N) = \exp( O(n^{1/2}) ) (3/4)^n N
\end{equation}
for ``typical'' $N$.  In particular, we expect the Syracuse orbit $N, \Syr(N), \Syr^2(N), \dots$ to decay geometrically in time for typical $N$, which underlies the usual heuristic argument supporting the truth of Conjecture \ref{collatz}; see \cite{lw}, \cite{kont} for further discussion.  We remark that the multiplicative inaccuracy of $\exp( O(n^{1/2}) )$ in \eqref{snn-2} is the main reason why we work with logarithmic density instead of natural density in this paper (see also \cite{km}, \cite{ls} for a closely related ``Benford's law'' phenomenon).

\subsection{Reduction to a stablisation property for first passage locations}

Roughly speaking, Proposition \ref{rach} lets one obtain good control on the Syracuse iterates $\Syr^n(N)$ for almost all $N$ and for times $n$ up to $c \log N$ for a small absolute constant $c$.  This already can be used in conjunction with a rigorous version of \eqref{snn} or \eqref{snn-2} to recover the previously mentioned result $\Syr_{\min}(N) \leq N^{1-c}$ for almost all $N$ and some absolute constant $c>0$; see Section \ref{main-sec} for details.  In the language of evolutionary partial differential equations, these type of results can be viewed as analogous to ``almost sure local wellposedness'' results, in which one has good short-time control on the evolution for almost all choices of initial condition $N$.

In this analogy, Theorem \ref{main-syr} then corresponds to an ``almost sure almost global wellposedness'' result, where one needs to control the solution for times so large that the evolution gets arbitrary close to the bounded state $N=O(1)$.  To bootstrap from almost sure local wellposedness to almost sure almost global wellposedness, we were inspired by the work of Bourgain \cite{bourgain}, who demonstrated an almost sure global wellposedness result for a certain nonlinear Schr\"odinger equation by combining local wellposedness theory with a construction of an invariant probability measure for the dynamics.  Roughly speaking, the point was that the invariance of the measure would almost surely keep the solution in a ``bounded'' region of the state space for arbitrarily long times, allowing one to iterate the local wellposedness theory indefinitely.

In our context, we do not expect to have any useful invariant probability measures for the dynamics due to the geometric decay \eqref{snn} (and indeed Conjecture \ref{collatz-syr} would imply that the only invariant probability measure is the Dirac measure on $\{1\}$).  Instead, we can construct a \emph{family} of probability measures $\nu_x$ which are \emph{approximately} transported to each other by certain iterations of the Syracuse map (by a variable amount of time).  More precisely,  given a threshold $x \geq 1$ and an odd natural number $N \in 2\N+1$, define the \emph{first passage time}
$$ T_x(N) \coloneqq \inf \{ n \in \N: \Syr^n(N) \leq x \},$$
with the convention that $T_x(N) \coloneqq +\infty$ if $\Syr^n(N) > x$ for all $n$.  (Of course, if Conjecture \ref{collatz-syr} were true, this latter possibility could not occur, but we will not be assuming this conjecture in our arguments.)  We then define the \emph{first passage location}
$$ \Pass_x(N) \coloneqq \Syr^{T_x(N)}(N)$$
with the (somewhat arbitrary and artificial) convention that $\Syr^\infty(N) \coloneqq 1$; thus $\Pass_x(N)$ is the first location of the Syracuse orbit $\Syr^\N(N)$ that falls inside $[1,x]$, or $1$ if no such location exists; if we ignore the latter possibility, then $\Pass_x$ can be viewed as a further acceleration of the Collatz and Syracuse maps.  We will also need a constant $\alpha > 1$ sufficiently close to one. The precise choice of this parameter is not critical, but for sake of concreteness we will set
\begin{equation}\label{alpha-def-const}
\alpha \coloneqq 1.001.
\end{equation}
The key proposition is then

\begin{proposition}[Stabilisation of first passage]\label{transport}  For any $y$ with $2\N+1 \cap [y,y^\alpha]$ is non-empty (and in particular, for any sufficiently large $y$), let $\mathbf{N}_y$ be a random variable with distribution $\mathbf{N}_y \equiv \Log( 2\N+1 \cap [y,y^\alpha] )$.  Then for sufficiently large $x$, we have the estimates
\begin{equation}\label{fail}
\P( T_x(\mathbf{N}_y) = +\infty ) \ll x^{-c}
\end{equation}
for $y = x^\alpha, x^{\alpha^2}$, and also
\begin{equation}\label{tv}
d_\TV( \Pass_x( \mathbf{N}_{x^\alpha} ), \Pass_x( \mathbf{N}_{x^{\alpha^2}} ) ) \ll \log^{-c} x
\end{equation}
for some absolute constant $c>0$. (The implied constants here are also absolute.)
\end{proposition}

Informally, this theorem asserts that the Syracuse orbits of $ \mathbf{N}_{x^\alpha} $ and $ \mathbf{N}_{x^{\alpha^2}}$ are almost indistinguishable from each other once they pass $x$, as long as one synchronises the orbits so that they simultaneously pass $x$ for the first time.  In Section \ref{main-reduce} we shall see how Theorem \ref{main-syr} (and hence Theorem \ref{main}) follows from Proposition \ref{transport}; basically the point is that \eqref{fail}, \eqref{tv} imply that the first passage map $\Pass_x$ approximately maps the distribution $\nu_{x^\alpha}$ of $\Pass_{x^{\alpha}}( \mathbf{N}_{x^{\alpha^2}} )$ to the distribution $\nu_x$ of $\Pass_{x}( \mathbf{N}_{x^{\alpha}} )$, and one can then iterate this to map almost all of the probabilistic mass of $\mathbf{N}_y$ for large $y$ to be arbitrarily close to the bounded state $N=O(1)$.  The implication is very general and does not use any particular properties of the Syracuse map beyond \eqref{fail}, \eqref{tv}.

The estimate \eqref{fail} is easy to establish; it is \eqref{tv} that is the most important and difficult conclusion of Proposition \ref{transport}.  We remark that the bound of $O(\log^{-c} x)$ in \eqref{tv} is stronger than is needed for this argument; any bound of the form $O((\log\log x)^{-1-c})$ would have sufficed. Conversely, it may be possible to improve the bound in \eqref{tv} further, perhaps all the way to $x^{-c}$.

\subsection{Fine-scale mixing of Syracuse random variables}

It remains to establish Proposition \ref{transport}.  Since the constant $\alpha$ in \eqref{alpha-def-const} is close to $1$, this proposition falls under the regime of a (refined) ``local wellposedness'' result, since from the heuristic \eqref{snn} (or \eqref{snn-2}) we expect the first passage time $T_x(\mathbf{N}_y)$ to be comparable to a small multiple of $\log \mathbf{N}_y$.  Inspecting the iteration formula \eqref{s-iter}, the behaviour of the $n$-Syracuse valuation $\vec a^{(n)}(\mathbf{N}_y)$ for such times $n$ is then well understood thanks to Proposition \ref{rach}; the main remaining difficulty is to understand the behaviour of the $n$-Syracuse offset map $F_n\colon (\N+1)^n \to \Z[\frac{1}{2}]$, and more specifically to analyse the distribution of the random variable $F_n(\Geom(2)^n) \mod 3^k$ for various $n,k$, where by abuse of notation we use $x \mapsto x \mod 3^k$ to denote the unique ring homomorphism from $\Z[\frac{1}{2}]$ to $\Z/3^k \Z$ (which in particular maps $\frac{1}{2}$ to the inverse $\frac{3^k+1}{2} \mod 3^k$ of $2 \mod 3^k$).  Indeed, from \eqref{s-iter} one has
\begin{equation}\label{syr-forward}
 \Syr^n(N) = F_n(\vec a^{(n)}(N)) \mod 3^k
\end{equation}
whenever $0 \leq k \leq n$ and $N \in 2\N+1$.  Thus,  if $n, \mathbf{N}, n', c_0$ obey the hypotheses of Proposition \ref{rach}, one has
$$ d_\TV( \Syr^n(\mathbf{N}) \mod 3^k, F_n( \Geom(2)^n ) \mod 3^k ) \ll 2^{-c_1 n}$$
for all $0 \leq k \leq n$.  If we now define the \emph{Syracuse random variables} $\Syrac(\Z/3^n\Z)$ for $n \in \N$ to be random variables on the cyclic group $\Z/3^n\Z$ with the distribution
\begin{equation}\label{syrac-def}
 \Syrac(\Z/3^n\Z) \equiv F_n( \Geom(2)^n ) \mod 3^n
\end{equation}
then from \eqref{fn-def} we see that
\begin{equation}\label{xlk}
 \Syrac(\Z/3^n\Z) \mod 3^k \equiv \Syrac(\Z/3^k\Z)
\end{equation}
whenever $k \leq n$, and thus
$$ d_\TV( \Syr^n(\mathbf{N}) \mod 3^k, \Syrac(\Z/3^k\Z) ) \ll 2^{-c_1 n}.$$
We thus see that the $3$-adic distribution of the Syracuse orbit $\Syr^\N(\mathbf{N})$ is controlled (initially, at least) by the random variables $\Syrac(\Z/3^n\Z)$. The distribution of these random variables can be computed explicitly for any given $n$ via the following recursive formula:

\begin{lemma}[Recursive formula for Syracuse random variables]\label{recursive} For any $n \in \N$ and $x \in \Z/3^{n+1}\Z$, one has
$$ \P( \Syrac(\Z/3^{n+1}\Z) = x ) = \frac{\sum_{1 \leq a \leq 2 \times 3^n: 2^a x = 1 \mod 3} 2^{-a} \P\left( \Syrac(\Z/3^{n}\Z)= \frac{2^a x-1}{3} \right)}{1 - 2^{-2 \times 3^n}},$$
where $\frac{2^a x-1}{3}$ is viewed as an element of $\Z/3^n\Z$.
\end{lemma}

\begin{proof}  Let $(\a_1,\dots,\a_{n+1}) \equiv \Geom(2)^{n+1}$ be $n+1$ iid copies of $\Geom(2)$.  From \eqref{fn-def} (after relabeling the variables $(\a_1,\dots,\a_{n+1})$ in reverse order $(\a_{n+1},\dots,\a_1)$) we have
\begin{equation}\label{fn-recurse}
 F_{n+1}(\a_{n+1},\dots,\a_{1}) = \frac{3 F_n(\a_{n+1},\dots,\a_2)+1}{2^{\a_{1}}}
\end{equation}
and thus we have
$$ \Syrac(\Z/3^{n+1}\Z) \equiv \frac{3\Syrac(\Z/3^n\Z)+1}{2^{\Geom(2)}},$$
where $3\Syrac(\Z/3^n\Z)$ is viewed as an element of $\Z/3^{n+1}\Z$, and the random variables $\Syrac(\Z/3^n\Z), \Geom(2)$ on the right hand side are understood to be independent.  We therefore have
\begin{align*}
 \P( \Syrac(\Z/3^{n+1}\Z) = x ) &= \sum_{a \in \N+1} 2^{-a} \P\left( \frac{3\Syrac(\Z/3^n\Z)+1}{2^a} = x \right) \\
&= \sum_{a \in \N+1: 2^a x = 1 \mod 3} 2^{-a} \P\left( \Syrac(\Z/3^n\Z) = \frac{2^a x-1}{3} \right).
\end{align*}
By Euler's theorem, the quantity $\frac{2^a x-1}{3} \in \Z/3^n \Z$ is periodic in $a$ with period $2 \times 3^n$.  Splitting $a$ into residue classes modulo $2 \times 3^n$ and using the geometric series formula, we obtain the claim.
\end{proof}

Thus for instance, we trivially have $\Syrac(\Z/3^0\Z)$ takes the value $0 \mod 1$ with probability $1$; then by the above lemma, $\Syrac(\Z/3\Z)$ takes the values $0,1,2 \mod 3$ with probabilities $0, 1/3, 2/3$ respectively; another application of the above lemma then reveals that $\Syrac(\Z/3^2\Z)$ takes the values $0,1,\dots,8 \mod 9$ with probabilities
$$ 0, \frac{8}{63}, \frac{16}{63}, 0, \frac{11}{63}, \frac{4}{63}, 0, \frac{2}{63}, \frac{22}{63} $$
respectively; and so forth. More generally, one can numerically compute the distribution of $\Syrac(\Z/3^n\Z)$ exactly for small values of $n$, although the time and space required to do so increases exponentially with $n$.

\begin{remark}  One could view the Syracuse random variables $\Syrac(\Z/3^n\Z)$ as projections
\begin{equation}\label{syr3}
 \Syrac(\Z/3^n\Z) \equiv \Syrac(\Z_3) \mod 3^n
\end{equation}
of a single random variable $\Syrac(\Z_3)$ taking values in the $3$-adics $\Z_3 \coloneqq \varprojlim_n \Z/3^n\Z$ (equipped with the usual metric $d(x,y) \coloneqq 3^{-\nu_3(x-y)}$), which can for instance be defined as
\begin{align*}
 \Syrac(\Z_3) &\equiv \sum_{j=0}^\infty 3^j 2^{-\a_{[1,j+1]}}\\
&= 2^{-\a_1} + 3^1 2^{-\a_{[1,2]}} + 3^2 2^{-\a_{[1,3]}} + \dots
\end{align*}
where $\a_1,\a_2,\dots$ are iid copies of $\Geom(2)$; note that this series converges in $\Z_3$, and the equivalence of distribution \eqref{syr3}follows from \eqref{syrac-def}, \eqref{fn-def} after reversing\footnote{As an alternative to reversing the order of the tuple $(\a_1,\dots,\a_n)$, one could instead index time by the negative integers $-1,-2,-3,\dots$ rather than the positive integers $1,2,3,\dots$, viewing $\Syrac(\Z_3)$ as the outcome of an ``ancient'' Syracuse iteration that extends to arbitrarily large negative times (and whose initial condition is irrelevant).  This perspective towards the Syracuse variables is arguably more natural, and could be adopted elsewhere in the paper; however, we have chosen (mostly for aesthetic reasons) to index time by positive integers rather than negative ones, which necessitates some reversal of the labeling at some junctures.} the order of the tuple $(\a_1,\dots,\a_n)$ (cf. \eqref{fn-recurse}).  One can view the distribution of $\Syrac(\Z_3)$ as the unique stationary measure for the discrete Markov process\footnote{This Markov process may possibly be related to the $3$-adic Markov process for the \emph{inverse} Collatz map studied in \cite{wirsch}. See also a recent investigation of $3$-adic irregularities of the Collatz iteration in \cite{thomas}.} on $\Z_3$ that maps each $x \in \Z_3$ to $\frac{3x+1}{2^{a}}$ for each $a \in \N+1$ with transition probability $2^{-a}$ (this fact is implicit in the proof of Lemma \ref{recursive}).  However, we will not explicitly adopt the $3$-adic perspective in this paper, preferring to work instead with the finite projections $\Syrac(\Z/3^n\Z)$ of $\Syrac(\Z_3)$.
\end{remark}

While the Syracuse random variables $\Syrac(\Z/3^n\Z)$ fail to be uniformly distributed on $\Z/3^n \Z$, we can show that they do approach uniform distribution $n \to \infty$ at fine scales (as measured in a $3$-adic sense), and this turns out to be the key ingredient needed to establish
Proposition \ref{transport}.  More precisely, we will show

\begin{proposition}[Fine scale mixing of $n$-Syracuse offsets]\label{tv-bound}  For all $1 \leq m \leq n$ one has
\begin{equation}\label{yal}
\Osc_{m,n} \left( \P( \Syrac(\Z/3^n\Z) = Y \mod 3^n ) \right)_{Y \in \Z/3^n\Z} \ll_A m^{-A}
\end{equation}
for any fixed $A>0$, where the oscillation $\Osc_{m,n}( c_Y )_{Y \in \Z/3^n\Z}$ of a tuple of real numbers $c_Y \in \R$ indexed by $\Z/3^n\Z$ at $3$-adic scale $3^{-m}$ is defined by
\begin{equation}\label{osc-def}
 \Osc_{m,n}( c_Y )_{Y \in \Z/3^n\Z} \coloneqq \sum_{Y \in \Z/3^n\Z} \left| c_Y - 3^{m-n} \sum_{Y' \in \Z/3^n\Z: Y' = Y \mod 3^m} c_{Y'} \right|.
\end{equation}
\end{proposition}

Informally, the above proposition asserts that the Syracuse random variable $\Syrac(\Z/3^n\Z)$ is approximately uniformly distributed in ``fine-scale'' or ``high-frequency'' cosets $Y + 3^m\Z/3^n\Z$, after conditioning to the event $\Syrac(\Z/3^n\Z) = Y \mod 3^m$.  Indeed, one could write the left-hand side of \eqref{yal} if desired as
$$ d_{\TV}( \Syrac(\Z/3^n\Z), \Syrac(\Z/3^n\Z) + \Unif( 3^m\Z/3^n\Z) )$$
where the random variables $\Syrac(\Z/3^n\Z), \Unif( 3^m\Z/3^n\Z)$ are understood to be independent.  In Section \ref{main-sec}, we show how Proposition \ref{transport} (and hence Theorem \ref{main}) follows from Proposition \ref{tv-bound} and Proposition \ref{rach}.

\begin{remark}  One can heuristically justify this mixing property as follows.  The geometric random variable $\Geom(2)$ can be computed to have a Shannon entropy of $\log 4$; thus, by asymptotic equipartition, the random variable $\Geom(2)^n$ is expected to behave like a uniform distribution on $4^{n+o(n)}$ separate tuples in $(\N+1)^n$.  On the other hand, the range $\Z/3^n\Z$ of the map $\vec a \mapsto F_n(\vec a) \mod 3^n$ only has cardinality $3^n$.  While this map does have substantial irregularities at coarse $3$-adic scales (for instance, it always avoids the multiples of $3$), it is not expected to exhibit any such irregularity at fine scales, and so if one models this map by a random map from $4^{n+o(n)}$ elements to $\Z/3^n\Z$ one is led to the estimate \eqref{yal} (in fact this argument predicts a stronger bound of $\exp( - cm )$ for some $c>0$, which we do not attempt to establish here).
\end{remark}

\begin{remark}  In order to upgrade logarithmic density to natural density in our results, it seems necessary to strengthen Proposition \ref{tv-bound} by establishing a suitable fine scale mixing property of the entire random affine map $\Aff_{\Geom(2)^n}$, as opposed to just the offset $F_n(\Geom(2)^n)$.  This looks plausibly attainable from the methods in this paper, but we do not pursue this question here.
\end{remark}

To prove Proposition \ref{tv-bound}, we use a partial convolution structure present in the $n$-Syracuse offset map, together with Plancherel's theorem, to reduce matters to establishing a superpolynomial decay bound for the characteristic function (or Fourier coefficients) of a Syracuse random variable $\Syrac(\Z/3^n\Z)$.  More precisely, in Section \ref{decay-sec} we derive Proposition \ref{tv-bound} from

\begin{proposition}[Decay of characteristic function]\label{f-decay} Let $n \geq 1$, and let $\xi \in \Z/3^n\Z$ be not divisible by $3$.  Then
\begin{equation}\label{lla}
 \E e^{-2\pi i \xi \Syrac(\Z/3^n\Z)/3^n} \ll_A n^{-A}
\end{equation}
for any fixed $A>0$.
\end{proposition}

A key point here is that the implied constant in \eqref{lla} is uniform in the parameters $n \geq 1$ and $\xi \in \Z/3^n\Z$ (assuming of course that $\xi$ is not divisible by $3$), though as indicated we permit this constant to depend on $A$.

\begin{remark} In the converse direction, it is not difficult to use the triangle inequality to establish the inequality
$$ |\E e^{-2\pi i \xi \Syrac(\Z/3^n\Z)/3^n}| \leq \Osc_{n-1,n} \left( \P( \Syrac(\Z/3^n\Z) = Y \mod 3^n ) \right)_{Y \in \Z/3^n\Z}$$
whenever $\xi$ is not a multiple of $3$ (so in particular the function $x \mapsto e^{-2\pi i \xi x/3^n}$ has mean zero on cosets of $3^{n-1}\Z/3^n\Z$).  Thus Proposition \ref{f-decay} and Proposition \ref{tv-bound} are in fact equivalent.  One could also equivalently phrase Proposition \ref{f-decay} in terms of the decay properties of the characteristic function of $\Syrac(\Z_3)$ (which would be defined on the Pontryagin dual $\hat \Z_3 = \Q_3/\Z_3$ of $\Z_3$), but we will not do so here.
\end{remark}

The remaining task is to establish Proposition \ref{f-decay}. This turns out to be the most difficult step in the argument, and is carried out in Section \ref{fourier-sec}.  From \eqref{fn-def}, \eqref{syrac-def} and reversing the order of the random variables $\a_1,\dots,\a_n$ (cf. \eqref{fn-recurse}), we can describe the distribution of the Syracuse random variable by the formula
\begin{equation}\label{syria}
\Syrac(\Z/3^n\Z) \equiv 2^{-\a_1} + 3^1 2^{-\a_{[1,2]}} + \dots + 3^{n-1} 2^{-\a_{[1,n]}} \mod 3^n,
\end{equation}
with $(\a_1,\dots,\a_n) \equiv \Geom(2)^n$; this also follows from \eqref{syr3}.  If this random variable $\Syrac(\Z/3^n\Z)$ was the sum of independent random variables, then the characteristic function of $\Syrac(\Z/3^n\Z)$ would factor as something like a Riesz product of cosines, and its estimation would be straightforward.  Unfortunately, the expression \eqref{syria} does not obviously resolve into such a sum of independent random variables; however, by grouping adjacent terms $3^{2j-2} 2^{-\a_{[1,2j-1]}},  3^{2j-1} 2^{-\a_{[1,2j]}}$ in \eqref{syria} into pairs, one can at least obtain a decomposition into the sum of independent expressions once one conditions on the sums $\b_j \coloneqq \a_{2j-1}+\a_{2j}$ (which are iid copies of a Pascal distribution $\Pascal$).  This lets one express the characteristic functions as an \emph{average} of products of cosines (times a phase), where the average is over trajectories of a certain random walk $\v_1, \v_{[1,2]}, \v_{[1,3]},\dots$ in $\Z^2$ with increments in the first quadrant that we call a \emph{two-dimensional renewal process}.  If we color certain elements of $\Z^2$ ``white'' when the associated cosines are small, and ``black'' otherwise, then the problem boils down to ensuring that this renewal process encounters a reasonably large number of white points (see Figure \ref{fig:triangles} in Section \ref{fourier-sec}).

From some elementary number theory, we will be able to describe the black regions of $\Z^2$ as a union of ``triangles'' $\Delta$ that are well separated from each other; again, see Figure \ref{fig:triangles}.	As a consequence, whenever the renewal process passes through a black triangle, it will very likely also pass through at least one white point after it exits the triangle.  This argument is adequate so long as the triangles are not too large in size; however, for very large triangles it does not produce a sufficient number of white points along the renewal process.  However, it turns out that large triangles tend to be fairly well separated from each other (at least in the neighbourhood of even larger triangles), and this geometric observation allows one to close the argument.

As with Proposition \ref{tv-bound}, it is possible that the bound in Proposition \ref{f-decay} could be improved, perhaps to as far as $O(\exp(-cn))$ for some $c>0$.  However, we will not need or pursue such a bound here.

The author is supported by NSF grant DMS-1764034 and by a Simons Investigator Award, and thanks Marek Biskup for useful discussions, and Ben Green, Matthias Hippold, Alex Kontorovich, Lech Mazur, Alexandre Patriota, Sankeerth Rao, Mary Rees, Lior Silberman, and several anonymous commenters on his blog for corrections and other comments.  We are especially indebted to the anonymous referee for a careful reading and many useful suggestions.

\section{Notation and preliminaries}\label{notation-sec}

We use the asymptotic notation $X \ll Y$, $Y \gg X$, or $X = O(Y)$ to denote the bound $|X| \leq CY$ for an absolute constant $C$. We also write $X \asymp Y$ for $X \ll Y \ll X$.  We also use $c>0$ to denote various small constants that are allowed to vary from line to line, or even within the same line. If we need the implied constants to depend on other parameters, we will indicate this by subscripts unless explicitly stated otherwise, thus for instance $X \ll_A Y$ denotes the estimate $|X| \leq C_A Y$ for some $C_A$ depending on $A$.

If $E$ is a set, we use $1_E$ to denote its indicator, thus $1_E(n)$ equals $1$ when $n \in E$ and $0$ otherwise.  Similarly, if $S$ is a statement, we define the indicator $1_S$ to equal $1$ when $S$ is true and $0$ otherwise, thus for instance $1_E(n) = 1_{n \in E}$.  If $E,F$ are two events, we use $E \wedge F$ to denote their conjunction (the event that both $E,F$ hold) and $\overline{E}$ to denote the complement of $E$ (the event that $E$ does not hold).

The following alternate description of the $n$-Syracuse valuation $\vec a^{(n)}(N)$ (variants of which have frequently occurred in the literature on the Collatz conjecture, see e.g., \cite{sinai}) will be useful.

\begin{lemma}[Description of $n$-Syracuse valuation]\label{iter-lem}  Let $N \in 2\N+1$ and $n \in \N$.  Then $\vec a^{(n)}(N)$ is the unique tuple $\vec a$ in $(\N+1)^n$ for which $\Aff_{\vec a}(N) \in 2\N+1$.
\end{lemma}

\begin{proof}  It is clear from \eqref{s-iter} that $\Aff_{\vec a^{(n)}(N)} \in 2\N+1$.  It remains to prove uniqueness.  The claim is easy for $n=0$, so suppose inductively that $n \geq 1$ and that uniqueness has already been established for $n-1$.  Suppose that we have found a tuple $\vec a \in (\N+1)^n$ for which $\Aff_{\vec a}(N)$ is an odd integer.  Then
$$ \Aff_{\vec a}(N) = \Aff_{a_n}( \Aff_{a_1,\dots,a_{n-1}}(N) ) = \frac{3\Aff_{a_1,\dots,a_{n-1}}(N)+1}{2^{a_n}}$$
and thus
\begin{equation}\label{y1}
 2^{a_n} \Aff_{\vec a}(N) =3\Aff_{a_1,\dots,a_{n-1}}(N)+1.
\end{equation}
This implies that $3\Aff_{a_1,\dots,a_{n-1}}(N)$ is an odd natural number.  But from \eqref{a-form}, $\Aff_{a_1,\dots,a_{n-1}}(N)$ also lies in $\Z[\frac{1}{2}]$.  The only way these claims can both be true is if $\Aff_{a_1,\dots,a_{n-1}}(N)$ is also an odd natural number, and then by induction $(a_1,\dots,a_{n-1}) = \vec a^{(n-1)}(N)$, which by \eqref{s-iter} implies that
$$ \Aff_{a_1,\dots,a_{n-1}}(N) = \Syr^{n-1}(N).$$
Inserting this into \eqref{y1} and using the fact that $\Aff_{\vec a}(N)$ is odd, we obtain
$$ a_n = \nu_2( 3\Syr^{N-1}(N) + 1 )$$
and hence by \eqref{van} we have $\vec a = \vec a^{(n)}$ as required.
\end{proof}

We record the following concentration of measure bound of Chernoff type, which also bears some resemblance to a local limit theorem.  We introduce the gaussian-type weights
\begin{equation}\label{gaussian-def}
 G_n(x) \coloneqq \exp( - |x|^2/n ) + \exp( - |x| )
\end{equation}
for any $n \geq 0$ and $x \in \R^d$ for some $d \geq 1$, where we adopt the convention that $\exp(-\infty)=0$ (so that $G_0(x) = \exp(-|x|)$).  Thus $G_n(x)$ is comparable to $1$ for $x = O(n^{1/2})$, decays in a gaussian fashion in the regime $n^{1/2} \leq |x| \leq n$, and decays exponentially for $|x| \geq n$.

\begin{lemma}[Chernoff type bound]\label{chern}  Let $d \in \N+1$, and let $\v$ be a random variable taking values in $\Z^d$ obeying the exponential tail condition
\begin{equation}\label{pvl}
 \P( |\v| \geq \lambda ) \ll \exp( -c_0 \lambda )
\end{equation}
for all $\lambda \geq 0$ and some $c_0>0$.  Assume the non-degeneracy condition that $\v$ is not almost surely concentrated on any coset of any proper subgroup of $\Z^d$.  Let $\vec \mu \coloneqq \E \v \in \R^d$ denote the mean of $\v$.   In this lemma all implied constants, as well as the constant $c$, can depend on $d$, $c_0$, and the distribution of $\v$.  Let $n \in \N$, and let $\v_1,\dots,\v_n$ be $n$ iid copies of $\v$.  Following \eqref{sum}, we write $\v_{[1,n]} \coloneqq \v_1 + \dots + \v_n$.
\begin{itemize}
\item[(i)]  For any $\vec L \in \Z^d$, one has
$$ \P\left( \v_{[1,n]} = \vec L \right) \ll \frac{1}{(n+1)^{d/2}}G_n\left( c \left(\vec L - n \vec \mu\right) \right).$$
\item[(ii)]  For any $\lambda \geq 0$, one has
$$ \P\left( |\v_{[1,n]}- n \vec \mu| \geq \lambda \right) \ll G_n( c \lambda ).$$
\end{itemize}
\end{lemma}

Thus, for instance for any $n \in \N$, we have
$$ \P\left( |\Geom(2)^n| = L \right) \ll \frac{1}{\sqrt{n+1}} G_n( c(L-2n))$$
for every $L \in \Z$, and
$$ \P\left( \left||\Geom(2)^n| - 2n\right| \geq \lambda \right) \ll G_n(c \lambda).$$
for any $\lambda \geq 0$.

\begin{proof}  We use the Fourier-analytic (and complex-analytic) method.  We may assume that $n$ is positive, since the claim is trivial for $n=0$.
We begin with (i).  Let $S$ denote the complex strip $S \coloneqq \{ z \in \C: |\mathrm{Re}(z)| < c_0 \}$, then we can define the (complexified) moment generating function $M \colon S^d \to \C$ by the formula
$$ M(z_1,\dots,z_d) \coloneqq \E \exp( (z_1,\dots,z_d) \cdot \v ),$$
where $\cdot$ is the usual bilinear dot product.
From \eqref{pvl} and Morera's theorem one verifies that this is a well-defined holomorphic function of $d$ complex variables on $S^d$, which is periodic with respect to the lattice $(2\pi i\Z)^d$.  By Fourier inversion, we have
$$ \P( \v_{[1,n]} = \vec L) = \frac{1}{(2\pi)^d} \int_{[-\pi,\pi]^d} M\left( i\vec t \right)^n \exp\left( - i \vec t \cdot \vec L \right)\ d\vec t.$$
By contour shifting, we then have
$$ \P( \v_{[1,n]} = \vec L) = \frac{1}{(2\pi)^d} \int_{[-\pi,\pi]^d} M\left( i\vec t + \vec \lambda\right)^n \exp\left( - (i\vec t + \lambda) \cdot \vec L \right)\ d\vec t$$
 whenever $\vec \lambda = (\lambda_1,\dots,\lambda_d) \in (-c_0,c_0)^d$. By the triangle inequality, we thus have
$$ \P( \v_{[1,n]} = \vec L) \ll \int_{[-\pi,\pi]^d} \left|M\left( i\vec t + \vec \lambda\right)\right|^n \exp\left( - \vec \lambda \cdot \vec L \right)\ d\vec t.$$

From Taylor expansion and the non-degeneracy condition we have
$$ M(\vec z) = \exp\left( \vec z \cdot \vec \mu + \frac{1}{2} \Sigma(\vec z) + O(|\vec z|^3) \right)$$
for all $\vec z \in S^d$ sufficiently close to $0$, where $\Sigma$ is a positive definite quadratic form (the covariance matrix of $\v$).  From the non-degeneracy condition we also see that $|M(i\vec t)| < 1$ whenever $\vec t \in [-\pi,\pi]^d$ is not identically zero, hence by continuity $|M(i\vec t + \vec \lambda)| \leq 1-c$ whenever $\vec t \in [-\pi,\pi]^d$ is bounded away from zero and $\vec \lambda$ is sufficiently small.  This implies the estimates
$$ |M(i\vec t + \vec \lambda )| \leq \exp\left( \vec \lambda \cdot \vec \mu - c |\vec t|^2 + O( |\vec \lambda|^2) \right)$$
for all $\vec t \in [-\pi,\pi]^d$ and all sufficiently small $\vec \lambda \in \R^d$.  Thus we have
\begin{align*}
 \P( \v_{[1,n]} = \vec L) &\ll \int_{[-\pi,\pi]^d} \exp\left( - \vec \lambda \cdot (\vec L-n\vec \mu) - c n|\vec t|^2 + O( n |\vec \lambda|^2 ) \right)\ d\vec t \\
&\ll n^{-1/2} \exp\left( - \vec \lambda \cdot (\vec L-n\vec \mu) + O( n |\vec \lambda|^2 ) \right).
\end{align*}
If $|\vec L-n\vec \mu| \leq n$, we can set $\vec \lambda \coloneqq c(\vec L-n\vec \mu) / n$ for a sufficiently small $c$ and obtain the claim; otherwise if $|\vec L-n\vec \mu| > n$ we set $\vec \lambda \coloneqq c(\vec L-n\vec \mu)/|\vec L-n\vec \mu|$ for a sufficiently small $c$ and again obtain the claim.  This gives (i), and the claim (ii) then follows from summing in $\vec L$ and applying the integral test.
\end{proof}

\begin{remark} Informally, the above lemma asserts that as a crude first approximation we have
\begin{equation}\label{ancl-0}
 \v_{[1,n]} \approx n\vec \mu + \Unif( \{ k \in \Z^d: k = O(\sqrt{n}) \} ),
\end{equation}
and in particular
\begin{equation}\label{ancl}
 |\Geom(2)^n| \approx \Unif( \Z \cap [2n - O(\sqrt{n}), 2n + O(\sqrt{n})] ),
\end{equation}
thus refining \eqref{lln}.   The reader may wish to use this heuristic for subsequent arguments (for instance, in heuristically justifying \eqref{snn-2}).
\end{remark}

\section{Reduction to stabilisation of first passage}\label{main-reduce}

In this section we show how Theorem \ref{main-syr} follows from Proposition \ref{transport}.  In fact we show that Proposition \ref{transport} implies a stronger claim\footnote{We thank the anonymous referee for suggesting this formulation of the main theorem.} :

\begin{theorem}[Alternate form of main theorem]\label{main-alt}  For $N_0 \geq 2$ and $x \geq 2$, one has
$$ \frac{1}{\log x}  \sum_{N \in 2\N+1 \cap [1,x]: \Syr_{\min}(N) > N_0} \frac{1}{N} \ll \frac{1}{\log^c N_0} $$
or equivalently
$$ \P( \Syr_{\min}( \Log(2\N+1 \cap [1,x]) ) \leq N_0 ) \geq 1 - O\left( \frac{1}{\log^c N_0} \right).$$
In particular, by \eqref{c-ident}, we have
$$ \P( \Col_{\min}( \Log(\N+1 \cap [1,x]) ) \leq N_0 ) \geq 1 - O\left( \frac{1}{\log^c N_0} \right)$$
for all $x \geq 2$.
\end{theorem}

In other words, for $N_0 \geq 2$, one has $\mathrm{Syr}_{\min}(N) \leq N_0$ for all $N$ in a set of odd natural numbers of (lower) logarithmic density $\frac{1}{2} - O( \log^{-c} N_0)$, and one also has $\mathrm{Col}_{\min}(N) \leq N_0$ for all $N$ in a set of positive natural numbers of (lower) logarithmic density $1 - O( \log^{-c} N_0)$.

\begin{proof}  We may assume that $N_0$ is larger than any given absolute constant, since the claim is trivial for bounded $N_0$.  Let $E_{N_0} \subset 2\N+1$ denote the set
$$ E_{N_0} \coloneqq \{ N \in 2\N+1: \Syr_{\min}(N) \leq N_0 \}$$
of starting positions $N$ of Syracuse orbits that reach $N_0$ or below.   Let $\alpha$ be defined by \eqref{alpha-def-const}, let $x \geq 2$, and let $\mathbf{N}_y$ be the random variables from Proposition \ref{transport}.  Let $B_x = B_{x,N_0}$ denote the event that $T_x(\mathbf{N}_{x^\alpha}) < +\infty$ and $\Pass_{x}(\mathbf{N}_{x^{\alpha}}) \in E_{N_0}$.  Informally, this is the event that the Syracuse orbit of $\mathbf{N}_{x^\alpha}$ reaches $x$ or below, and then reaches $N_0$ or below.  (For $x < N_0$, the latter condition is automatic, while for $x \geq N_0$, it is the former condition which is redundant.)

Observe that if $T_x( \mathbf{N}_{x^{\alpha^2}} ) < +\infty$ and $\Pass_x(\mathbf{N}_{x^{\alpha^2}}) \in E_{N_0}$, then
$$ T_{x^\alpha}( \mathbf{N}_{x^{\alpha^2}} ) \leq T_x( \mathbf{N}_{x^{\alpha^2}} ) < +\infty$$
and
$$ \Syr^\N(\Pass_x(\mathbf{N}_{x^{\alpha^2}})) \subset \Syr^\N(\Pass_{x^\alpha}(\mathbf{N}_{x^{\alpha^2}})) $$
which implies that
$$ \Syr_{\min}( \Pass_{x^\alpha}(\mathbf{N}_{x^{\alpha^2}}) ) \leq \Syr_{\min}( \Pass_x(\mathbf{N}_{x^{\alpha^2}})  ) \leq N_0.$$
In particular, the event $B_{x^\alpha}$ holds in this case.  From this, \eqref{fail}, and \eqref{tv}, \eqref{tv-2} we have
\begin{align*}
 \P( B_{x^\alpha}) &\geq \P( \Pass_{x}(\mathbf{N}_{x^{\alpha^2}}) \in E_{N_0} \wedge T_{x}( \mathbf{N}_{x^{\alpha^2}} ) < +\infty ) \\
&\geq \P( \Pass_{x}(\mathbf{N}_{x^{\alpha^2}}) \in E_{N_0} ) - O( x^{-c} ) \\
&\geq \P( \Pass_{x}(\mathbf{N}_{x^{\alpha}}) \in E_{N_0} ) - O( \log^{-c} x ) \\
&\geq \P( B_x ) - O( \log^{-c} x )
\end{align*}
whenever $x$ is larger than a suitable absolute constant (note that the $O(x^{-c})$ error can be absorbed into the $O(\log^{-c} x)$ term). In fact the bound holds for all $x \geq 2$, since the estimate is trivial for bounded values of $x$.

Let $J = J(x,N_0)$ be the first natural number such that the quantity $y \coloneqq x^{\alpha^{-J}}$ is less than $N_0^{1/\alpha}$.  Since $N_0$ is assumed to be large, we then have (by replacing $x$ with $y^{\alpha^{j-2}}$ in the preceding estimate) that
$$
 \P( B_{y^{\alpha^{j-1}}} ) \geq
 \P( B_{y^{\alpha^{j-2}}} )
- O( (\alpha^j \log y)^{-c} )
$$
for all $j=1,\dots,J$.  The event $B_{y^{\alpha^{-1}}}$ occurs with probability $1 - O(y^{-c})$, thanks to \eqref{fail} and the fact that $\mathbf{N}_{y} \leq y^\alpha \leq N_0$.  Summing the telescoping series, we conclude that
$$
 \P( B_{y^{\alpha^{J-1}}} ) \geq 1 - O( \log^{-c} y )$$
(note that the $O(y^{-c})$ error can be absorbed into the $O( \log^{-c} y )$ term).    By construction, $y \geq N_0^{1/\alpha^2}$ and $y^{\alpha^J} = x$, so
$$
 \P( B_{x^{1/\alpha}} ) \geq 1 - O( \log^{-c} N_0 ).$$
If $B_{x^{1/\alpha}}$ holds, then $\Pass_{x^{1/\alpha}}( \mathbf{N}_x )$ lies in the Syracuse orbit $\Syr^\N(\mathbf{N}_x)$, and thus $\Syr_{\min}(\mathbf{N}_x) \leq \Syr_{\min}(\Pass_{x^{1/\alpha}}( \mathbf{N}_x )) \leq N_0$.  We conclude that for any $x \geq 2$, one has
$$ \P( \Syr_{\min}(\mathbf{N}_x) > N_0 ) \ll \log^{-c} N_0.$$
By definition of $\mathbf{N}_x$ (and using the integral test to sum the harmonic series $\sum_{N \in 2\N+1 \cap [x,x^\alpha]} \frac{1}{N}$), we conclude that
\begin{equation}\label{loc}
 \sum_{N \in 2\N+1 \cap [x,x^\alpha]: \Syr_{\min}(N) > N_0} \frac{1}{N} \ll \frac{1}{\log^c N_0} \log x
\end{equation}
for all $x \geq 2$.  Covering the interval $2\N+1 \cap [1,x]$ by intervals of the form $2\N+1 \cap [y,y^\alpha]$ for various $y$, we obtain the claim.
\end{proof}

Now let $f\colon 2\N+1 \to [0,+\infty)$ be such that $\lim_{N \to \infty} f(N) = +\infty$.  Set $\tilde f(x) \coloneqq \inf_{N \in 2\N+1: N \geq x} f(N)$, then $\tilde f(x) \to \infty$ as $x \to \infty$.  Applying Theorem \ref{main-alt} with $N_0 \coloneqq \tilde f(x)$, we conclude that
$$ \sum_{N \in 2\N+1 \cap [1,x]: \Syr_{\min}(N) > f(N)} \frac{1}{N} \ll \frac{1}{\log^c \tilde f(x)} \log x$$
for all sufficiently large $x$.  Since $\frac{1}{\log^c\tilde f(x)}$ goes to zero as $x \to \infty$, we conclude from telescoping series that the set $\{ N \in 2\N+1: \Syr_{\min}(N) > f(N) \}$ has zero logarithmic density, and Theorem \ref{main-syr} follows.

\section{$3$-adic distribution of iterates}\label{rach-sec}

In this section we establish Proposition \ref{rach}.  Let $n, \mathbf{N}, c_0, n'$ be as in that proposition; in particular, $n' \geq (2+c_0) n$. In this section we allow implied constants in the asymptotic notation, as well as the constants $c>0$, to depend on $c_0$.

We first need a tail bound on the size of the $n$-Syracuse valuation $\vec a^{(n)}(\mathbf{N})$:

\begin{lemma}[Tail bound]\label{tail}  We have
$$ \P( |\vec a^{(n)}(\mathbf{N})| \geq n' ) \ll 2^{-cn}.$$
\end{lemma}

\begin{proof}  Write $\vec a^{(n)}(\mathbf{N}) = (\a_1,\dots,\a_n)$, then we may split
$$ \P( |\vec a^{(n)}(\mathbf{N})| \geq n' ) = \sum_{k=0}^{n-1}  \P( \a_{[1,k]} < n' \leq \a_{[1,k+1]} )$$
(using the summation convention \eqref{sum})
and so it suffices to show that
$$ \P( \a_{[1,k]} < n' \leq \a_{[1,k+1]} ) \ll 2^{-cn}$$
for each $0 \leq k \leq n-1$.

From Lemma \ref{iter-lem} and \eqref{a-form} we see that
$$ 3^{k+1} 2^{- \a_{[1,k+1]}} \mathbf{N} + \sum_{i=1}^{k+1} 3^{k+1-i} 2^{-\a_{[i,k+1]}} $$
is an odd integer, and thus
$$ 3^{k+1} \mathbf{N} + \sum_{i=1}^{k+1} 3^{k+1-i} 2^{\a_{[1,i-1]}}$$
is a multiple of $2^{\a_{[1,k+1]}}$.  In particular, when the event $\a_{[1,k]} < n' \leq \a_{[1,k+1]}$ holds, one has
$$ 3^{k+1} \mathbf{N} + \sum_{i=1}^{k+1} 3^{k+1-i} 2^{\a_{[1,i-1]}} = 0 \mod 2^{n'}.$$
Thus, if one conditions to the event $\a_j = a_j, j=1,\dots,k$ for some positive integers $a_1,\dots,a_k$, then $\mathbf{N}$ is constrained to a single residue class $b \mod 2^{n'}$ depending on $a_1,\dots,a_k$ (because $3^{k+1}$ is invertible in the ring $\Z/2^{n'}\Z$).  From \eqref{boy}, \eqref{tv-1} we have the quite crude estimate
$$ \P( \mathbf{N} = b \mod 2^{n'} ) \ll 2^{-n'}$$
and hence
$$ \P( \a_{[1,k]} \leq n' < \a_{[1,k+1]} ) \ll \sum_{a_1,\dots,a_k \in \N+1: a_{[1,k]} < n'} 2^{-n'}.$$
The tuples $(a_1,\dots,a_k)$ in the above sum are in one-to-one correspondence with the $k$-element subsets $\{ a_1, a_{[1,2]},\dots,a_{[1,k]}\}$ of $\{1,\dots,n'-1\}$, and hence have cardinality $\binom{n'-1}{k}$, thus
$$ \P( \a_{[1,k]} < n' \leq \a_{[1,k+1]} ) \ll 2^{-n'} \binom{n'-1}{k}.$$
Since $k \leq n-1$ and $n' \geq (2+c_0) n$, the right-hand side is $O(2^{-cn})$ by Stirling's formula (one can also use the Chernoff inequality for the sum of $n'-1$ Bernoulli random variables $\mathbf{Ber}(\frac{1}{2})$, or Lemma \ref{chern}).  The claim follows.
\end{proof}

From Lemma \ref{chern} we also have
$$ \P( |\Geom(2)^n| \geq n' ) \ll 2^{-cn}.$$
From \eqref{tv-1} and the triangle inequality we therefore have
$$ d_\TV(\vec a^{(n)}(\mathbf{N}), \Geom(2)^n) =
\sum_{\vec a \in (\N+1)^n: |\vec a| < m} |\P(\vec a^{(n)}(\mathbf{N})=\vec a) - \P(\Geom(2)^n=\vec a)| + O( 2^{-cn} ).$$
From Definition \ref{geom} we have
$$ \P(\Geom(2)^n=\vec a) = 2^{-|\vec a|} $$
so it remains to show that
\begin{equation}\label{mup}
 \sum_{\vec a \in (\N+1)^n: |\vec a| < m} |\P(\vec a^{(n)}(\mathbf{N})=\vec a) - 2^{-|\vec a|}| \ll 2^{-cn}.
\end{equation}
By Lemma \ref{iter-lem}, the event $\vec a^{(n)}(\mathbf{N})=\vec a$ occurs precisely when $\Aff_{\vec a}(\mathbf{N})$ is an odd integer, which by \eqref{a-form} we may write (for $\vec a = (a_1,\dots,a_n)$) as
$$
3^n 2^{-a_{[1,n]}} \mathbf{N} + 3^{n-1} 2^{-a_{[1,n]}} + 3^{n-2} 2^{-a_{[2,n]}} + \dots + 2^{-a_n} \in 2\N+1.$$
Equivalently one has
$$ 3^n \mathbf{N} = - 3^{n-1} - 3^{n-2} 2^{a_1} - \dots - 2^{a_{[1,n-1]}} + 2^{|\vec a|} \mod 2^{|\vec a|+1}.$$
This constrains $\mathbf{N}$ to a single odd residue class modulo $2^{|\vec a|+1}$.  For $|\vec a| < n'$, the probability of falling in this class can be computed using \eqref{boy}, \eqref{tv-1} as $2^{-|\vec a|} + O( 2^{-n'} )$.  The left-hand side of \eqref{mup} is then bounded by
$$ \ll 2^{-n'} \# \{ \vec a \in (\N+1)^n: |\vec a| < n' \} = 2^{-n'} \binom{n'-1}{n}.$$
The claim now follows from Stirling's formula (or Chernoff's inequality), as in the proof of Lemma \ref{tail}.  This completes the proof of Proposition \ref{rach}.

\section{Reduction to fine scale mixing of the $n$-Syracuse offset map}\label{main-sec}

We are now ready to derive Proposition \ref{transport} (and thus Theorem \ref{main}) assuming Proposition \ref{tv-bound}.  Let $x$ be sufficiently large.  We take $y$ to be either $x^\alpha$ or $x^{\alpha^2}$.  From the heuristic \eqref{snn} (or \eqref{snn-2}) we expect the first passage time $\Pass_x(\mathbf{N}_y)$ to be roughly
$$ \Pass_x(\mathbf{N}_y) \approx \frac{\log \mathbf{N}_y / x}{\log(4/3)}$$
with high probability.  Now introduce the quantities
\begin{equation}\label{no}
 n_0 \coloneqq \left\lfloor \frac{\log x}{10 \log 2} \right\rfloor
\end{equation}
(so that $2^{n_0} \asymp x^{0.1}$) and
\begin{equation}\label{mon}
 m_0 \coloneqq \left\lfloor \frac{\alpha-1}{100} \log x \right\rfloor.
\end{equation}
Since the random variable $\mathbf{N}_y$ takes values in $[y,y^\alpha]$, we see from \eqref{alpha-def-const} that we would expect the bounds
\begin{equation}\label{heur}
 m_0 \leq T_x(\mathbf{N}_y)  \leq n_0
\end{equation}
to hold with high probability.  We will use these parameters $m_0, n_0$ to help control the distribution of $T_x(\mathbf{N}_y)$ and $\Pass_x(\mathbf{N}_y)$ in order to prove \eqref{fail}, \eqref{tv}.

\begin{figure} [t]
\centering
\includegraphics[width=4in]{./falling.png}
\caption{The Syracuse orbit $n \mapsto \mathrm{Syr}^n(\mathbf{N}_y)$, where the vertical axis is drawn in shifted log-scale.  The diagonal lines have slope $-\log(4/3)$.  For times $n$ up to $n_0$, the orbit usually stays close to the dashed line, and hence usually lies between the two dotted diagonal lines; in particular, the first passage time $T_x(\mathbf{N}_y)$ will usually lie in the interval $I_y$.  Outside of a rare exceptional event, for any given $n \in I_y$, $\Syr^{n-m}(\mathbf{N}_y)$ will lie in $E'$ if and only if $n = T_x(\mathbf{N}_y)$ and $\Syr^n(\mathbf{N}_y)$ lies in $E$; equivalently, outside of a rare exceptional event, $\Pass_x(\mathbf{N}_y)$ lies in $E$ if and only if $\Syr^{n-m}(\mathbf{N}_y)$ lies in $E'$ for precisely one $n \in I_y$.}
\label{fig:falling}
\end{figure}

We begin with the proof of \eqref{fail}.  Let $n_0$ be defined by \eqref{no}. Since $\mathbf{N}_y \equiv \mathbf{Log}(2\N+1 \cap [y,y^\alpha])$,
a routine application of the integral test reveals that
$$d_\TV( \mathbf{N}_y \mod 2^{3n_0}, \Unif((2\Z+1)/2^{3n_0}\Z)) \ll 2^{-3n_0} $$
(with plenty of room to spare), hence by Proposition \ref{rach}
\begin{equation}\label{cho}
 d_\TV( \vec a^{(n_0)}(\mathbf{N}_y), \Geom(2)^{n_0} ) \ll 2^{-c n_0}.
\end{equation}
In particular, by \eqref{tv-2} and Lemma \ref{chern} we have
\begin{equation}\label{huh}
 \P( |\vec a^{(n_0)}(\mathbf{N}_y)| \leq 1.9 n_0 ) \leq \P( |\Geom(2)^{n_0}| \leq 1.9 n_0 ) + O(2^{-cn_0}) \ll 2^{-cn_0} \ll x^{-c}
\end{equation}
(recall we allow $c$ to vary even within the same line).
On the other hand, from \eqref{s-iter}, \eqref{fn-def} we have
$$ \Syr^{n_0}(\mathbf{N}_y) \leq 3^{n_0} 2^{-|\vec a^{(n_0)}(\mathbf{N}_y)|} \mathbf{N}_y + O( 3^{n_0} ) \leq 3^{n_0} 2^{-|\vec a^{(n_0)}(\mathbf{N}_y)|} x^{\alpha^3} + O(3^{n_0}) $$
and hence if $|\vec a^{(n_0)}(\mathbf{N}_y)| > 1.9 n$ then
$$\Syr^{n_0}(\mathbf{N}_y)  \ll 3^{n_0} 2^{-1.9 n_0} x^{\alpha^3} + O(3^{n_0}).$$
From \eqref{no}, \eqref{alpha-def-const} and a brief calculation, the right-hand side is $O(x^{0.99})$ (say).  In particular, for $x$ large enough, we have
$$ \Syr^{n_0}(\mathbf{N}_y)  \leq x,$$
and hence $T_x(\mathbf{N}_y) \leq n_0 < +\infty$ whenever $|\vec a^{(n_0)}(\mathbf{N}_y)| > 1.9 n_0$ (cf., the upper bound in \eqref{heur}).  The claim \eqref{fail} now follows from \eqref{huh}.

\begin{remark} This argument already establishes that $\Syr_{\min}(N) \leq N^\theta$ for almost all $N$ for any $\theta > 1/\alpha$; by optimising the numerical exponents in this argument one can eventually recover the results of Korec \cite{korec} mentioned in the introduction.  It also shows that most odd numbers do not lie in a periodic Syracuse orbit, or more precisely that
$$ \P( \Syr^n(\mathbf{N}_y) = \mathbf{N}_y \hbox{ for some } n \in \N+1 ) \ll x^{-c}.$$
Indeed, the above arguments show that outside of an event of probability $x^{-c}$, one has $\Syr^{\mathbf{m}}(\mathbf{N}_y) \leq x$ for some $\mathbf{m} \leq n_0$, which we can assume to be minimal amongst all such $\mathbf{m}$.  If $\Syr^n(\mathbf{N}_y) = \mathbf{N}_y$ for some $n$, we then have
\begin{equation}\label{ny}
\mathbf{N}_y = \Syr^{n(\mathbf{M})-\mathbf{m}}(\mathbf{M})
\end{equation}
for $\mathbf{M} \coloneqq \Syr^\mathbf{m}(\mathbf{N}_y) \in [1,x]$ that generates a periodic Syracuse orbit with period $n(\mathbf{M})$.  (This period $n(\mathbf{M})$ could be extremely large, and the periodic orbit could attain values much larger than $x$ or $y$, but we will not need any upper bounds on the period in our arguments, other than that it is finite.) The number of possible pairs $(\mathbf{M},\mathbf{m})$ obtained in this fashion is $O(xn_0)$.  By \eqref{ny}, the pair $(\mathbf{M},\mathbf{m})$ uniquely determines $\mathbf{N}_y$.  Thus, outside of the aforementioned event, a periodic orbit is only possible for at most $O(xn_0)$ possible values of $\mathbf{N}_y$; as this is much smaller than $y$, we thus see that a periodic orbit is only attained with probability $O(x^{-c})$, giving the claim.  It is then a routine matter to then deduce that almost all positive integers do not lie in a periodic Collatz orbit; we leave the details to the interested reader.
\end{remark}

Now we establish \eqref{tv}.  By \eqref{tv-2}, it suffices to show that for $E \subset 2\N+1 \cap [1,x]$, that
\begin{equation}\label{pye}
 \P( \Pass_x( \mathbf{N}_y ) \in E ) = \left(1 + O( \log^{-c} x )\right) Q + O( \log^{-c} x )
\end{equation}
for some quantity $Q$ that can depend on $x,\alpha,E$ but is independent of whether $y$ is equal to $x^\alpha$ or $x^{\alpha^2}$ (note that this bound automatically forces $Q = O(1)$ when $x$ is large, so the first error term $O(\log^{-c} x) Q$ on the right-hand side may be absorbed into the second term $O(\log^{-c} x)$).  The strategy is to manipulate the left-hand side of \eqref{pye} into an expression that involves the Syracuse random variables $\Syrac(\Z/3^n\Z)$ for various $n$ (in a range $I_y$ depending on $y$) plus a small error, and then appeal to Proposition \ref{tv-bound} to remove the dependence on $n$ and hence on $y$ in the main term.  The main difficulty is that the first passage location $\Pass_x(\mathbf{N}_y)$ involves a first passage time $n = T_x(\mathbf{N}_y)$ whose value is not known in advance; but by stepping back in time by a fixed number of steps $m_0$, we will be able to express the left-hand side of \eqref{pye} (up to negligible errors) without having to explicitly refer to the first passage time.

The first step is to establish the following approximate formula for the left-hand side of \eqref{pye}.

\begin{proposition}[Approximate formula]  Let $E \subset 2\N+1 \cap [1,x]$ and $y = x^\alpha, x^{\alpha^2}$.  Then we have
\begin{equation}\label{approx-form}
  \P( \Pass_x( \mathbf{N}_y ) \in E )  =  \sum_{n \in I_y} \sum_{\vec a \in {\mathcal A}^{(n-m_0)}} \sum_{M \in E'} \P( \Aff_{\vec a}(\mathbf{N}_y) = M ) + O( \log^{-c} x  )
	\end{equation}
where $I_y$ is the interval
\begin{equation}\label{iy-def}
 I_y \coloneqq \left[\frac{\log( y / x )}{\log \frac{4}{3}} + \log^{0.8} x, \frac{\log( y^\alpha / x )}{\log \frac{4}{3}} - \log^{0.8} x\right],
\end{equation}
$E'$ is the set of odd natural numbers $M \in 2\N+1$ such that $T_x(M) = m_0$ and $\Pass_x(M) \in E$ with
\begin{equation}\label{lost}
 \exp( - \log^{0.7} x ) (4/3)^{m_0} x \leq M \leq \exp( \log^{0.7} x ) (4/3)^{m_0} x.
\end{equation}
and for any natural number $n'$, ${\mathcal A}^{(n')} \subset (\N+1)^{n'}$ denotes the set of all tuples $(a_1,\dots,a_{n'}) \in (\N+1)^{n'}$ such that
\begin{equation}\label{an-bring}
|a_{[1,n]} - 2n| < \log^{0.6} x
\end{equation}
for all $0 \leq n \leq n'$.
\end{proposition}

A key point in this formula \eqref{approx-form} is that the right-hand side does not involve the passage time $T_x(\mathbf{N}_y)$ or the first passage location $\Pass_x(\mathbf{N}_y)$, and the dependence on whether $y$ is equal to $x^\alpha$ or $x^{\alpha^2}$ is confined to the range $I_y$ of the summation variable $n$, as well as the input $\mathbf{N}_y$ of the affine map $\Aff_{\vec a}$.  (In particular, note that the set $E'$ does \emph{not} depend on $y$.) We also observe from \eqref{iy-def}, \eqref{no}, \eqref{mon} that $I_y \subset [m_0,n_0]$, which is consistent with the heuristic \eqref{heur}.

\begin{proof}  Fix $E$, and write $\vec a^{(n_0)}(\mathbf{N}_y) = (\a_1,\dots,\a_{n_0})$.  From \eqref{cho}, \eqref{tv-2}, and Lemma \ref{chern} we see that for every $0 \leq n \leq n_0$, one has
$$ \P( |\a_{[1,n]} - 2n| \geq \log^{0.6} x ) \ll \exp( - c \log^{0.2} x ).$$
Hence, if ${\mathcal A}^{(n_0)}$ is the set defined in the proposition, we see from the union bound that
\begin{equation}\label{lo}
 \P( \vec a^{(n_0)}(\mathbf{N}_y)  \not \in {\mathcal A}^{(n_0)} ) \ll \log^{-10} x
\end{equation}
(say); this can be viewed as a rigorous analogue of the heuristic \eqref{ancl}.  Hence
$$  \P( \Pass_x( \mathbf{N}_y ) \in E )  =  \P( \Pass_x( \mathbf{N}_y ) \in E \wedge \vec a^{(n_0)}(\mathbf{N}_y) \in {\mathcal A}^{(n_0)} ) +   O( \log^{-c} x  ).$$
Suppose that $\vec a^{(n_0)}(\mathbf{N}_y) \in {\mathcal A}^{(n_0)}$.    For any $0 \leq n \leq n_0$, we have from \eqref{s-iter}, \eqref{fn-bound} that
$$
 \Syr^n(\mathbf{N}_y) = 3^{n} 2^{-\a_{[1,n]}} \mathbf{N}_y + O( 3^{n_0} ) $$
and hence by \eqref{an-bring}, \eqref{no} and some calculation
\begin{equation}\label{snny}
\Syr^n(\mathbf{N}_y) = (1 + O(x^{-0.1})) 3^{n} 2^{-\a_{[1,n]}} \mathbf{N}_y.
\end{equation}
In particular, from \eqref{an-bring} one has
\begin{equation}\label{anb}
\Syr^n(\mathbf{N}_y) = \exp( O( \log^{0.6} x)) (3/4)^n \mathbf{N}_y
\end{equation}
for all $0 \leq n \leq n_0$, which can be viewed as a rigorous version of the heuristic \eqref{snn-2}.  With regards to Figure \ref{fig:falling}, \eqref{anb} asserts that the Syracuse orbit stays close to the dashed line.

As $T_x(\mathbf{N}_y)$ is the first time $n$ for which $\Syr^n(\mathbf{N}_y) \leq x$, the estimate \eqref{anb} gives an approximation
\begin{equation}\label{txy}
 T_x(\mathbf{N}_y) = \frac{\log( \mathbf{N}_y / x )}{\log \frac{4}{3}} + O( \log^{0.6} x);
\end{equation}
note from \eqref{no}, \eqref{alpha-def-const} and a brief calculation that the right-hand side automatically lies between $0$ and $n_0$ if $x$ is large enough.  In particular, if $I_y$ is the interval \eqref{iy-def}, then \eqref{anb} will imply that $T_x(\mathbf{N}_y) \in I_y$ whenever
$$ \mathbf{N}_y \subset [y + 2 \log^{0.8} x, y^\alpha - 2 \log^{0.8} x];$$
a straightforward calculation using the integral test (and \eqref{lo}) then shows that
\begin{equation}\label{tiy}
 \P( T_x(\mathbf{N}_y) \in I_y) = 1 - O( \log^{-c} x ).
\end{equation}
Again, see Figure \ref{fig:falling}.  Note from \eqref{no}, \eqref{mon} that $I_y \subset [m_0,n_0]$; compare with \eqref{heur}.

Now suppose that $n$ is an element of $I_y$.  In particular, $n \geq m_0$.  We observe the following implications:
\begin{itemize}
\item If $T_x(\mathbf{N}_y) = n$, then certainly $T_x( \Syr^{n-m_0}(\mathbf{N}_y) ) = m_0$.
\item Conversely, if $T_x( \Syr^{n-m_0}(\mathbf{N}_y) ) = m_0$ and $\vec a^{(n_0)}(\mathbf{N}_y) \in {\mathcal A}^{(n_0)}$, we have $\Syr^n(\mathbf{N}_y) \leq x < \Syr^{n-1}(\mathbf{N}_y)$, which by \eqref{anb} forces
$$ n = \frac{\log( \mathbf{N}_y / x )}{\log \frac{4}{3}} + O( \log^{0.6} x),$$
which by \eqref{txy}, \eqref{mon} implies that $T_x(\mathbf{N}_y) \geq n - m_0$, and hence
$$ T_x(\mathbf{N}_y) = n - m_0 + T_x(\Syr^{n-m_0}(\mathbf{N}_y)) = n.$$
\end{itemize}
We conclude that for any $n \in I_y$, the event
$$\left( T_x(\mathbf{N}_y) = n\right) \wedge \left( \Pass_x( \mathbf{N}_y ) \in E \right) \wedge \left( \vec a^{(n_0)}(\mathbf{N}_y) \in {\mathcal A}^{(n_0)}\right) $$
holds precisely when the event
$$ B_{n,y} \coloneqq \left( T_x( \Syr^{n-m_0}(\mathbf{N}_y) ) = m_0\right) \wedge \left(\Pass_x(\Syr^{n-m_0}(\mathbf{N}_y)) \in E \right) \wedge \left( \vec a^{(n_0)}(\mathbf{N}_y) \in {\mathcal A}^{(n_0)}\right)$$
does.  From \eqref{tiy} we therefore have the estimate
$$
\P( \Pass_x( \mathbf{N}_y ) \in E )  =  \sum_{n \in I_y} \P(B_{n,y})  +   O( \log^{-c} x  ).
$$
With $E'$ the set defined in the proposition, we observe the following implications:
\begin{itemize}
\item If $B_{n,y}$ occurs, then from \eqref{anb}, \eqref{txy} we have
$$ \Syr^{n-m_0}(\mathbf{N}_y) = \exp( O( \log^{0.6} x)) (3/4)^{T_x(\mathbf{N}_y)-m_0} \mathbf{N}_y = \exp( O( \log^{0.6} x)) (4/3)^{m_0} x$$
and hence
\begin{equation}\label{sor}
 \left( \Syr^{n-m_0}(\mathbf{N}_y) \in E'\right) \wedge \left(\vec a^{(n_0)}(\mathbf{N}_y) \in {\mathcal A}^{(n_0)}\right).
\end{equation}
\item Conversely, if \eqref{sor} holds, then from \eqref{anb} we have
$$ \Syr^{n'}(\mathbf{N}_y) = \exp(O(\log^{0.6} x)) (4/3)^{n-m_0-n'} \Syr^{n-m_0}(\mathbf{N}_y) \geq \exp(O(\log^{0.6} x)) \Syr^{n-m_0}(\mathbf{N}_y) $$
for all $0 \leq n' \leq n-m_0$, and hence by \eqref{lost}
$$ \Syr^{n'}(\mathbf{N}_y) > x$$
for all $0 \leq n' \leq n-m_0$.  We conclude that
$$ T_x(\mathbf{N}_y) = n-m_0 + T_x(\Syr^{n-m_0}(\mathbf{N}_y)) = n $$
thanks to the definition of $E'$, and hence also
$$ \Pass_x(\mathbf{N}_y) = \Pass_x(\Syr^{n-m_0}(\mathbf{N}_y)) \in E.$$
In particular, the event $B_{n,y}$ holds.
\end{itemize}
We conclude that we have the equality of events
$$  B_{n,y}  = \left(\Syr^{n-m_0}(\mathbf{N}_y) \in E'\right) \wedge \left(\vec a^{(n_0)}(\mathbf{N}_y) \in {\mathcal A}^{(n_0)}\right)$$
for any $n \in I_y$.  Since the event $\vec a^{(n_0)}(\mathbf{N}_y) \in {\mathcal A}^{(n_0)}$ is contained in the event $\vec a^{(n-m_0)}(\mathbf{N}_y) \in {\mathcal A}^{(n-m_0)}$, we conclude from \eqref{lo} that
$$  \P( \Pass_x( \mathbf{N}_y ) \in E )  =  \sum_{n \in I_y} \P\left( \left(\Syr^{n-m_0}(\mathbf{N}_y) \in E'\right) \wedge \left(\vec a^{(n-m_0)}(\mathbf{N}_y) \in {\mathcal A}^{(n-m_0)}\right) \right) +   O( \log^{-c} x  ).$$

Suppose that $\vec a = (a_1,\dots,a_{n-m})$ is a tuple in ${\mathcal A}^{(n-m)}$, and $M \in E'$.  From Lemma \ref{iter-lem}, we see that the event $\left(\Syr^{n-m_0}(\mathbf{N}_y) = M\right) \wedge \left(\vec a^{(n-m_0)}(\mathbf{N}_y) = \vec a\right)$ holds if and only if $\Aff_{\vec a}( \mathbf{N}_y) \in E'$, and the claim \eqref{approx-form} follows.
\end{proof}

Now we compute the right-hand side of \eqref{approx-form}.  Let $n \in I_y$, $\vec a \in {\mathcal A}^{(n-m_0)}$, and $M \in E'$. Then by \eqref{a-form}, the event $\Aff_{\vec a}(\mathbf{N}_y) = M$ is only non-empty when
\begin{equation}\label{max}
M = F_{n-m_0}(\vec a) \mod 3^{n-m_0}
\end{equation}
Conversely, if \eqref{max} holds, then $\Aff_{\vec a}(\mathbf{N}_y) = M$ holds precisely when
\begin{equation}\label{this}
 \mathbf{N}_y = 2^{|\vec a|} \frac{M - F_{n-m_0}(\vec a)}{3^{n-m_0}}.
\end{equation}
Note from \eqref{an-bring}, \eqref{fn-bound} that the right-hand side of \eqref{this} is equal to
$$ 2^{2(n-m_0)+O(\log^{0.6} x)} \frac{M + O( 3^{n-m_0} )}{3^{n-m_0}} $$
which by \eqref{lost}, \eqref{no} simplifies to
$$ \exp( O(\log^{0.7} x)) (4/3)^{n} x.$$
Since $n \in I_y$, we conclude from \eqref{iy-def} that the right-hand side of \eqref{this} lies in $[y, y^\alpha]$; from \eqref{max}, \eqref{fn-def} we also see that this right-hand side is a odd integer.  Since $\mathbf{N}_y \equiv \Log( 2\N+1 \cap [y,y^\alpha] )$ and
$$ \sum_{N \in 2\N+1 \cap [y,y^\alpha]} \frac{1}{N} = \left(1 + O\left( \frac{1}{x}\right)\right) \frac{\alpha-1}{2} \log y,$$
we thus see that when \eqref{max} occurs, one has
$$ \P( \Aff_{\vec a}(\mathbf{N}_y) = M ) = \frac{1}{\left(1 + O( \frac{1}{x})\right)\frac{\alpha-1}{2} \log y} 2^{-|\vec a|} \frac{3^{n-m_0}}{M - F_{n-m_0}(\vec a)}.$$
From \eqref{lost}, \eqref{no}, \eqref{fn-bound} we can write
$$ M - F_{n-m_0}(\vec a) = M - O(3^{n_0}) = (1 + O(x^{-c})) M$$
and thus
$$ \P( \Aff_{\vec a}(\mathbf{N}_y) = M ) = \frac{1 + O(x^{-c})}{\frac{\alpha-1}{2} \log y} \frac{2^{-|\vec a|} 3^{n-m_0}}{M}.$$
We conclude that
\begin{align*}
  \P( \Pass_x( \mathbf{N}_y ) \in E )  &= \frac{1 + O(x^{-c})}{\frac{\alpha-1}{2} \log y} \sum_{n \in I_y} 3^{n-m_0} \sum_{\vec a \in {\mathcal A}^{(n-m_0)}} 2^{-|\vec a|} \sum_{M \in E': M = F_{n-m_0}(\vec a) \mod 3^{n-m_0}} \frac{1}{M} \\
	&\quad +   O( \log^{-c} x  ).
	\end{align*}
We will eventually establish the estimate
\begin{equation}\label{zeno}
 3^{n-m_0} \sum_{\vec a \in {\mathcal A}^{(n-m_0)}} 2^{-|\vec a|} \sum_{M \in E': M = F_{n-m_0}(\vec a) \mod 3^{n-m_0}} \frac{1}{M} = Z + O( \log^{-c} x  )
\end{equation}
for all $n \in I_y$, where $Z$ is the quantity
\begin{equation}\label{Z-def}
 Z \coloneqq \sum_{M \in E'} \frac{3^{m_0} \P( M = \Syrac(\Z/3^{m_0}\Z) \mod 3^{m_0})}{M}.
\end{equation}
Since  from \eqref{iy-def} we have
$$ \# I_y = (1 + O( \log^{-c} x  )) \frac{\alpha-1}{\log \frac{4}{3}} \log y,$$
we see that \eqref{zeno} would imply the bound
$$  \P( \Pass_x( \mathbf{N}_y ) \in E )  = (1 + O(\log^{-c} x)) \frac{2}{\log \frac{4}{3}} Z + O(\log^{-c} x)$$
which would give the desired estimate \eqref{pye} since $Z$ does not depend on whether $y$ is equal to $x^\alpha$ or $x^{\alpha^2}$.

It remains to establish \eqref{zeno}.
Fix $n \in I_y$.  The left-hand side of \eqref{zeno} may be written as
\begin{equation}\label{soda}
  \E 1_{(\a_1,\dots,\a_{n-m_0}) \in {\mathcal A}^{(n-m_0)}} c_n( F_{n-m_0}(\a_1,\dots,\a_{n-m_0}) \mod 3^{n-m_0} )
	\end{equation}
where $(\a_1,\dots,\a_{n-m_0})  \equiv \Geom(2)^{n-m_0}$ and $c_n\colon \Z/3^{n-m_0}\Z \to \R^+$ is the function
\begin{equation}\label{cn-def}
 c_n( X ) \coloneqq 3^{n-m_0} \sum_{M \in E': M = X \mod 3^{n-m_0}} \frac{1}{M}.
\end{equation}
We have a basic estimate:

\begin{lemma}\label{loam}  We have $c_n(X) \ll 1$ for all $n \in I_y$ and $X \in \Z/3^{n-m_0}\Z$.
\end{lemma}

\begin{proof}  We can split
$$ c_n(X) \leq \sum_{(a_1,\dots,a_{m_0}) \in \N^{m_0}} c_{n,a_1,\dots,a_{m_0}}(X)$$
where
$$ c_{n,a_1,\dots,a_{m_0}}(X) \coloneqq 3^{n-m_0} \sum_{M \in E': M = X \mod 3^{n-m_0}; (a_1,\dots,a_{m_0}) \coloneqq \vec a^{(m_0)}(M)} \frac{1}{M}.$$
We now estimate $c_{n,a_1,\dots,a_{m_0}}(X)$ for a given $(a_1,\dots,a_{m_0}) \in \N^{m_0}$.
If $M \in E'$, then on setting $(a_1,\dots,a_{m_0}) \coloneqq \vec a^{(m_0)}(M)$ we see from \eqref{s-iter} that
$$
3^{m_0} 2^{-a_{[1,m_0]}} M + F_{m_0}(a_1,\dots,a_{m_0}) \leq x < 3^{m_0} 2^{-a_{[1,m_0-1]}} M + F_{m_0-1}(a_1,\dots,a_{m_0-1}) $$
which by \eqref{mon} and \eqref{fn-bound}
implies that
$$
3^{m_0} 2^{-a_{[1,m_0]}} M  \leq x \ll 3^{m_0} 2^{-a_{[1,m_0-1]}} M$$
or equivalently
\begin{equation}\label{chan}
3^{-m_0} 2^{a_{[1,m_0-1]}} x \ll M \leq 3^{-m_0} 2^{a_{[1,m_0]}} x.
\end{equation}
Also, from \eqref{s-iter} we also have that
$$ 3^{m_0} M + 2^{a_{[1,m_0]}} F_{m_0}(a_1,\dots,a_{m_0}) = 2^{a_{[1,m_0]}} \mod 2^{a_{[1,m_0]}+1} $$
and so $M$ is constrained to a single residue class modulo $2^{a_{[1,m_0]}+1}$.   In \eqref{cn-def} we are also constraining $M$ to a single residue class modulo $3^{n-m_0}$; by the Chinese remainder theorem, these constraints can be combined into a single residue class modulo $2^{a_{[1,m_0]}+1} 3^{n-m_0}$.  Note from the integral test that
\begin{equation}\label{mant}
\begin{split}
 \sum_{M_0 \leq M \leq M_1: M = a \mod q} \frac{1}{M} &\leq \frac{1}{M_0} + \sum_{M_0+q \leq M \leq M_1: M = a \mod q} \frac{1}{M} \\
&\leq \frac{1}{M_0} + \frac{1}{q} \int_{M_0}^{M_1} \frac{dt}{t} \\
&= \frac{1}{M_0} + \frac{1}{q} \log \frac{M_1}{M_0}
\end{split}
\end{equation}
for any $M_0 \leq M_1$ and any residue class $a \mod q$.  In particular, for $q \leq M_0$, we have
\begin{equation}\label{integral}
 \sum_{M_0 \leq M \leq M_1: M = a \mod q} \frac{1}{M}  \ll \frac{1}{q} \log O \left( \frac{M_1}{M_0} \right).
\end{equation}

If $2^{a_{[1,m_0]}} \leq x^{0.5}$ (say), then the modulus $2^{a_{[1,m_0]}+1} 3^{n-m_0}$ is much less than the lower bound on $M$ in \eqref{chan}, and  we can then use the integral test to bound
\begin{align*}
c_{n,a_1,\dots,a_{m_0}}(X) & \ll 3^{n-m_0} (2^{a_{[1,m_0]}+1} 3^{n-m_0})^{-1} \log O \left( \frac{3^{-m_0} 2^{a_{[1,m_0]}} x}{3^{-m_0} 2^{a_{[1,m_0-1]}} x } \right)\\
&\ll 2^{-a_{[1,m_0]}} a_{m_0} \\
&\ll 2^{-a_{[1,m_0]}/2}.
\end{align*}
Now suppose instead that $2^{a_{[1,m_0]}} > x^{0.5}$, we recall from \eqref{s-iter} that
$$ a_{m_0} = \nu_2\left( 3 (3^{m_0} 2^{-a_{[1,m_0-1]}} M + F_{m_0-1}(a_1,\dots,a_{m_0-1})) + 1\right)$$
so
$$ 2^{a_{m_0}} \ll 3^{m_0} 2^{-a_{[1,m_0-1]}} M + F_{m_0-1}(a_1,\dots,a_{m_0-1}) \ll  3^{m_0} 2^{-a_{[1,m_0-1]}} M$$
(using \eqref{fn-bound}, \eqref{chan} to handle the lower order term).  Hence we we have the additional lower bound
$$ M \gg 3^{-m_0} 2^{a_{[1,m_0]}}.$$
Applying \eqref{mant} with $M_0$ equal to the larger of the two lower bounds on $M$, we conclude that
\begin{align*}
c_{n,a_1,\dots,a_{m_0}}(X) & \ll \frac{3^{n-m_0}}{3^{-m_0} 2^{a_{[1,m_0]}}} +
3^{n-m_0} (2^{a_{[1,m_0]}+1} 3^{n-m_0})^{-1} \log O \left( \frac{3^{-m_0} 2^{a_{[1,m_0]}} x}{3^{-m_0} 2^{a_{[1,m_0-1]}} x } \right)\\
&\ll 3^n 2^{-a_{[1,m_0]}} + 2^{-a_{[1,m_0]}} a_{m_0}\\
&\ll 2^{-a_{[1,m_0]}/2}
\end{align*}
since $2^{-a_{[1,m_0]}} \leq x^{-1/4} 2^{-a_{[1,m_0]}/2} \leq 3^{-n} 2^{-a_{[1,m_0]}/2}$ for $n \in I_y$.
Thus we have
$$ c_n(X) \ll \sum_{a_1,\dots,a_{m_0} \in \N} 2^{-a_{[1,m_0]}/2}$$
and the claim follows from summing the geometric series.
\end{proof}

From the above lemma and \eqref{lo}, we may write \eqref{soda} as
$$  \E c_n( F_{n-m_0}(\a_1,\dots,\a_{n-m_0}) \mod 3^{n-m_0} ) +  O( \log^{-c} x  )$$
which by \eqref{syrac-def} is equal to
$$ \sum_{X \in \Z/3^{n-m_0}\Z} c_n(X) \P( \Syrac(\Z/3^{n-m_0}\Z) = X ) +  O( \log^{-c} x  ).$$
From \eqref{iy-def}, \eqref{mon} we have $n-m_0 \geq m_0$.
Applying Proposition \ref{tv-bound}, Lemma \ref{loam} and the triangle inequality, one can thus write the preceding expression as
$$ \sum_{X \in \Z/3^{n-m_0}\Z} c_n(X) 3^{2m_0-n} \P( \Syrac(\Z/3^{m_0}\Z) = X \mod 3^{m_0}) +  O( \log^{-c} x  )$$
and the claim \eqref{zeno} then follows from \eqref{cn-def}.

\section{Reduction to Fourier decay bound}\label{decay-sec}

In this section we derive Proposition \ref{tv-bound} from Proposition \ref{f-decay}.  We first observe that to prove Proposition \ref{tv-bound}, it suffices to do so in the regime
\begin{equation}\label{9n}
 0.9 n \leq m \leq n.
\end{equation}
(The main significance of the constant $0.9$ here is that it lies between $\frac{\log 3}{2\log 2} \approx 0.7925$ and $1$.)  Indeed, once one has \eqref{yal} in this regime, one also has from \eqref{xlk} that
$$\sum_{Y \in \Z/3^{n'}\Z} \left|3^{n-n'} \P( \Syrac(\Z/3^n\Z) = Y \mod 3^n ) -
3^{m-n'} \P( \Syrac(\Z/3^n\Z) = Y \mod 3^m ) \right| \ll_A m^{-A}
$$
whenever $0.9 n \leq m \leq n \leq n'$, and the claim \eqref{yal} for general $10 \leq m \leq n$ then follows from telescoping series, with the remaining cases $1 \leq m < 10$ following trivially from the triangle inequality.

Henceforth we assume \eqref{9n}.  We also fix $A>0$, and let $C_A$ be a constant that is sufficiently large depending on $A$.
We may assume that $n$ (and hence $m$) are sufficiently large depending on $A,C_A$, since the claim is trivial otherwise.

Let $(\a_1,\dots,\a_n) \equiv \Geom(2)^n$, and define the random variable
$$ \X_n \coloneqq 2^{-\a_1} + 3^1 2^{-\a_{[1,2]}} + \dots + 3^{n-1} 2^{-\a_{[1,n]}} \mod 3^n,$$
thus $\X_n \equiv \Syrac(\Z/3^n\Z)$.  The strategy will be to split $\X_n$ (after some conditioning and removal of exceptional events) as the sum of two independent components, one of which has quite large entropy (or more precisely, Renyi $2$-entropy) in $\Z/3^n\Z$ thanks to some elementary number theory, and the other having very small Fourier coefficients at high frequencies thanks to Proposition \ref{f-decay}.  The desired bound will then follow from some $L^2$-based Fourier analysis (i.e.,  Plancherel's theorem).

We turn to the details.  Let $E$ denote the event that the inequalities
\begin{equation}\label{ij}
 |\a_{[i,j]} - 2(j-i)| \leq C_A ( \sqrt{(j-i)(\log n)} + \log n )
\end{equation}
hold for every $1 \leq i \leq j \leq n$.  The event $E$ occurs with nearly full probability; indeed, from Lemma \ref{chern} and the union bound, we can bound the probability of the complementary event $\overline{E}$ by
\begin{equation}\label{ayo}
\begin{split}
\P( \overline{E}) &\ll \sum_{1 \leq i \leq j \leq n} G_{j-i}( c C_A ( \sqrt{(j-i)(\log n)} + \log n ) ) \\
&\ll \sum_{1 \leq i \leq j \leq n} \exp( - c C_A \log n ) + \exp( - c C_A \log n) \\
&\ll n^2 n^{-c C_A} \\
&\ll n^{-A-1}
\end{split}
\end{equation}
if $C_A$ is large enough.  By the triangle inequality, we may then bound the left-hand side of \eqref{yal} by
$$
\Osc_{m,n}\left( \P( (\X_n = Y) \wedge E ) \right)_{Y \in \Z/3^n\Z} + O(n^{-A-1}),$$
so it now suffices to show that
$$
\Osc_{m,n}\left( \P( (\X_n = Y) \wedge E ) \right)_{Y \in \Z/3^n\Z} \ll_{A,C_A} n^{-A}.$$

Now suppose that $E$ holds.  From \eqref{ij} we have
$$ \a_{[1,n]} \geq 2(n-1) - C_A (\sqrt{n \log n} + \log n) > n \frac{\log 3}{\log 2}$$
since $\frac{\log 3}{\log 2} < 2$ and $n$ is large.
Thus, there is a well defined \emph{stopping time} $0 \leq \k < n$, defined as the unique natural number $\k$ for which
$$ \a_{[1,\k]} \leq n \frac{\log 3}{\log 2} - (C_A)^2 \log n < \a_{[1,\k+1]}.$$
From \eqref{ij} we have
$$ \k = n \frac{\log 3}{2 \log 2} + O( C_A \sqrt{n\log n} ).$$
It thus suffices by the union bound to show that
\begin{equation}\label{psy}
\Osc_{m,n}\left( \P( (\X_n = Y) \wedge E \wedge B_k) \right)_{Y \in \Z/3^n\Z} \ll_{A,C_A} n^{-A-1}
\end{equation}
for all
\begin{equation}\label{hepta}
k = n \frac{\log 3}{2 \log 2} + O( C_A \sqrt{n \log n} ),
\end{equation}
where $B_k$ is the event that $\k=k$, or equivalently that
\begin{equation}\label{bmd}
 \a_{[1,k]}\leq n \frac{\log 3}{\log 2} - (C_A)^2 \log n < \a_{[1,k+1]}.
\end{equation}
Fix $k$.  In order to decouple the events involved in \eqref{psy} we need to enlarge the event $E$ slightly, so that it only depends on $\a_1,\dots,\a_{k+1}$ and not on $\a_{k+2},\dots,\a_n$.  Let $E_k$ denote the event that the inequalities \eqref{ij} hold for $1 \leq i < j \leq k+1$, thus $E_k$ contains $E$.   Then the difference between $E$ and $E_k$ has probability $O(n^{-A-1})$ by \eqref{ayo}.  Thus by the triangle inequality, the estimate \eqref{psy} is equivalent to
$$
\Osc_{m,n}\left( \P( (\X_n = Y) \wedge E_k \wedge B_k ) \right)_{Y \in \Z/3^n\Z} \ll_{A,C_A} n^{-A-1}.$$
From \eqref{bmd} and \eqref{ij} we see that we have
\begin{equation}\label{bmd-2}
n \frac{\log 3}{\log 2} - (C_A)^2 \log n \leq  \a_{[1,k+1]} \leq n \frac{\log 3}{\log 2} - 0.99 (C_A)^2 \log n.
\end{equation}
whenever one is in the event $E_k \wedge B_k$.  By a further application of the triangle inequality, it suffices to show that
$$
\Osc_{m,n}\left( \P( (\X_n = Y) \wedge E_k \wedge B_k \wedge C_{k,l} ) \right)_{Y \in \Z/3^n\Z} \ll_{A,C_A} n^{-A-2}
$$
for all $l$ in the range
\begin{equation}\label{bmd-3}
 n \frac{\log 3}{\log 2} - (C_A)^2 \log n \leq  l \leq n \frac{\log 3}{\log 2} - 0.99 (C_A)^2 \log n,
\end{equation}
where $C_{k,l}$ is the event that $\a_{[1,k+1]}=l$.

Fix $l$.  If we let $g = g_{n,k,l}\colon \Z/3^n\Z \to \R$ denote the function
\begin{equation}\label{g-def}
g(Y) =  g_{n,k,l}(Y) \coloneqq \P( (\X_n = Y) \wedge E_k \wedge B_k \wedge C_{k,l})
\end{equation}
then our task can be written as
$$
\sum_{Y \in \Z/3^n\Z} \left|g(Y) - \frac{1}{3^{n-m}} \sum_{Y' \in \Z/3^n\Z: Y' = Y \mod 3^m} g(Y') \right| \ll_{A,C_A} n^{-A-2}.$$
By Cauchy-Schwarz, it suffices to show that
\begin{equation}\label{half}
3^{n} \sum_{Y \in \Z/3^n\Z} \left|g(Y) - \frac{1}{3^{n-m}} \sum_{Y' \in \Z/3^n\Z: Y' = Y \mod 3^m} g(Y') \right|^2 \ll_{A,C_A} n^{-2A-4}.
\end{equation}
By the Fourier inversion formula, we have
$$ g(Y) =
3^{-n} \sum_{\xi \in \Z/3^n\Z} \left( \sum_{Y' \in \Z/3^n\Z} g(Y') e^{-2\pi i \xi Y' / 3^n} \right) e^{2\pi i \xi Y/3^n}$$
and
$$ \frac{1}{3^{n-m}} \sum_{Y' \in \Z/3^n\Z: Y' = Y \mod 3^m} g(Y') =
3^{-n} \sum_{\xi \in 3^{n-m}\Z/3^n\Z} \left( \sum_{Y' \in \Z/3^n\Z} g(Y') e^{-2\pi i \xi Y' / 3^n} \right) e^{2\pi i \xi Y/3^n}$$
for any $Y \in \Z/3^n\Z$, so by Plancherel's theorem, the left-hand side of \eqref{half} may be written as
$$ \sum_{\xi \in \Z/3^n\Z: \xi \not \in 3^{n-m}\Z/3^n\Z} \left| \sum_{Y \in \Z/3^n\Z} g(Y) e^{-2\pi i \xi Y / 3^n} \right|^2.$$
By \eqref{g-def}, we can write
$$ \sum_{Y \in \Z/3^n\Z} g(Y) e^{-2\pi i \xi Y / 3^n}  = \E e^{-2\pi i \xi \X_n / 3^n} 1_{E_k \wedge B_k \wedge C_{k,l}}.$$
On the event $C_{k,l}$, one can use \eqref{fn-def}, \eqref{syria} to write
$$ \X_n = F_{k+1}(\a_{k+1},\dots,\a_1) + 3^{k+1} 2^{-l} F_{n-k-1}(\a_n,\dots,\a_{k+2}) \mod 3^n.$$
The key point here is that the random variable $3^{k+1} 2^{-l} F_{n-k-1}(\a_n,\dots,\a_{k+2})$ is independent of $\a_1,\dots,\a_{k+1}, E_k, B_k, C_{k,l}$.  Thus we may factor
\begin{align*}
 \sum_{Y \in \Z/3^n\Z} g(Y) e^{-2\pi i \xi Y / 3^n}  &= \E e^{-2\pi i \xi (F_{k+1}(\a_{k+1},\dots,\a_1) \mod 3^n) / 3^n} 1_{E_k \wedge B_k \wedge C_{k,l}} \\
&\quad \times \E e^{-2\pi i \xi (2^{-l} F_{n-k-1}(\a_n,\dots,\a_{k+2}) \mod 3^{n-k-1}) / 3^{n-k-1}}.
\end{align*}
For $\xi$ in $\Z/3^n\Z$ that does not lie in $3^{n-m}\Z/3^n\Z$, we can write $\xi = 3^j 2^l \xi' \mod 3^n$ where $0 \leq j < n-m \leq 0.1 n$ and $\xi'$ is not divisible by $3$.  In particular, from \eqref{hepta} one has
$$ n-k-j-1 \geq 0.9 n - n \frac{\log 3}{2 \log 2} - O(C_A \sqrt{n \log n}) - 1 \gg n.$$
Then by \eqref{xlk} we have
$$ \E e^{-2\pi i \xi (2^{-l} F_{n-k-1}(\a_n,\dots,\a_{k+2}) \mod 3^{n-k-1}) / 3^{n-k-1}} = \E e^{-2\pi i \xi' \Syrac(\Z/3^{n-k-j-1}\Z) / 3^{n-k-j-1}} $$
and hence by Proposition \ref{f-decay} this quantity is $O_{A'}(n^{-A'})$ for any $A'$.  Thus we can bound the left-hand side of \eqref{half} by
\begin{equation}\label{half-2}
 \ll_{A'} n^{-2 A'} \sum_{\xi \in \Z/3^n\Z} \left|\E e^{-2\pi i \xi (F_{k+1}(\a_{k+1},\dots,\a_1) \mod 3^n) / 3^n} 1_{E_k \wedge B_k \wedge C_{k,l}} \right|^2
\end{equation}
(where we have now discarded the restriction $\xi \not \in 3^{n-m}\Z/3^n\Z$); by Plancherel's theorem, this expression can be written as
$$ \ll_{A'} n^{-2 A'} 3^n \sum_{Y_{k+1} \in \Z/3^n\Z} \P ( (F_{k+1}(\a_{k+1},\dots,\a_1) = Y_{k+1}) \wedge E_k \wedge B_k \wedge C_{k,l} )^2.$$

\begin{remark} If we ignore the technical restriction to the events $E_k, B_k, C_{k,l}$, this quantity is essentially the Renyi $2$-entropy (also known as \emph{collision entropy}) of the random variable $F_{k+1}(\a_{k+1},\dots,\a_1) \mod 3^n$.
\end{remark}

Now we make a key elementary number theory observation:

\begin{lemma}[Injectivity of offsets]\label{inj}  For each natural number $n$, the $n$-Syracuse offset map $F_n\colon (\N+1)^n \to \Z[\frac{1}{2}]$ is injective.
\end{lemma}

\begin{proof}  Suppose that $(a_1,\dots,a_n), (a'_1,\dots,a'_n) \in (\N+1)^n$ are such that $F_n(a_1,\dots,a_n) = F_n(a'_1,\dots,a'_n)$.  Taking $2$-valuations of both sides using \eqref{fn-def}, we conclude that
$$ - a_{[1,n]} = - a'_{[1,n]}.$$
On the other hand, from \eqref{fn-def} we have
$$ F_n(a_1,\dots,a_n) = 3^n 2^{-a_{[1,n]}} + F_{n-1}(a_2,\dots,a_n)$$
and similarly for $a'_1,\dots,a'_n$, hence
$$ F_{n-1}(a_2,\dots,a_n) = F_{n-1}(a'_2,\dots,a'_n).$$
The claim now follows from iteration (or an induction on $n$).
\end{proof}

We will need a more quantitative $3$-adic version of this injectivity:

\begin{corollary}[$3$-adic separation of offsets]  Let $C_A$ be sufficiently large, let $n$ be sufficiently large (depending on $C_A$), let $k$ be a natural number, and let $l$ be a natural number obeying \eqref{bmd-3}.  Then the residue classes $F_{k+1}(a_{k+1},\dots,a_1) \mod 3^n$, as $(a_1,\dots,a_{k+1}) \in (\N+1)^{k+1}$ range over ${k+1}$-tuples of positive integers that obey the conditions
\begin{equation}\label{amb-x}
 |a_{[i+1,j]} - 2(j-i)| \leq C_A \left( \sqrt{(j-i)(\log n)} + \log n \right)
\end{equation}
for $1 \leq i < j \leq k+1$ as well as
\begin{equation}\label{amb-2x}
a_{[1,k+1]} = l,
\end{equation}
are distinct.
\end{corollary}

\begin{proof}  Suppose that $(a_1,\dots,a_{k+1}), (a'_1,\dots,a'_{k+1})$ are two tuples of positive integers that both obey \eqref{amb-x}, \eqref{amb-2x}, and such that
$$ F_{k+1}(a_{k+1},\dots,a_1) = F_{k+1}(a'_{k+1},\dots,a'_1) \mod 3^n.$$
Applying \eqref{fn-def} and multiplying by $2^l$, we conclude that
\begin{equation}\label{lrs}
 \sum_{j=1}^{k+1} 3^{j-1} 2^{l - a_{[1,j]}} = \sum_{j=1}^{k+1} 3^{j-1} 2^{l - a'_{[1,j]}} \mod 3^n.
\end{equation}
From \eqref{amb-2x}, the expressions on the left and right sides are natural numbers.  Using \eqref{amb-x}, \eqref{bmd-3}, and Young's inequality $C_A j^{1/2} \log^{1/2} n \leq \frac{\eps}{2} j + \frac{1}{2\eps} C_A^2 \log n$ for a suitable choice of $\eps>0$, the left-hand side may be bounded for $C_A$ large enough by
\begin{align*}
 \sum_{j=1}^{k+1} 3^{j-1} 2^{l - a_{[1,j]}}
&\ll 2^l \sum_{j=1}^{k+1} 3^{j} 2^{- 2j + C_A (\sqrt{j \log n} + \log n)} \\
&\ll \exp( - 0.99 \log 2 (C_A)^2 \log n ) 3^n \sum_{j=1}^{k+1} \exp\left( - j \log \frac{4}{3} + \log 2 C_A j^{1/2} \log^{1/2} n + O( C_A \log n ) \right) \\
&\ll \exp\left( - c (C_A)^2 \log n \right) 3^n \sum_{j=1}^{k+1} \exp( - c j ) \\
&\ll n^{- c (C_A)^2} 3^n
\end{align*}
(here we use the fact that $\frac{\log^2 2}{4 \log \frac{4}{3}} \approx 0.4175$ is smaller than $0.99 \log 2 \approx 0.6862$); in particular, for $n$ large enough, this expression is less than $3^n$.  Similarly for the right-hand side of \eqref{lrs}.  Thus these two sides are equal as natural numbers, not simply as residue classes modulo $3^n$:
\begin{equation}\label{span}
 \sum_{j=1}^{k+1} 3^{j-1} 2^{l-a_{[1,j]}} = \sum_{j=1}^{k+1} 3^{j-1} 2^{l-a'_{[1,j]}}.
\end{equation}
Dividing by $2^l$, we conclude $F_{k+1}(a_{k+1},\dots,a_1) = F_{k+1}(a'_{k+1},\dots,a'_1)$.  From Lemma \ref{inj} we conclude that $(a_1,\dots,a_{k+1}) = (a'_1,\dots,a'_{k+1})$, and the claim follows.
\end{proof}

In view of the above lemma, we see that for a given choice of $Y_{k+1} \in \Z/3^n\Z$, the event
$$ (F_{k+1}(\a_{k+1},\dots,\a_1) = Y_{k+1}) \wedge E_k \wedge B_k \wedge C_{k,l} $$
can only be non-empty for at most one value $(a_1,\dots,a_{m})$ of the tuple $(\a_1,\dots,\a_{m})$.  By Definition \ref{geom}, such a value is attained with probability $2^{-a_{[1,m]}} = 2^{-l}$, which by \eqref{bmd-3} is equal to $n^{O((C_A)^2)} 3^{-n}$.  We can thus bound \eqref{half-2} (and hence the left-hand side of \eqref{half}) by
$$ \ll_{A'} n^{-2 A' + O( (C_A)^2 ) },$$
and the claim now follows by taking $A'$ large enough.  This concludes the proof of Proposition \ref{tv-bound} assuming Proposition \ref{f-decay}.

\section{Decay of Fourier coefficients}\label{fourier-sec}

In this section we establish Proposition \ref{f-decay}, which when combined with all the implications established in preceding sections will yield Theorem \ref{main}.

Let $n \geq 1$, let $\xi \in \Z/3^n\Z$ be not divisible by $3$, and let $A>0$ be fixed. We will not vary $n$ or $\xi$ in this argument, but it is important that all of our estimates are uniform in these parameters.
Without loss of generality we may assume $A$ to be larger than any fixed absolute constant.  We let $\chi = \chi_{n,\xi} \colon \Z[\frac{1}{2}] \to \C$ denote the character
\begin{equation}\label{chin}
 \chi(x) \coloneqq e^{-2\pi i \xi (x \mod 3^n) / 3^n}
\end{equation}
where $x \mapsto x \mod 3^n$ is the ring homomorphism from $\Z[\frac{1}{2}]$ to $\Z/3^n\Z$ (mapping $\frac{1}{2}$ to $\frac{1}{2} \mod 3^n = \frac{3^n+1}{2} \mod 3^n$).  Note that $\chi$ is a group homomorphism from the additive group $\Z[\frac{1}{2}]$ to the multiplicative group $\C$, which is periodic modulo $3^n$, so it also descends to a group homomorphism from $\Z/3^n\Z$ to $\C$, which is still defined by the same formula \eqref{chin}. From \eqref{syria}, our task now reduces\footnote{Note that we have reversed the order of variables $\a_1,\dots,\a_n$ from that in \eqref{fn-def}, as this will be a slightly more convenient normalization for the arguments in this section.} to establishing the following claim.

\begin{proposition}[Key Fourier decay estimate]\label{key}  Let $\chi$ be defined by \eqref{chin}, and let $(\a_1,\dots,\a_n) \equiv \Geom(2)^n$ be $n$ iid copies of the geometric distribution $\Geom(2)$ (as defined in Definition \ref{geom}).  Then the quantity
\begin{equation}\label{schi-def}
S_\chi(n) \coloneqq \E \chi( 2^{-\a_1} + 3^1 2^{-\a_{[1,2]}} + \dots + 3^{n-1} 2^{-\a_{[1,n]}} )
\end{equation}
obeys the estimate
\begin{equation}\label{echi}
S_\chi(n) \ll_A n^{-A}
\end{equation}
for any $A>0$, where we use the summation convention $\a_{[i,j]} \coloneqq \a_i + \dots + \a_j$ from \eqref{sum}.
\end{proposition}

\subsection{Estimation in terms of white points}

To extract some usable cancellation in the expression $S_\chi(n)$, we will group the sum on the left-hand side into pairs.  For any real $x>0$, let $[x]$ denote the discrete interval
$$ [x] \coloneqq \{j \in \N+1: j \leq x \} = \{ 1, \dots, \lfloor x\rfloor\}.$$
For $j \in [n/2]$, set $\b_j \coloneqq \a_{2j-1} + \a_{2j}$, so that
$$2^{-\a_1} + 3^1 2^{-\a_{[1,2]}} + \dots + 3^{n-1} 2^{-\a_{[1,n]}} = \sum_{j \in [n/2]} 3^{2j-2} 2^{-\b_{[1,j]}} ( 2^{\a_{2j}} + 3 ) + 3^{n-1} 2^{-\b_{[1,\lfloor n/2\rfloor]}-\a_n}$$
when $n$ is odd, where we extend the summation notation \eqref{sum} to the $\b_j$.  For $n$ even, the formula is the same except that the final term $3^{n-1} 2^{-\b_{[1,\lfloor n/2\rfloor]}-\a_n}$ is omitted.  Note that the $\b_1,\dots,\b_{\lfloor n/2\rfloor}$ are jointly independent random variables taking values in $\N+2 = \{2,3,4,\dots\}$; they are iid copies of a Pascal (or negative binomial) random variable $\Pascal \equiv \mathbf{NB}(2,\frac{1}{2})$ on $\N+2$, defined by
$$ \P( \Pascal = b ) = \frac{b-1}{2^b}$$
for $b \in \N+2$.

For any $j \in [n/2]$, $\a_{2j}$ is independent of all of the $\b_1,\dots,\b_{\lfloor n/2\rfloor}$ except for $\b_j$.  For $n$ odd, $\a_n$ is independent of all of the $\b_j$.  Regardless of whether $n$ is even or odd, once one conditions on all of the $\b_j$ to be fixed, the random variables $\a_{2j}, j \leq [n/2]$ (as well as $\a_n$, if $n$ is odd) are all independent of each other.  We conclude that
$$ S_\chi(n) = \E \left(\prod_{j \in [n/2]} f( 3^{2j-2} 2^{-\b_{[1,j]}}, \b_j )\right) g( 3^{n-1} 2^{-\b_{[1,\lfloor n/2\rfloor]}} )$$
when $n$ is odd, with the factor $g( 2^{-\b_{[1,\lfloor n/2\rfloor]}} )$ omitted when $n$ is even, where $f(x,b)$ is the conditional expectation
\begin{equation}\label{fxb}
 f( x, b ) \coloneqq \E \left( \chi( x (2^{\a_2}+3)) | \a_1 + \a_2 = b \right)
\end{equation}
(with $(\a_1,\a_2) \equiv \Geom(2)^2$) and
$$ g(x) \coloneqq \E \chi( x 2^{-\Geom(2)} ).$$
Clearly $|g(x)| \leq 1$, so by the triangle inequality we can bound
\begin{equation}\label{74}
|S_\chi(n)| \leq \E \prod_{j \in [n/2]} |f( 3^{2j-2} 2^{-\b_{[1,j]}}, \b_j )|
\end{equation}
regardless of whether $n$ is even or odd.

From \eqref{fxb} we certainly have
\begin{equation}\label{fxb-triv}
|f(x,b)| \leq 1.
\end{equation}
We now perform an explicit computation to improve upon this estimate for many values of $x$ (of the form $x = 3^{2j-2} 2^{-l}$) in the case $b = 3$, which is the least value of $b \in \N+2$ for which the event $\a_1+\a_2=b$ does not completely determine $\a_1$ or $\a_2$.  For any $(j,l) \in (\N+1) \times \Z$, we can write
\begin{equation}\label{chi-id}
\chi( 3^{2j-2} 2^{-l+1} ) = e^{-2\pi i \theta(j,l)}
\end{equation}
where $\theta(j,l) = \theta_{n,\xi}(j,l) \in (-1/2,1/2]$ denotes the argument
\begin{equation}\label{alpha-def}
 \theta(j,l) \coloneqq \left\{ \frac{\xi 3^{2j-2} (2^{-l+1} \mod 3^n)}{3^n} \right\}
\end{equation}
and $\{\}\colon \R/\Z \to (-1/2,1/2]$ is the signed fractional part function, thus $\{x\}$ denotes the unique element of the coset $x + \Z$ that lies in $(-1/2,1/2]$.

Let $0 < \eps < \frac{1}{100}$ be a sufficiently small absolute constant to be chosen later; we will take care to ensure that the implied constants in many of our asymptotic estimates do not depend on $\eps$.
Call a point $(j,l) \in [n/2] \times \Z$ \emph{black}\footnote{This choice of notation was chosen purely in order to be consistent with the color choices in Figures \ref{fig:triangle}, \ref{fig:triangles}, \ref{fig:cases}.} if
\begin{equation}\label{wh}
|\theta(j,l)| \leq \eps,
\end{equation}
and \emph{white} otherwise.  We let $B = B_{n,\xi}, W = W_{n,\xi}$ denote the black and white points of $[n/2] \times \Z$ respectively, thus we have the partition $[n/2] \times \Z = B \uplus W$.

\begin{lemma}[Cancellation for white points]  If $(j,l)$ is white, then
$$|f(3^{2j-2} 2^{-l},3)| \leq \exp( -\eps^3 ).$$
\end{lemma}

\begin{proof}
If $\a_1,\a_2$ are independent copies of $\Geom(2)$, then after conditioning to the event $\a_1+\a_2 = 3$, the pair $(\a_1,\a_2)$ is equal to either $(1,2)$ or $(2,1)$, with each pair occuring with (conditional) probability $1/2$.  From \eqref{fxb} we thus have
$$ f(x,3) = \frac{1}{2} \chi( 5 x ) + \frac{1}{2} \chi( 7 x ) = \frac{\chi(5x)}{2} (1 + \chi(2x))$$
for any $x$, so that
$$ |f(x,3)| = \frac{|1 + \chi(2x)|}{2}.$$
We specialise to the case $x \coloneqq 3^{2j-2} 2^{-l}$.  By \eqref{chi-id} we have
$$  \chi(2x) = e^{-2\pi i \theta(j, \b_{[1,j]})} $$
and hence by elementary trigonometry
$$ |f(3^{2j-2} 2^{-l},3)| = \cos( \pi \theta(j, l) ).$$
By hypothesis we have
$$
 |\theta(j, l)| > \eps
$$
and the claim now follows by Taylor expansion (if $\eps$ is small enough); indeed one can even obtain an upper bound of $\exp(-c\eps^2)$ for some absolute constant $c>0$ independent of $\eps$.
\end{proof}

From the above lemma, \eqref{fxb-triv}, and the law of total probability, we see that
$$ |S_\chi(n)| \leq \E \exp( - \eps^3 \# \{ j \in [n/2]: \b_j = 3, (j,\b_{[1,j]}) \in W \} ).$$
As we shall see later, we can interpret the $(j,\b_{[1,j]})$ with $\b_j=3$ as a two-dimensional renewal process. To establish Proposition \ref{key} (and thus Proposition \ref{f-decay} and Theorem \ref{main}), it thus suffices to show the following estimate.

\begin{proposition}[Renewal process encounters many white points]\label{jeo-prop}
\begin{equation}\label{jeo}
\E \exp( - \eps^3 \# \{ j \in [n/2]: \b_j = 3, (j,\b_{[1,j]}) \in W \} ) \ll_A n^{-A}.
\end{equation}
\end{proposition}

We remark that this proposition is of a simpler nature to establish than Proposition \ref{key} as it is entirely ``non-negative''; it does not require the need to capture any cancellation in an oscillating sum, as was the case in Proposition \ref{key}.

\subsection{Deterministic structural analysis of black points}

The proof of Proposition \ref{jeo-prop} consists of a ``deterministic'' part, in which we understand the structure of the white set $W$ (or the black set $B$), and a ``probabilistic'' part, in which we control the random walk $\b_{[1,j]}$ and the events $\b_j=3$.  We begin with the former task.  Define a \emph{triangle} to be a subset $\Delta$ of $(\N+1) \times \Z$ of the form
\begin{equation}\label{ts}
 \Delta = \{ (j,l): j \geq j_\Delta; l \leq l_\Delta; (j-j_\Delta) \log 9 + (l_\Delta-l) \log 2 \leq s_\Delta \}
\end{equation}
for some $(j_\Delta, l_\Delta) \in (\N+1) \times \Z$ (which we call the \emph{top left corner} of $\Delta$) and some $s_\Delta \geq 0$ (which we call the \emph{size} of $\Delta$); see Figure \ref{fig:triangle}.

\begin{figure} [t]
\centering
\includegraphics[width=4in]{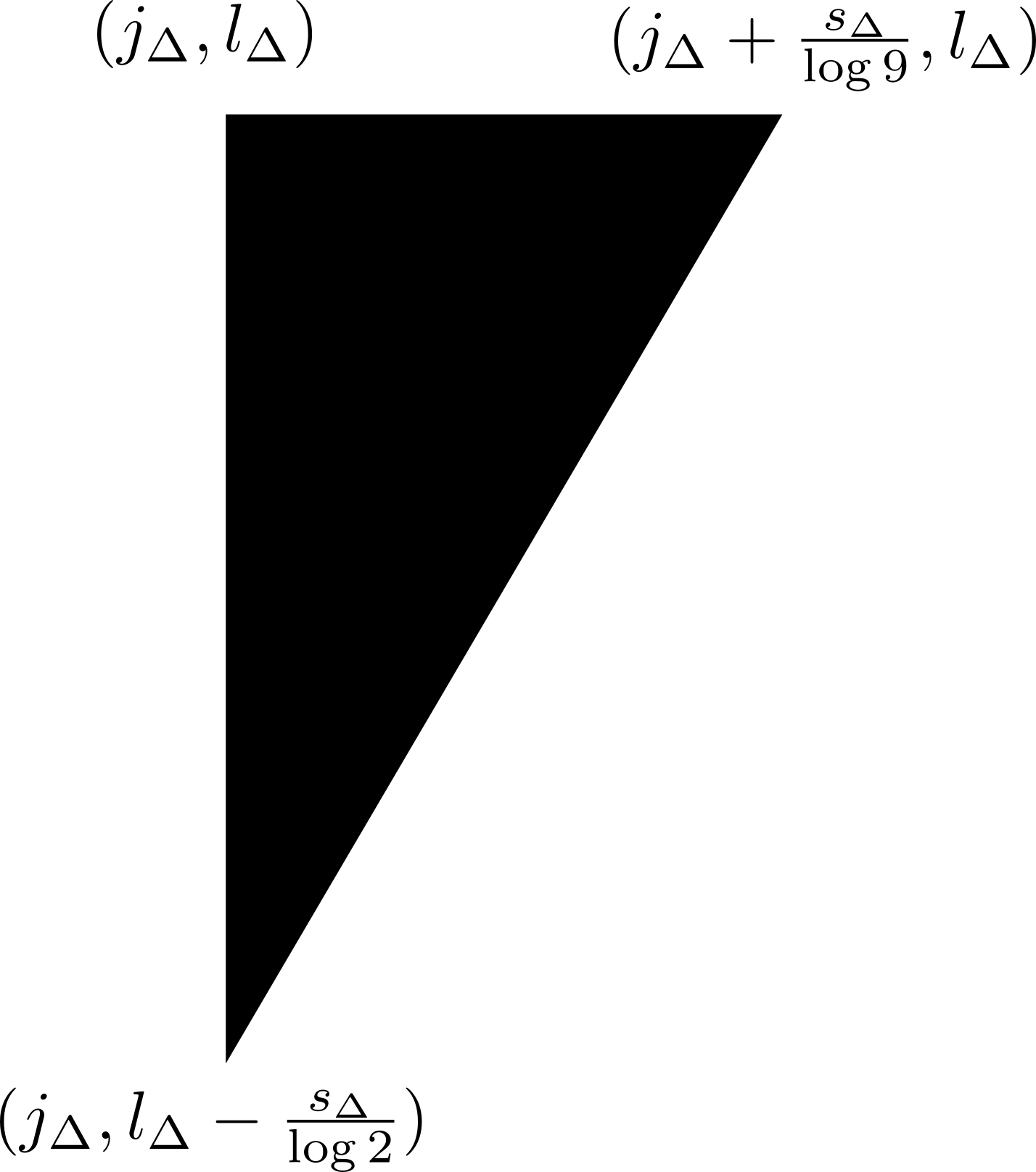}
\caption{A triangle $\Delta$, which we have drawn as a solid region rather than as a subset of the discrete lattice $\Z^2$.}
\label{fig:triangle}
\end{figure}

\begin{lemma}[Structure of black set]\label{black}  The black set $B \subset [n/2] \times \Z$ of points $(j,l)$ with $|\theta(j,l)| \leq \eps$ can be expressed as a disjoint union
$$ B = \biguplus_{\Delta \in {\mathcal T}} \Delta$$
of triangles $\Delta$, each of which is contained in $[\frac{n}{2} - \frac{1}{10} \log \frac{1}{\eps}] \times \Z$.
Furthermore, any two triangles $\Delta,\Delta'$ in ${\mathcal T}$ are separated by a distance $\geq \frac{1}{10} \log \frac{1}{\eps}$ (using the Euclidean metric on $[n/2] \times \Z \subset \R^2$).  (See Figure \ref{fig:triangles}.)
\end{lemma}

\begin{figure} [t]
\centering
\includegraphics[width=5in]{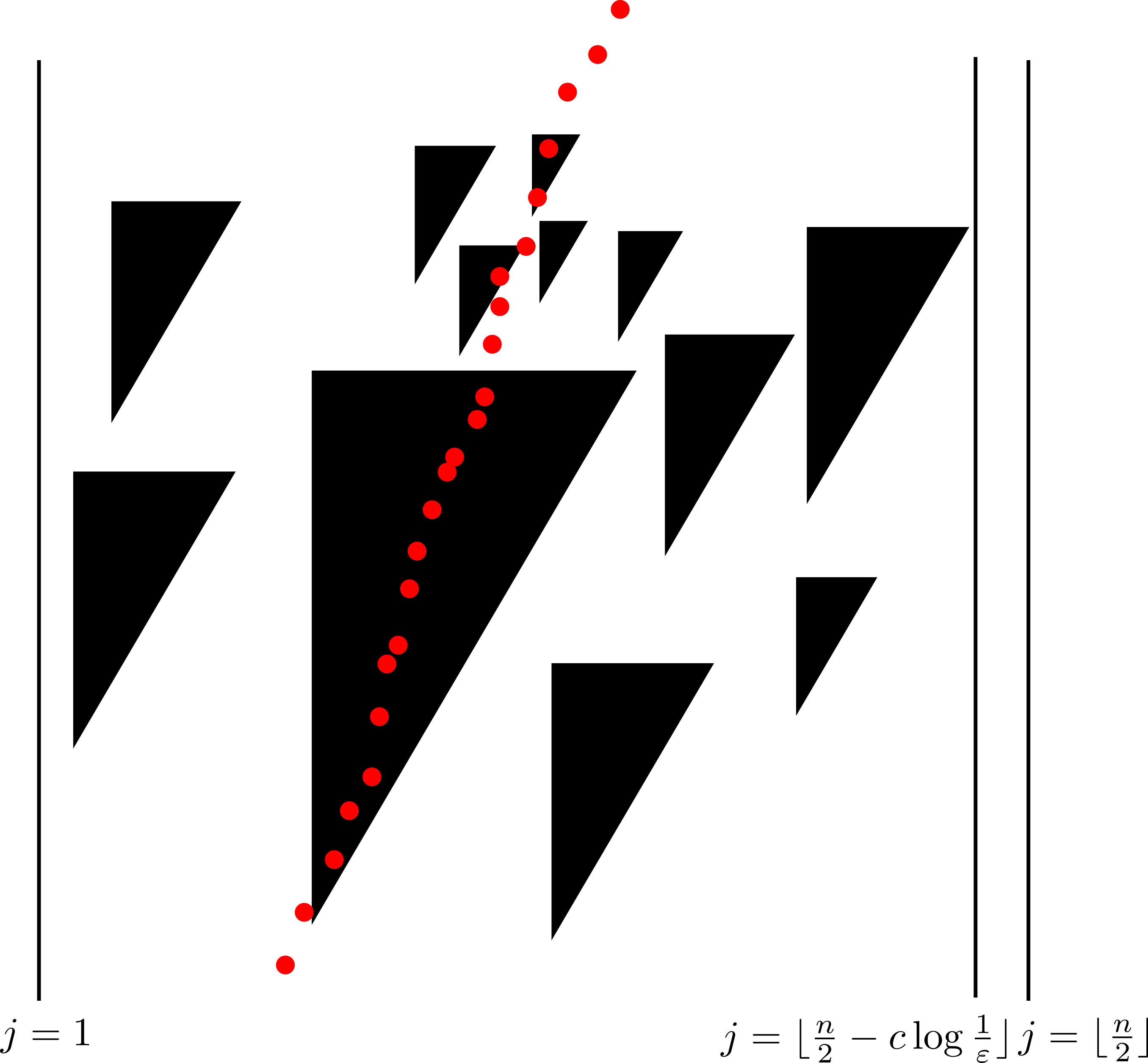}
\caption{The black set is a union of triangles, in the strip $[\frac{n}{2} - \frac{1}{10} \log \frac{1}{\eps}] \times \Z$, that are separated from each other by $\gg \log \frac{1}{\eps}$.  The red dots depict (a portion of) a renewal process $\v_1, \v_{[1,2]}, \v_{[1,3]}$ that we will encounter later in this section; our main objective will be to establish that this process usually contains a fair number of white points.  We remark that the average slope $\frac{16}{4}=4$ of this renewal process will exceed the slope $\frac{\log 9}{\log 2} \approx 3.17$ of the triangle diagonals, so that the process tends to exit a given triangle through its horizontal side.  The coordinate $j$ increases in the rightward direction, while the coordinate $l$ increases in the upward direction.}
\label{fig:triangles}
\end{figure}

\begin{proof}  We first observe some simple relations between adjacent values of $\theta$.
From \eqref{alpha-def} (or \eqref{chi-id}) we observe the identity
\begin{equation}\label{ide}
 3^{2(j_*-j)} 2^{(l-l_*)} \theta(j,l) = \theta(j_*,l_*) \mod \Z
\end{equation}
whenever $j \leq j_*$ and $l \geq l_*$.  Thus for instance
\begin{equation}\label{ide-right}
\theta(j+1,l) = 9\theta(j,l) \mod \Z
\end{equation}
and
\begin{equation}\label{ide-down}
\theta(j,l-1) = 2 \theta(j,l) \mod \Z.
\end{equation}
Among other things, this implies that
$$ \theta(j,l) = \theta(j+1,l) - 4 \theta(j,l-1) \mod \Z$$
and hence by the triangle inequality
\begin{equation}\label{ide-upleft}
|\theta(j,l)| \leq |\theta(j+1,l)| + 4 |\theta(j,l-1)|.
\end{equation}

These identities have the following consequences.  Call a point $(j,l) \in [n/2] \times \Z$ \emph{weakly black} if
$$ |\theta(j,l)| \leq \frac{1}{100}.$$
Clearly any black point is weakly black.  We have the following further claims.
\begin{itemize}
\item[(i)]  If $(j,l)$ is weakly black, and either $(j+1,l)$ or $(j,l-1)$ is black, then $(j,l)$ is black.  (This follows from \eqref{ide-right} or \eqref{ide-down} respectively.)
\item[(ii)] If $(j+1,l), (j,l-1)$ are weakly black, then $(j,l)$ is also weakly black.  (Indeed, from \eqref{ide-upleft} we have $|\theta(j,l)| \leq \frac{5}{100}$, and the claim now follows from \eqref{ide-right} or \eqref{ide-down}.)
\item[(iii)]  If $(j-1, l)$ and $(j,l-1)$ are weakly black, then $(j,l)$ is also weakly black.  (Indeed, from \eqref{ide-right} we have $|\theta(j,l)| \leq \frac{9}{100}$, and the claim now follows from \eqref{ide-down}.)
\end{itemize}

Now we begin the proof of the lemma.  Suppose $(j,l) \in [n/2] \times \Z$ is black, then by \eqref{wh}, \eqref{alpha-def} we have
$$ \frac{\xi 3^{2j-2} (2^{-l+1} \mod 3^n)}{3^n} \in [-\eps, \eps] \mod \Z$$
and hence
$$
 \frac{\xi 3^{n-1} (2^{-l+1} \mod 3^n)}{3^n} \in [-3^{n+1-2j} \eps, 3^{n+1-2j}\eps] \mod \Z.
$$
On the other hand, since $\xi$ is not a multiple of $3$, the expression $\frac{\xi 3^{n-1} (2^{-l+1} \mod 3^n)}{3^n}$ is either equal to $1/3$ or $2/3$ mod $\Z$.  We conclude that
\begin{equation}\label{stil}
3^{n+1-2j}\eps \geq \frac{1}{3},
\end{equation}
so the black points in $[n/2] \times \Z$ actually lie in $[\frac{n}{2} - \frac{1}{10} \log \frac{1}{\eps}] \times \Z$.

Suppose that $(j,l) \in [n/2] \times \Z$ is such that $(j,l')$ is black for all $l' \geq l$, thus
$$ |\theta(j,l')| \leq \eps$$
for all $l' \geq l$.  From \eqref{ide-down} this implies that
$$ \theta(j,l') = 2 \theta(j,l'+1)$$
for all $l' \geq l$, hence
$$ \theta(j,l') \leq 2^{l-l'} \eps$$
for all $l' \geq l$.  Repeating the proof of \eqref{stil}, one concludes that
$$ 3^{n+1-2j} 2^{l-l'} \eps \geq \frac{1}{3},$$
which is absurd for $l'$ large enough.  Thus it is not possible for $(j,l')$ to be black for all $l' \geq l$.

Now let $(j,l) \in [n/2] \times \Z$ be black.  By the preceding discussion, there exists a unique $l_* = l_*(j,l) \geq l$ such that $(j,l')$ is black for all $l \leq l' \leq l_*$, but such that $(j,l_*+1)$ is white.  Now let $j_* = j_*(j,l) \leq j$ be the unique positive integer such that $(j',l_*)$ is black for all $j_* \leq j' \leq j$, but such that either $j_*=1$ or $(j_*-1,l_*)$ is white.  Informally, $(j_*,l_*)$ is obtained from $(j,l)$ by first moving upwards as far as one can go in $B$, then moving leftwards as far as one can go in $B$; see Figure \ref{fig:cases}.  As one should expect from glancing at this figure (or Figure \ref{fig:triangles}), $(j_*,l_*)$ should be the top left corner of the triangle containing $(j,l)$, and the arguments below are intended to support this claim.

By construction, $(j_*,l_*)$ is black, thus by \eqref{wh} we have
\begin{equation}\label{ap}
|\theta(j_*,l_*)| = \eps \exp(-s_*)
\end{equation}
for some $s_* \geq 0$. From \eqref{ide} this implies in particular that
\begin{equation}\label{ap-2}
|\theta(j',l')| \leq \eps \exp(-s_* + (j'-j_*) \log 9 + (l_*-l') \log 2 )
\end{equation}
whenever $j' \geq j_*, l' \geq l_*$, with equality whenever the right-hand side is strictly less than $1/2$.

Let $\Delta_*$ denote the triangle with top left corner $(j_*,l_*)$ and size $s_*$.  If $(j',l') \in \Delta_*$, then by \eqref{ap-2} we have
$$ |\theta(j',l')| \leq 3^{2(j'-j_*)} 2^{(l_*-l')} \eps \exp(-s_*) \leq \eps$$
and hence every element of $\Delta_*$ is black (and thus lies in $[\frac{n}{2} - c \log \frac{1}{\eps}] \times \Z$).

Next, we make the following claim:

\begin{itemize}
\item[(*)] Every point $(j',l') \in [n/2] \times \Z$ that lies outside of $\Delta_*$, but is at a distance of at most $\frac{1}{10} \log \frac{1}{\eps}$ to $\Delta_*$, is white.
\end{itemize}

To verify Claim (*), we divide into three cases (see Figure \ref{fig:cases}):

\begin{figure} [t]
\centering
\includegraphics[width=5in]{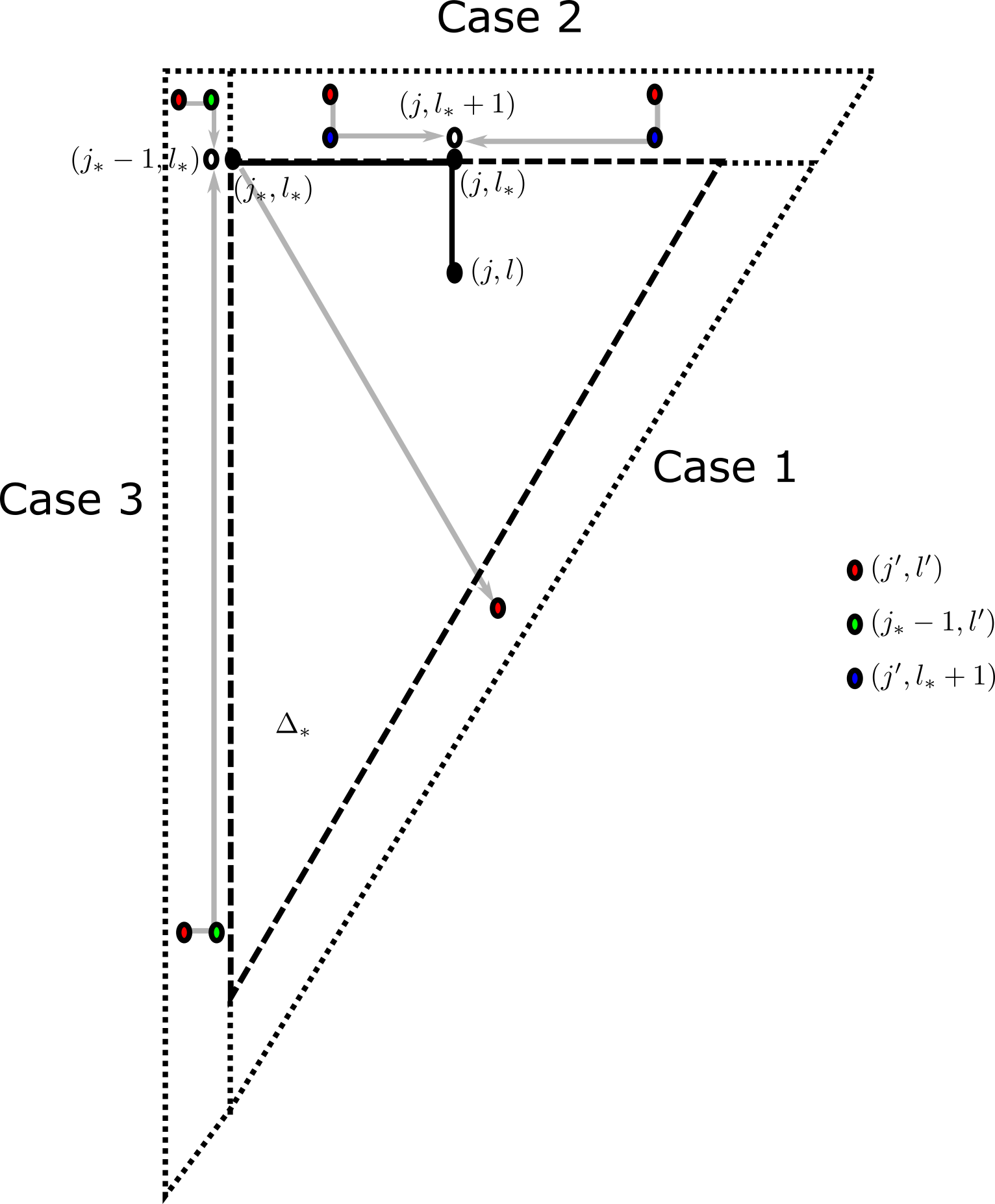}
\caption{The proof of Lemma \ref{black}.  The points connecting $(j,l)$ to $(j,l_*)$, and from $(j,l_*)$ to $(j_*,l_*)$, are known to be black, while the points $(j, l_*+1), (j_*-1, l_*)$ are known to be white.  The point $(j',l')$ can be in various locations, as illustrated by the red dots here.  From \eqref{ap-2} one can obtain that every point in the dashed triangle $\Delta_*$ is black (and every point in the Case 1 region is weakly black), which can treat the Case 1 locations of $(j',l')$ (and also forces $(j,l)$ to lie inside $\Delta_*$).  In Case 2, $(j',l')$ can be to the right or left of $(j,l_*+1)$, but in either case one can show that if $(j',l')$ is black, then $(j',l_*+1)$ (displayed here in blue) is weakly black and hence $(j,l_*+1)$ is weakly black and in fact black, a contradiction.  Similarly, in Case 3, $(j',l')$ can be above or below $(j_*-1,l_*)$, but in either case one can show that if $(j',l')$ is black, then so $(j_*-1,l')$ (displayed here in green) is weakly black and hence $(j_*-1,l_*)$ is weakly black and in fact black, again giving a contradiction.}
\label{fig:cases}
\end{figure}

\textbf{Case 1: $j' \geq j_*, l' \leq l_*$.}  In this case we have from \eqref{ts} that
$$ s_* < (j'-j_*) \log 9 + (l_*-l') \log 2 \leq s_* + \frac{\log 9 + \log 2}{10} \log \frac{1}{\eps}$$
and hence
$$ \eps \exp(-s_* + (j'-j_*) \log 9 + (l_*-l') \log 2 ) \eps^{1-\frac{\log 9 + \log 2}{10}} < \frac{1}{2}.$$
Applying the equality case of \eqref{ap-2}, we conclude that
$$ \theta = \eps \exp(-s_* + (j'-j_*) \log 9 + (l_*-l') \log 2 ) \eps^{1-\frac{\log 9 + \log 2}{10}} > \eps$$
and thus $(j',l')$ is white as claimed.

\textbf{Case 2: $j' \geq j_*, l' > l_*$.}  In this case we have from \eqref{ts} that
\begin{equation}\label{lp}
 0 < (l'-l_*) \log 2 \leq \frac{\log 2}{10} \log \frac{1}{\eps}
\end{equation}
and
\begin{equation}\label{joe}
 (j'-j_*) \log 9 \leq s_* +  \frac{\log 9}{10} \log \frac{1}{\eps}
\end{equation}
(say).  Suppose for contradiction that $(j',l')$ was black, thus
$$ |\theta(j',l')| \leq \eps.$$
From \eqref{lp} and \eqref{ide} (or \eqref{ide-down}) this implies that
$$ |\theta(j',l_*+1)| \leq \eps^{1-\frac{\log 2}{10}},$$
so in particular $(j',l_*+1)$ is weakly black.

If $j' \geq j$, then from \eqref{ap-2}, \eqref{joe} we also have
\begin{equation}\label{theps}
 |\theta(j'-1,l_*)| \leq \eps^{1-\frac{\log 9}{10}},
\end{equation}
thus $(j'-1,l_*)$ is weakly black.  Applying claim (ii) and the fact that $(j',l_*+1)$ is weakly black, we conclude that $(j'-1,l_*+1)$ is weakly black.  Iterating this argument, we conclude that $(j'',l_*+1)$ is weakly black for all $j_* \leq j'' \leq j'$.  In particular, $(j,l_*+1)$ is weakly black; since $(j,l_*)$ is black by construction of $l_*$, we conclude from Claim (i) that $(j,l_*+1)$ is black.  But this contradicts the construction of $l_*$.

Now suppose that $j' < j$.  From construction of $l_*, j_*$ we see that $(j'+1,l_*)$ is black, hence weakly black; since $(j',l_*+1)$ is weakly black, we conclude from Claim (iii) that $(j'+1,l_*+1)$ is weakly black.  Iterating this argument we conclude that $(j'',l_*+1)$ is weakly black for all $j' \leq j'' \leq j$, thus in particular $(j,l_*+1)$ is weakly black, and we obtain a contradiction as before.

\textbf{Case 3: $j' < j_*$.}  Clearly this implies $j_* > 1$; also, from \eqref{ts} we have
\begin{equation}\label{lolo}
 - \frac{\log 2}{10} \log \frac{1}{\eps} \leq (l_*-l') \log 2 \leq s_* + \frac{\log 2}{10} \log \frac{1}{\eps}
\end{equation}
and
\begin{equation}\label{jojo}
 0 < (j_* - j') \log 9 \leq \frac{\log 9}{10} \log \frac{1}{\eps}.
\end{equation}
Suppose for contradiction that $(j',l')$ was black, thus
$$ |\theta(j',l')|  \leq \eps.$$
From \eqref{jojo} and \eqref{ide} (or \eqref{ide-right}) we thus have
\begin{equation}\label{jl}
 |\theta(j_*-1,l')| \leq \eps^{1-\frac{\log 9}{10}}.
\end{equation}
If $l' \geq l_*$, then from \eqref{lolo}, \eqref{ide} we then have
$$ |\theta(j_*-1,l_*)| \leq \eps^{1-\frac{\log 9 + \log 2}{10}},$$
so $(j_*-1,l_*)$ is weakly black.  By construction of $j_*$, $(j_*,l_*)$ is black, hence by Claim (i) $(j_*-1,l_*)$ is black, contradicting the construction of $j_*$.

Now suppose that $l' < l_*$.  From \eqref{jl}, $(j_*-1, l')$ is weakly black.  On the other hand, from \eqref{lolo}, \eqref{ap-2} that
$$ |\theta(j_*,l'+1)| \leq \eps^{1-\frac{\log 2}{10}}$$
so $(j_*, l'+1)$ is also weakly black.  By Claim (ii), this implies that $(j_*-1, l'+1)$ is weakly black.  Iterating this argument we see that $(j_*-1, l'')$ is weakly black for all $l' \leq l'' \leq l_*$, hence $(j_*-1,l_*)$ is weakly black and we can obtain a contradiction as before.  This concludes the treatment of Case 3 of Claim (*).

We have now verified Claim (*) in all cases.  From this claim and the construction $(j_*,l_*)$ from $(j,l)$, we now see that $(j,l)$ must lie in $\Delta_*$; indeed, if $(j,l_*)$ was outside of $\Delta_*$ then one of the (necessarily black) points between $(j_*,l_*)$ and $(j,l_*)$ would violate Case 1 of Claim (*), and similarly if $(j,l_*)$ was in $\Delta_*$ but $(j,l)$ was outside $\Delta_*$ then one of the (necessarily black points) between $(j,l_*)$ and $(j,l)$ would again violate Case 1 of Claim (*); see Figure \ref{fig:cases}.  Furthermore, for any $(j',l') \in \Delta_*$, that $l_*(j',l') = l_*$ and $j_*(j',l') = j_*$.  In other words, we have
$$ \Delta_* = \{ (j',l') \in B: l_*(j',l') = l_*; j_*(j',l') = j_* \},$$
and so the triangles $\Delta_*$ form a partition of $B$.  By the preceding arguments we see that these triangles lie in $[\frac{n}{2} - \frac{1}{10} \log \frac{1}{\eps}] \times \Z$ and are separated from each other by at least $\frac{1}{10} \log \frac{1}{\eps}$.  This proves the lemma.
\end{proof}

\begin{remark} One can say a little bit more about the structure of the black set $B$; for instance, from Euler's theorem we see that $B$ is periodic with respect to the vertical shift $(0, 2 \times 3^{n-1})$ (cf. Lemma \ref{recursive}), and one could use Baker's theorem \cite{baker} that (among other things) establishes a Diophantine property of $\frac{\log 3}{\log 2}$ in order to obtain some further control on $B$.  However, we will not exploit any further structure of the black set in our arguments beyond what is provided by Lemma \ref{black}.
\end{remark}

\subsection{Formulation in terms of holding time}

We now return to the probabilistic portion of the proof of Proposition \ref{jeo-prop}.  Currently we have a finite sequence $\b_1,\dots,\b_{\lfloor n/2\rfloor}$ of random variables that are iid copies of the sum $\a_1+\a_2$ of two independent copies $\a_1,\a_2$ of $\Geom(2)$.  We may extend this sequence to an infinite sequence $\b_1,\b_2,\b_3,\dots$ of iid copies of $\a_1+\a_2$.  Recalling from definition that $W$ is a subset of $[n/2] \times \Z$, the point $(j,\b_{[1,j]})$ can only lie in $W$ when $j \in [n/2]$. Thus the left-hand side of \eqref{jeo} can then be written as
$$
\E \exp( - \eps^3 \# \{ j \in \N+1: \b_j = 3, (j,\b_{[1,j]}) \in W \} ).$$

We now describe the random set $\{ (j,\b_{[1,j]}): j \in \N+1, \b_j = 3\}$ as\footnote{We are indebted to Marek Biskup for this suggestion.} a two-dimensional renewal process (a special case of a \emph{renewal-reward process}).  Since the events $\b_j=3$ are independent and each occur with probability
\begin{equation}\label{pj3}
 \P( \b_j = 3 ) = \P( \Pascal = 3 ) = \frac{1}{4} > 0,
\end{equation}
we see that almost surely one has $\b_j=3$ for at least one $j \in \N$.  Define the \emph{two-dimensional holding time} $\Hold \in (\N+1) \times (\N+2)$ to be the random shift $(\j,\b_{[1,\j]})$, where $\j$ is the least positive integer for which $\b_\j =3$; this random variable is almost surely well defined.  Note from \eqref{pj3} that the first component $\j$ of $\Hold$ has the distribution $\j \equiv \Geom(4)$.
A little thought then reveals that the random set
\begin{equation}\label{bas}
 \{ (j,\b_{[1,j]}): j \in \N+1, \b_j = 3\}
\end{equation}
has the same distribution as the random set
\begin{equation}\label{vas}
 \{ \v_1, \v_{[1,2]}, \v_{[1,3]}, \dots \},
\end{equation}
where $\v_1, \v_2, \dots$ are iid copies of $\Hold$, and we extend the summation notation \eqref{sum} to the $\v_j$, thus for instance $\v_{[1,k]} \coloneqq \v_1 + \dots + \v_k$.  In particular, we have
$$
 \# \{ j \in \N+1: \b_j = 3, (j,\b_{[1,j]}) \in W \} \equiv \# \{ k \in \N+1: \v_{[1,k]} \in W \},$$
and so we can write the left-hand side of \eqref{jeo} as
\begin{equation}\label{jeo-2}
 \E \prod_{k \in \N+1} \exp( - \eps^3 1_W(\v_{[1,k]}) );
\end{equation}
note that all but finitely many of the terms in this product are equal to $1$.

We now pause our analysis of \eqref{jeo}, \eqref{jeo-2} to record some basic properties about the distribution of $\Hold$.

\begin{lemma}[Basic properties of holding time]\label{Exptail}  The random variable $\Hold$ has exponential tail (in the sense of \eqref{pvl}), is not supported in any coset of any proper subgroup of $\Z^2$, and has mean $(4,16)$.  In particular, the conclusion of Lemma \ref{chern} holds for $\Hold$ with $\vec \mu = (4,16)$.
\end{lemma}

\begin{proof}  From the definition of $\Hold$ and \eqref{pj3}, we see that $\Hold$ is equal to $(1,3)$ with probability $1/4$, and on the remaining event of probability $3/4$, it has the distribution of $(1,\Pascal') + \Hold'$, where $\Pascal'$ is a copy of $\Pascal$ that is conditioned to the event $\Pascal \neq 3$, so that
\begin{equation}\label{beb}
 \P( \Pascal' = b ) = \frac{4}{3} \frac{b-1}{2^b}
\end{equation}
for $b \in \N+2 \backslash \{3\}$, and $\Hold'$ is a copy of $\Hold$ that is independent of $\Pascal'$.   Thus $\Hold$ has the distribution of $(1,3) + (1,\b'_1) + \dots + (1,\b'_{\j-1})$, where $\b'_1,\b'_2,\dots$ are iid copies of $\Pascal'$ and $\j \equiv \Geom(4)$ is independent of the $\b'_j$.  In particular, for any $k = (k_1,k_2) \in \R^2$, one has from monotone convergence that
\begin{equation}\label{vk}
\E \exp( \Hold \cdot k ) = \sum_{j \in \N} \frac{1}{4} \left(\frac{3}{4}\right)^{j-1} \exp\left( (1,3) \cdot k\right) \left(\E \exp( (1,\Pascal') \cdot k) \right)^j.
\end{equation}
From \eqref{beb} and dominated convergence, we have $\E \exp( (1,\Pascal') \cdot k ) < \frac{4}{3}$ for $k$ sufficiently close to $0$, which by \eqref{vk} implies that $\E \exp( \Hold \cdot k ) < \infty$ for $k$ sufficiently close to zero.  This gives the exponential tail property by Markov's inequality.

Since $\Hold$ attains the value $(1,3)+(1,b)$ for any $b \in \N+2 \backslash \{3\}$ with positive probability, as well as attaining $(1,3)$ with positive probability, we see that the support of $\Hold$ is not supported in any coset of any proper subgroup of $\Z^2$.  Finally, from the description of $\Hold$ at the start of this proof we have
$$ \E \Hold = \frac{1}{4} (1,3) + \frac{3}{4} \left((1,\E \Pascal') + \E \Hold\right);$$
also, from the definition of $\Pascal'$ we have
$$ \E \Pascal = \frac{1}{4} 3 + \frac{3}{4} \E \Pascal'.$$
We conclude that
$$ \E \Hold = (1,\E \Pascal) + \frac{3}{4} \E \Hold;$$
since $\E \Pascal = 2 \E \Geom(2) =4$, we thus have $\E \Hold = (4,16)$ as required.
\end{proof}

The following lemma allows us to control the distribution of first passage locations of renewal processes with holding times $\equiv \Hold$, which will be important for us as it lets us understand how such renewal processes exit a given triangle $\Delta$:

\begin{lemma}[Distribution of first passage location]\label{stop} Let $\v_1,\v_2,\dots$ be iid copies of $\Hold$, and write $\v_k = (\j_k,\l_k)$.  Let $s \in \N$, and define the first passage time $\k$ to be the least positive integer such that $\l_{[1,k]} > s$.  Then for any $j,l \in \N$ with $l > s$, one has
$$
\P( \v_{[1,\k]} = (j,l) ) \ll \frac{e^{-c(l-s)}}{(1+s)^{1/2}} G_{1+s}\left( c\left(j - \frac{s}{4}\right) \right),$$
where $G_{1+s}(x) = \exp(-\frac{|x|^2}{1+s}) + \exp(-|x|)$ was the function defined in \eqref{gaussian-def}.
\end{lemma}

Informally, this lemma asserts that as a rough first approximation one has
\begin{equation}\label{rough}
 \v_{[1,\k]} \approx \Unif\left( \left\{ (j,l): j = \frac{s}{4} + O( (1+s)^{1/2}); s < l \leq s + O(1) \right\} \right).
\end{equation}

\begin{proof}  Note that by construction of $\k$ one has $\l_{[1,\k]} - \l_\k \leq s$, so that $\l_\k \geq \l_{[1,\k]}-s$.  From the union bound, we therefore have
$$ \P( \v_{[1,\k]} = (j,l) ) \leq \sum_{k \in \N+1} \P( (\v_{[1,k]} = (j,l)) \wedge (\l_k \geq l - s) );$$
since $\v_k$ has the exponential tail and is independent of $\v_1,\dots,\v_{k-1}$, we thus have
$$ \P( \v_{[1,\k]} = (j,l) ) \ll \sum_{k \in \N+1} \sum_{l_k \geq l-s} \sum_{j_k \in \N+1} e^{- c(j_k+l_k)} \P( \v_{[1,k-1]} = (j-j_k,l-l_k) ).$$
Writing $l_k = l - s + l'_k$, we then have
\begin{align*}
\P( \v_{[1,\k]} = (j,l) ) &\ll e^{-c(l-s)} \sum_{k \in \N+1} \sum_{l'_k \in \N} \sum_{j_k \in \N+1} \\
&\quad\quad\quad e^{-c(j_k+l'_k)} \P( \v_{[1,k-1]} = (j-j_k,s-l'_k) ).
\end{align*}
We can restrict to the region $l'_k \leq s$, since the summand vanishes otherwise.  It now suffices to show that
\begin{equation}\label{head}
\begin{split}
&\sum_{k \in \N+1} \sum_{0 \leq l_k' \leq s} \sum_{j_k \in \N+1} e^{- c(j_k+l'_k)} \P\left( \v_{[1,k-1]} = (j-j_k,s-l'_k) \right)\\
&\quad \ll (1+s)^{-1/2} G_{1+s}\left( c(j - \frac{s}{4}) \right).
\end{split}
\end{equation}
This is in turn implied by
\begin{equation}\label{turn}
\begin{split}
&\sum_{k \in \N+1} \sum_{0 \leq l_k' \leq s} e^{-cl'_k} \P( \v_{[1,k-1]} = (j',s-l'_k) ) \\
&\quad \ll (1+s)^{-1/2} G_{1+s}\left( c(j' - \frac{s}{4}) \right)
\end{split}
\end{equation}
for all $j' \in \Z$, since \eqref{head} then follows by replacing $j'$ by $j - j_k$, multiplying by $\exp(-cj_k)$, and summing in $j_k$ (and adjusting the constants $c$ appropriately).  In a similar vein, it suffices to show that
$$\sum_{k \in \N+1} \P( \v_{[1,k-1]} = (j',s') ) \ll (1+s')^{-1/2} G_{1+s'}\left( c(j' - \frac{s'}{4}) \right)$$
for all $s' \in \N$, since \eqref{turn} follows after setting $s' = s - l'_k$, multiplying by $\exp(-cl'_k)$, and summing in $l'_k$ (splitting into the regions $l'_k \leq s/2$ and $l'_k > s/2$ if desired to simplify the calculations).

From Lemma \ref{Exptail} and Lemma \ref{chern} one has
$$ \P( \v_{[1,k-1]} = (j',s') ) \ll k^{-1} G_{k-1}\left( c ((j',s') - (k-1)(4,16))\right),$$
and the claim now follows from summing in $k$ and a routine calculation (splitting for instance into the regions $16 (k-1) \in [s'/2,2s']$, $16(k-1) < s'/2$, and $16(k-1)>2s'$).
\end{proof}

\subsection{Recursively controlling a maximal expression}

We return to the study of the left-hand side of \eqref{jeo}, which we have expressed as \eqref{jeo-2}.
For any $(j,l) \in \N+1 \times \Z$, let $Q(j,l)$ denote the quantity
\begin{equation}\label{qjl-def}
 Q(j,l) \coloneqq \E \prod_{k \in \N} \exp( - \eps^3 1_W((j,l) + \v_{[1,k]}) )
\end{equation}
then we have the recursive formula
\begin{equation}\label{recurse}
 Q(j,l) = \exp( - \eps^3 1_W(j,l) ) \E Q((j,l) + \Hold).
\end{equation}
Observe that for each $(j,l) \in \N+1 \times \Z$, we have the conditional expectation
$$ \E\left( \prod_{k \in \N+1} \exp( - \eps^3 1_W(\v_{[1,k]}) ) | \v_1 = (j,l) \right) = Q(j,l)$$
since after conditioning on $\v_1 = (j,l)$ then the $\v_{[1,k]}$ have the same distribution as $(j,l) + \v'_{[1,k-1]}$ where $\v'_1,\v'_2,\dots$ is another sequence of iid copies of $\Hold$.  Since $\v_1$ has the distribution of $\Hold$, we conclude from the law of total probability that
$$ \E \prod_{k \in \N+1} \exp( - \eps^3 1_W(\v_{[1,k]}) ) = \E Q(\Hold).$$
From \eqref{jeo-2} we thus see that we can rewrite the desired estimate \eqref{jeo} as
\begin{equation}\label{jeo-3}
 \E Q(\Hold) \ll_A n^{-A}.
\end{equation}
One can think of $Q(j,l)$ as a quantity controlling how often one encounters white points when one walks along a two-dimensional renewal process  $(j,l), (j,l) + \v_1, (j,l)+\v_{[1,2]}, \dots$ starting at $(j,l)$ with holding times given by iid copies of $\Hold$.  The smaller this quantity is, the more white points one is likely to encounter.  The main difficulty is thus to ensure that this renewal process is usually not trapped within the black triangles $\Delta$ from Lemma \ref{black}; as it turns out (and as may be evident from an inspection of Figure \ref{fig:triangles}), the large triangles will be the most troublesome to handle (as they are so large compared to the narrow band of white points surrounding them that are provided by Lemma \ref{black}).

Suppose that we can prove a bound of the form
\begin{equation}\label{qla}
Q(j,l) \ll_A \max(\lfloor n/2\rfloor -j,1)^{-A}
\end{equation}
for all $(j,l) \in (\N+1) \times \Z$; this is trivial for $j \geq n/2$ but becomes increasingly non-trivial for smaller values of $j$.  Then
$$ Q(\Hold) \ll_A \max(\lfloor n/2\rfloor -\j,1)^{-A} \ll_A n^{-A} \j^A$$
where $\j \equiv \Geom(4)$ is the first component of $\Hold$.  As $\Geom(4)$ has exponential tail, we conclude \eqref{jeo-3} and hence Proposition \ref{jeo-prop}, which then implies Propositions \ref{key}, \ref{f-decay} and Theorem \ref{main}.

It remains to prove \eqref{qla}.  Roughly speaking, we will accomplish this by a downwards induction on $j$, or more precisely, by an upwards induction on a quantity $m$, which is morally equivalent to $\lfloor n/2\rfloor - j$.  To make this more precise, it is convenient to introduce the quantities $Q_m$ for any $m \in [n/2]$ by the formula
\begin{equation}\label{qm-def}
 Q_m \coloneqq \sup_{(j,l) \in (\N+1) \times \Z: j \geq \lfloor n/2\rfloor - m} \max(\lfloor n/2\rfloor -j,1)^A Q(j,l).
\end{equation}
Clearly we have
\begin{equation}\label{qma}
Q_m \leq m^A,
\end{equation}
since $Q(j,l) \leq 1$ for all $j,l$; this bound can be thought of as supplying the ``base case'' for our induction).  We trivially have $Q_m \geq Q_{m-1}$ for any $1 \leq m \leq n/2$.  We will shortly establish the opposite inequality:

\begin{proposition}[Monotonicity]\label{mono-prop}  We have
\begin{equation}\label{mono}
Q_m \leq Q_{m-1}
\end{equation}
whenever $C_{A,\eps} \leq m \leq n/2$ for some sufficiently large $C_{A,\eps}$ depending on $A,\eps$.
\end{proposition}

Assuming Proposition \ref{mono-prop}, we conclude from \eqref{qma} and a (forwards) induction on $m$ that $Q_m \leq C_{A,\eps}^A \ll_A 1$ for all $1 \leq m \leq n/2$, which gives \eqref{qla}.  This in turn implies Proposition \ref{jeo-prop}, and hence Proposition \ref{key}, Proposition \ref{f-decay}, and Theorem \ref{main}.

It remains to establish Proposition \eqref{mono-prop}.  Let $C_{A,\eps} \leq m \leq n/2$ for some sufficiently large $C_{A,\eps}$.  It suffices to show that
\begin{equation}\label{haha}
 Q(j,l) \leq m^{-A} Q_{m-1}
\end{equation}
whenever $j = \lfloor n/2 \rfloor - m$ and $l \in \Z$.  Note from \eqref{qm-def} that we immediately obtain $Q(j,l) \leq m^{-A} Q_m$, but to be able to use $Q_{m-1}$ instead of $Q_m$ we will apply \eqref{recurse} at least once, in order to estimate $Q(j,l)$ in terms of other values $Q(j',l')$ of $Q$ with $j' > j$.  This causes a degradation in the $m^{-A}$ term, even when $m$ is large; to overcome this loss we need to ensure that (with high probability) the two-dimensional renewal process visits a sufficient number of white points before we use $Q_{m-1}$ to bound the resulting expression.  This is of course consistent with the interpretation of \eqref{jeo} as an assertion that the renewal process encounters plenty of white points.

We divide the proof of \eqref{haha} into three cases.  Let ${\mathcal T}$ be the family of triangles from Lemma \ref{black}.

\textbf{Case 1: $(j,l) \in W$.}  This is the easiest case, as one can immediately get a gain from the white point $(j,l)$.  From \eqref{recurse} we have
$$ Q(j,l) = \exp( - \eps^3 ) \E Q((j,l) + \Hold).$$
For any $(j',l') \in (\N+1) \times \Z$, we have from \eqref{qm-def} (applied with $m$ replaced by $m-1$) that
$$ Q((j,l) + (j',l')) \leq \max( \lfloor n/2\rfloor -j - j', 1)^{-A} Q_{m-1} =
\max( m - j', 1)^{-A} Q_{m-1}$$
since $j+j' \geq j+1 = \lfloor n/2 \rfloor - (m-1)$.  Replacing $(j',l')$ by $\Hold$ (so that $j'$ has the distribution of $\Geom(4)$) and taking expectations, we conclude that
$$ Q(j,l) \leq \exp( - \eps^3 ) Q_{m-1} \E \max( m - \Geom(4), 1)^{-A}.$$
We can bound
\begin{equation}\label{magi}
 \max(m-r,1)^{-1} \leq m^{-1} \exp\left( O\left( \frac{r\log m}{m}\right ) \right)
\end{equation}
for any $r \in \N+1$; indeed this bound is trivial for $r \geq m$, and for $r < m$ one can use the concave nature of $x \mapsto \log(1-x)$ for $0 < x < 1$ to conclude that
$$ \frac{\log\left(1-\frac{r}{m}\right)}{r/m} \geq \frac{\log \left(1 - \frac{m-1}{m}\right)}{(m-1)/m}$$
which rearranges to give the stated bound.   Replacing $r$ by $\Geom(4)$ and raising to the $A^{\mathrm{th}}$ power, we obtain
$$ Q(j,l) \leq \exp( - \eps^3 ) m^{-A} Q_{m-1} \E \exp\left( O\left( \frac{A\log m}{m} \Geom(4) \right) \right).$$
For $m$ large enough depending on $A,\eps$, we then have
\begin{equation}\label{qjl}
 Q(j,l) \leq \exp( - \eps^3/2 ) m^{-A} Q_{m-1}
\end{equation}
which gives \eqref{haha} in this case (with some room to spare).

\textbf{Case 2: $(j,l) \in \Delta$ for some triangle $\Delta \in {\mathcal T}$, and $l \geq l_\Delta - \frac{m}{\log^2 m}$.}   This case is slightly harder than the preceding one, as one has to walk randomly through the triangle $\Delta$ before one has a good chance to encounter a white point, but because this portion of the walk is relatively short, the degradation of the weight $m^{-A}$ during this portion will be negligible.

We turn to the details. Set $s \coloneqq l_\Delta - l$, thus $0 \leq s \leq \frac{m}{\log^2 m}$.  Let $\v_1,\v_2,\dots$ be iid copies of $\Hold$, write $\v_k = (\j_k, \l_k)$ for each $k$ with the usual summation notations \eqref{sum}, and define the first passage time $\k \in \N+1$ to be the least positive integer such that
\begin{equation}\label{lam}
 \l_{[1,\k]} > s.
\end{equation}
This is a finite random variable since the $\l_k$ are all positive integers.  Heuristially, $\k$ represents the time in which the sequence first exits the triangle $\Delta$, assuming that this exit occurs on the top edge of the triangle.  It is in principle possible for the sequence to instead exit $\Delta$ through the hypotenuse of the triangle, in which case $\k$ will be somewhat larger than the first exit time; however, as we shall see below, the Chernoff bound in Lemma \ref{stop} can be used to show that the former scenario will occur with probability $\gg 1$, which will be sufficient for the purposes of establishing \eqref{haha} in this case.

By iterating \eqref{recurse} appropriately (or using \eqref{qjl-def}), we have the identity
\begin{equation}\label{litera}
 Q(j,l) = \E \left[ \exp\left( - \eps^3 \sum_{i=0}^{\k-1} 1_W((j,l) + \v_{[1,i]}) \right) Q((j,l) + \v_{[1,\k]}) \right]
\end{equation}
and hence by \eqref{qm-def}
$$ Q(j,l) \leq  Q_{m-1} \E \left[ \exp\left( - \frac{\eps^3}{2} 1_W((j,l) + \v_{[1,\k]}) \right) \max(m - \j_{[1,\k]},1)^{-A}\right]$$
which by \eqref{magi} gives
$$ Q(j,l) \leq m^{-A} Q_{m-1} \E \exp\left( - \frac{\eps^3}{2} 1_W((j,l) + \v_{[1,\k]})\right) \exp\left( O\left( \frac{A \log m}{m} \j_{[1,\k]} \right)\right).$$
To prove \eqref{haha} in this case, it thus suffices to show that
\begin{equation}\label{loko}
\E \exp\left( - \frac{\eps^3}{2} 1_W((j,l) + \v_{[1,\k]}) \right) \exp\left( O\left( \frac{A \log m}{m} \j_{[1,\k]} \right)\right) \leq 1.
\end{equation}
Since $\exp(-\eps^3/2) \leq 1 - \eps^3/4$, we can upper bound the left-hand side by
\begin{equation}\label{loko-2}
 \E \exp\left( O\left( \frac{A\log m}{m} \j_{[1,\k]} \right) \right) - \frac{\eps^3}{4} \P( (j,l) + \v_{[1,\k]} \in W ).
\end{equation}

We begin by controlling the first term on the right-hand side of \eqref{loko-2}.
By definition, the first passage location $(j,l) + \v_{[1,\k]}$ takes values in the region $\{ (j',l') \in \Z^2: j' > j, l' > l_\Delta \}$.  From Lemma \ref{stop} we have
\begin{equation}\label{ston}
\P( (j,l) + \v_{[1,\k]} = (j',l') ) \ll \frac{e^{-c(l'-l_\Delta)}}{(1+s)^{1/2} } G_{1+s}\left(c(j'-j - \frac{s}{4}) \right).
\end{equation}
Summing in $l'$, we conclude that
$$\P( \j_{[1,\k]} = j'-j ) \ll (1+s)^{-1/2} G_{1+s}\left( c(j'-j - \frac{s}{4}) \right)$$
for any $j'$; informally, $\j_{[1,\k]}$ is behaving like a Gaussian random variable centred at $s/4$ with standard deviation $\asymp (1+s)^{1/2}$.  In particular, because of the hypothesis $s \leq \frac{m}{\log^2 m}$, we have
$$\P( \j_{[1,\k]} = r ) \ll  \exp( - |r| )$$
when $r > \frac{m}{\log^2 m}$ (say).  With our hypotheses $s \leq \frac{m}{\log^2 m}$ and $m \geq C_{A,\eps}$, the quantity $\frac{A \log m}{m}$ is much smaller than $1$, and by using the above bound to control the contribution when $\j_{[1,\k]} > \frac{m}{\log^2 m}$ we have
\begin{equation}\label{loko-3}
\begin{split}
\E \exp\left( O\left( \frac{A\log m}{m} \j_{[1,\k]} \right) \right)
&\leq  \E \exp\left( O\left( \frac{A\log m}{m} \frac{m}{\log^2 m} \right) \right) + O\left( \exp\left( - c \frac{m}{\log^2 m} \right) \right)\\
= 1 + O\left( \frac{A}{\log m} \right).
\end{split}
\end{equation}
Now we turn attention to the second term on the right-hand side of \eqref{loko-2}.
Using \eqref{ston} to handle all points $(j',l')$ outside the region $l' = l_\Delta+O(1)$ and $j' = j + \frac{s}{4} + O( (1+s)^{1/2} )$, we have
\begin{equation}\label{hash}
 \P\left( (j,l) + \v_{[1,\k]} = \left(j+\frac{s}{4} + O((1+s)^{1/2}),l_\Delta + O(1)\right) \right) \gg 1
\end{equation}
for a suitable choice of implied constants in the $O$-notation that is independent of $\eps$ (cf. \eqref{rough}).
On the other hand, since $(j,l) \in \Delta$ and $s = l_\Delta - l$, we have from \eqref{ts} that
$$ 0 \leq (j-j_\Delta) \log 9 \leq s_\Delta - s \log 2$$
and thus (since $0 < \frac{1}{4} \log 9 < \log 2$) one has
$$ -O(1) \leq (j'-j_\Delta ) \log 9 \leq s_\Delta + O(1)$$
whenever $j' = j + \frac{s}{4} + O((1+s)^{1/2})$, with the implied constants independent of $\eps$.  We conclude that with probability $\gg 1$, the first passage location $(j,l) + \v_{[1,\k]}$ lies outside of $\Delta$, but at a distance $O(1)$ from $\Delta$, hence is white by Lemma \ref{black}.  We conclude that
\begin{equation}\label{loko-4}
 \P( (j,l) + \v_{[1,\k]} \in W ) \gg 1
\end{equation}
and \eqref{haha} (and hence \eqref{loko}) now follows from \eqref{loko-2}, \eqref{loko-3}, \eqref{loko-4} since $m \geq C_{A,\eps}$.

\textbf{Case 3: $(j,l) \in \Delta$ for some triangle $\Delta \in {\mathcal T}$, and $l < l_\Delta - \frac{m}{\log^2 m}$.}  This is the most difficult case, as one has to walk so far before exiting $\Delta$ that one needs to encounter multiple white points, not just a single white point, in order to counteract the degradation of the weight $m^{-A}$.  Fortunately, the number of white points one needs to encounter is $O_{A,\eps}(1)$, and we will be able to locate such a number of white points on average for $m$ large enough.

We will need a large constant $P$ (much larger than $A$ or $1/\eps$, but much smaller than $m$) depending on $A,\eps$ to be chosen later; the implied constants in the asymptotic notation below will not depend on $P$ unless otherwise specified.  As before, we set $s \coloneqq l_\Delta - l$, so now $s > \frac{m}{\log^2 m}$.  From \eqref{ts} we have
$$ (j-j_\Delta) \log 9 + s \log 2 \leq s_\Delta$$
while from Lemma \ref{black} one has $j_\Delta + \frac{s_\Delta}{\log 9} \leq \lfloor \frac{n}{2} \rfloor \leq j+m$, hence
\begin{equation}\label{bound}
 s \leq \frac{\log 9}{\log 2} m.
\end{equation}
We again let $\v_1,\v_2,\dots$ be iid copies of $\Hold$, write $\v_k = (\j_k, \l_k)$ for each $k$, and define the first passage time $\k \in \N+1$ to be the least positive integer such that \eqref{lam} holds.  From \eqref{litera} we have
$$
 Q(j,l) \leq \E Q((j,l) + \v_{[1,\k]}).
$$
Applying \eqref{recurse} we then have
\begin{equation}\label{qjl2}
 Q(j,l) \leq \E \exp\left( - \eps^3 \sum_{p=0}^{P-1} 1_W((j,l) + \v_{[1,\k+p]} ) \right) Q((j,l) + \v_{[1,\k+P]}).
\end{equation}
Applying \eqref{qm-def} to $Q((j,l) + \v_{[1,\k+P]}) = Q(j+\j_{[1,\k+P]}, l+\l_{[1,\k+P]})$, we have
$$ \max( \lfloor n/2 \rfloor - j - \j_{[1,\k+P]}, 1)^A Q((j,l) + \v_{[1,\k+P]}) \leq Q_{m-1}$$
(since $j + \j_{[1,\k+P]} \geq j+1 \geq \lfloor n/2 \rfloor - (m-1)$).  We can rearrange this inequality as
$$ Q((j,l) + \v_{[1,\k+P]}) \leq m^{-A} Q_{m-1} \max\left( 1 - \frac{\j_{[1,\k+P]}}{m}, \frac{1}{m}\right)^{-A};$$
inserting this back into \eqref{qjl2}, we conclude that
$$
 Q(j,l) \leq m^{-A} Q_{m-1} \E \exp\left( - \eps^3 \sum_{p=0}^{P-1} 1_W((j,l) + \v_{[1,\k+p]}) \right) \max\left( 1 - \frac{\j_{[1,\k+P]}}{m}, \frac{1}{m}\right)^{-A}.
$$
Thus, to establish \eqref{haha} in this case, it suffices to show that
\begin{equation}\label{pelota}
 \E \exp\left( - \eps^3 \sum_{p=0}^{P-1} 1_W((j,l) + \v_{[1,\k+p]} )\right) \max\left( 1 - \frac{\j_{[1,\k+P]}}{m}, \frac{1}{m}\right)^{-A} \leq 1.
\end{equation}
Let us first consider the event that $\j_{[1,\k+P]} \geq 0.9 m$.  From Lemma \ref{stop} and the bound \eqref{bound}, we have
$$ \P( \j_{[1,\k]} \geq 0.8 m ) \ll \exp( -c m)$$
(noting that $0.8 > \frac{1}{4} \frac{\log 9}{\log 2}$) while from Lemma \ref{chern} (recalling that the $\j_k$ are iid copies of $\Geom(4)$) we have
$$ \P( \j_{[\k+1,\k+P]} \geq 0.1m ) \ll_P \exp( -c m)$$
and thus by the triangle inequality
$$ \P( \j_{[1,\k+P]} \geq 0.9 m ) \ll_P \exp( -c m).$$
Thus the contribution of this case to \eqref{pelota} is $O_{P,A}(m^A \exp(-cm)) = O_{P,A}(\exp(-cm/2))$.  If instead we have $\j_{[1,\k+P]} < 0.9 m$, then
$$ \max\left( 1 - \frac{\j_{[1,\k+P]}}{m}, \frac{1}{m}\right)^{-A} \leq 10^A.$$
Since $m$ is large compared to $A,P$, to show \eqref{pelota} it thus suffices to show that
\begin{equation}\label{cheap-0}
 \E \exp\left( - \eps^3 \sum_{p=0}^{P-1} 1_W((j,l) + \v_{[1,\k+p]} )\right) \leq 10^{-A-1}.
\end{equation}
Since the left-hand side of \eqref{cheap-0} is at most
$$
\P\left( \sum_{p=0}^{P-1} 1_W((j,l) + \v_{[1,\k+p]} ) \leq \frac{10 A}{\eps^3} \right) + \exp(-10A),$$
it will suffice to establish the bound
\begin{equation}\label{cheap}
\P\left( \sum_{p=0}^{P-1} 1_W((j,l) + \v_{[1,\k+p]} ) \leq \frac{10 A}{\eps^3} \right) \leq 10^{-A-2}
\end{equation}
(say).

Roughly speaking, the estimate \eqref{cheap} asserts that once one exits the large triangle $\Delta$ then one should almost always encounter at least $10A/\eps^3$ white points by a certain time $P = O_{A,\eps}(1)$.

To prove \eqref{cheap}, we introduce another random statistic that measures the number of triangles that one encounters on an infinite two-dimensional renewal process $(j',l'), (j',l') + \v_1, (j',l') + \v_{[1,2]},\dots$, where $(j',l') \in (\N+1) \times \Z$ and $\v_1,\v_2,\dots$ are iid copies of $\Hold$.  (We will eventually set $(j',l') \coloneqq (j,l) + \v_{[1,\k]}$, so that the above renewal process is identical in distribution to $(j,l) + \v_{[1,\k]}$, $(j,l) + \v_{[1,\k+1]}$, $(j,l) + \v_{[1,\k+2]}, \dots$.)

Given an initial point $(j',l') \in (\N+1) \times \Z$, we recursively introduce the \emph{stopping times} $\t_1 = \t_1(j',l'),\dots,\t_\r = \t_{\r(j',l')}(j,l)$ by defining $\t_1$ to be the first natural number (if it exists) for which $(j',l') + \v_{[1,\t_1]}$ lies in a triangle $\mathbf{\Delta}_1 \in {\mathcal T}$, then for each $i>1$, defining $\t_i$ to be the first natural number (if it exists) with $l' + \l_{[1,\t_i]} > l_{\mathbf{\Delta}_{i-1}}$ and $(j',l') + \v_{[1,\t_i]}$ lies in a triangle $\mathbf{\Delta}_i \in {\mathcal T}$.  We set $\r = \r(j',l')$ to be the number of stopping times that can be constructed in this fashion (thus, there are no natural numbers $k$ with $l + \l_{[1,k]} > l_{\mathbf{\Delta}_\r}$ and $(j',l') + \v_{[1,k]}$ black).  Note that $\r$ is finite since the process $(j',l')+\v_{[1,k]}$ eventually exits the strip $[n/2] \times \Z$ when $k$ is large enough, at which point it no longer encounters any black triangles.

The key estimate relating $\r$ with the expression in \eqref{cheap} is then

\begin{lemma}[Many triangles usually implies many white points]\label{rip}  Let $\v_1,\v_2,\dots$ be iid copies of $\Hold$.  Then for any $(j',l') \in (\N+1) \times \Z$ and any positive integer $R$, we have
\begin{equation}\label{que}
 \E 1_{R \leq \r} \exp\left( - \sum_{p=1}^{\t_{\min(\r,R)}} 1_W((j',l') + \v_{[1,p]}) + \eps R \right) \leq \exp(\eps),
\end{equation}
where $0 < \eps < 1/100$ is the sufficiently small absolute constant that has been in use throughout this section.
\end{lemma}

Informally the estimate \eqref{que} asserts that when $\r$ is large (so that the renewal process $(j',l'), (j',l') + \v_1, (j',l') + \v_{[1,2]},\dots$ passes through many different triangles), then the quantity $\sum_{p=1}^{\t_{\min(\r,R)}} 1_W((j',l') + \v_{[1,p]})$ is usually also large, implying that the same renewal process also visits many white points.  This is basically due to the separation between triangles that is given by Lemma \ref{black}.

\begin{proof}  Denote the quantity on the left-hand side of \eqref{que} by $Z( (j',l'), R )$.
We induct on $R$.  The case $R=1$ is trivial, so suppose $R \geq 2$ and that we have already established that
\begin{equation}\label{z-bound}
 Z((j'',l''), R-1) \leq \exp(\eps)
\end{equation}
for all $(j'',l'') \in (\N+1) \times \Z$. If $\r=0$, the expression vanishes.
 Suppose instead that $\r \neq 0$, so that the first stopping time $\t_1$ and triangle $\mathbf{\Delta}_1$ exists.  Let $\k_1$ be the first natural number for which $l' + \l_{[1,\k_1]} > l_{\Delta_1}$; then $\k_1$ is well-defined (since we have an infinite number of $\l_k$, all of which are at least $2$) and $\k_1 > \t_1$.  The conditional expectation of $\exp( - \sum_{p=1}^{\t_{\min(\r,R)}} 1_W((j',l') + \v_{[1,p]} ) + \eps \min(\r,R))$ relative to the random variables $\v_1,\dots,\v_{\k_1}$ is equal to
$$  \exp\left( -  \sum_{p=1}^{\k_1} 1_W((j',l') + \v_{[1,p]} ) + \eps \right) Z( 1_W((j',l') + \v_{[1,\k_1]}, R-1)$$
which we can upper bound using the inductive hypothesis \eqref{z-bound} as
$$  \exp\left( -  1_W((j',l') + \v_{[1,\k_1]} ) + 2\eps \right).$$
We thus obtain the inequality
$$
Z( (j',l'), R) \leq \exp(2\eps) \E 1_{\r \neq 0} \exp( - 1_W((j',l') + \v_{[1,\k_1]} ) )
$$
so to close the induction it suffices to show that
$$ \E 1_{\r \neq 0} \exp( - 1_W((j',l') + \v_{[1,\k_1]} ) ) \leq \exp(-\eps) \P( \r \neq 0).$$
Since the left-hand side is equal to
$$ \P( \r \neq 0 ) - (1-1/e) \P( (\r \neq 0) \wedge ((j',l') + \v_{[1,\k_1]} \in W) ) $$
and $\eps>0$ is a sufficiently small absolute constant, it will thus suffice to establish the bound
$$ \P( (\r \neq 0) \wedge ((j',l') + \v_{[1,\k_1]} \in W) ) \gg \P( \r \neq 0 ).$$
For each $p \in \N+1$, triangle $\Delta_1 \in {\mathcal T}$, and $(j'',l'') \in \Delta_1$, let $E_{p,\Delta_1,(j'',l'')}$ denote the event that $(j',l') + \v_{[1,p]} = (j'',l'')$, and $(j',l') + \v_{[1,p']} \in W$ for all $1 \leq p' < p$.  Observe that the event $\r \neq 0$ is the disjoint union of the events $E_{p,\Delta_1,(j'',l'')}$.  It therefore suffices to show that
\begin{equation}\label{econ}
 \P\left( E_{p,\Delta_1,(j'',l'')} \wedge ((j',l') + \v_{[1,\k_1]} \in W) \right) \gg \P( E_{p,\Delta_1,(j'',l'')} ).
\end{equation}
We may of course assume that the event $E_{p,\Delta_1,(j'',l'')}$ occurs with non-zero probability. Conditioning to this event, we see that $(j',l') + \v_{[1,\k_1]}$ has the same distribution as (the unconditioned random variable) $(j'',l'') + \v_{[1,\k'']}$, where the first passage time $\k''$ is the first natural number for which $l'' + \l_{[1,\k'']} > l_{\Delta_1}$.  By repeating the proof of \eqref{loko-4}, one has
$$
 \P( (j'',l'') + \v_{[1,\k'']} \in W | E_{p,\Delta_1,(j'',l'')} ) \gg 1
$$
giving \eqref{econ}.  This establishes the lemma.
\end{proof}

To use this bound we need to show that the renewal process $(j,l)+\v_[1,\k], (j,l) + \v_[1,\k+1], (j,l) + \v_{[1,\k+2]},\dots$ either passes through many white points, or through many triangles.  This will be established via a probabilistic upper bound on the size $s_\Delta$ of the triangles encountered.  The key lemma in this regard is

\begin{lemma}[Large triangles are rarely encountered shortly after a lengthy crossing]\label{77}  Let $(j,l)$ be an element of a black triangle $\Delta$ with $s \coloneqq l_\Delta - l$ obeying $s > \frac{m}{\log^2 m}$ (where we recall $m = \lfloor n/2\rfloor - j$), and let $\k$ be the first passage time associated to $s$ defined in Lemma \ref{stop}.  Let $p \in \N$ and $1 \leq s' \leq m^{0.4}$. Let $E_{p,s'}$ denote the event that $(j,l) + \v_{[1,\k+p]}$ lies in a triangle $\Delta' \in {\mathcal T}$ of size $s_{\Delta'} \geq s'$.  Then
$$ \P( E_{p,s'} ) \ll A^2 \frac{1+p}{s'} + \exp( - c A^2 (1+p) ).$$
\end{lemma}

As in the rest of this section, we stress that the implied constants in our asymptotic notation are uniform in $n$ and $\xi$.

\begin{proof} We can assume that
\begin{equation}\label{sbig}
s' \geq C A^2 (1+p)
\end{equation}
for a large constant $C$, since the claim is trivial otherwise.

 From Lemma \ref{stop} we have \eqref{ston} as before, so on summing in $j'$ we have
$$
\P( l + \l_{[1,k]} = l' ) \ll \exp( - c (l'-l_\Delta) ) $$
and thus
$$
\P( l + \l_{[1,k]} \geq l_\Delta + A^2 (1+p) ) \ll \exp( - c A^2 (1+p) ).$$
Similarly, from Lemma \ref{chern} one has
$$
\P( \l_{[\k+1,\k+p]} \geq A^2 (1+p) ) \ll \exp( - c A^2 (1+p) )$$
and thus
$$
\P( l + \l_{[1,\k+p]} \geq l_\Delta + 2A^2 (1+p) ) \ll \exp( - c A^2 (1+p) ).$$
In a similar spirit, from \eqref{ston} and summing in $l'$ one has
$$
\P( j + \j_{[1,\k]} = j' ) \ll s^{-1/2} G_{1+s}\left( c(j'-j - \frac{s}{4}) \right)$$
so in particular
$$
\P\left( \left|\j_{[1,\k]} - \frac{s}{4}\right| \geq s^{0.6} \right) \ll \exp( - c s^{0.2} ) \ll A^2 \frac{1+p}{s'}$$
from the upper bound on $s'$.  From Lemma \ref{chern} we also have
$$
\P( |\j_{[\k+1,\k+p]}| \geq s^{0.6} ) \ll \exp( - c s^{0.6} ) \ll A^2 \frac{1+p}{s'}$$
and hence
$$
\P\left( \left|\j_{[1,\k+p]} - \frac{s}{4}\right| \geq 2s^{0.6} \right) \ll A^2 \frac{1+p}{s'}$$
Thus, if $E'$ denotes the event that $l + \l_{[1,\k+p]} \geq l_\Delta + 2A^2 (1+p)$ or $|\j_{[1,\k+p]} - \frac{s}{4}| \geq 2s^{0.6}$, then
\begin{equation}\label{pang}
\P( E') \ll A^2 \frac{1+p}{s'} + \exp( - c A^2 (1+p) ).
\end{equation}
We will devote the rest of the proof to establishing the complementary estimate
\begin{equation}\label{pang-2}
\P(  E_{p,s'} \wedge \bar{E'}) \ll A^2 \frac{1+p}{s'}
\end{equation}
which together with \eqref{pang} implies the lemma.

Suppose now that we are outside the event $E'$, and that  $(j,l) + \v_{[1,\k+p]}$ lies in a triangle $\Delta'$, thus
\begin{equation}\label{bang-1}
l + \l_{[1,\k+p]} = l_\Delta + O( A^2 (1+p) )
\end{equation}
and
\begin{equation}\label{bang-2}
\j_{[1,\k+p]} = \frac{s}{4} + O(s^{0.6}) = \frac{s}{4} + O(m^{0.6})
\end{equation}
thanks to \eqref{bound}.
From \eqref{ts} we then have
\begin{equation}\label{jjd}
0 \leq j+\j_{[1,\k+p]}-j_{\Delta'} \leq \frac{1}{\log 9} s_{\Delta'} - \frac{\log 2}{\log 9} (l_{\Delta'}-l-\l_{[1,\k+p]}).
\end{equation}
Suppose that the lower tip of $\Delta'$ lies well below the upper edge of $\Delta$ in the sense that
$$ l_{\Delta'} - \frac{s_{\Delta'}}{\log 2} \leq l_\Delta - 10.$$
Then by \eqref{bang-1} we can find an integer $j' = j+\j_{[1,\k+p]} + O( A^2 (1+p) )$ such that $j' \geq j_{\Delta'}$ and
$$0 \leq j'-j_{\Delta'} \leq \frac{1}{\log 9} s_{\Delta'} - \frac{\log 2}{\log 9} (l_{\Delta'}-l_\Delta).$$
In other words, $(j',l_\Delta) \in \Delta'$. But by \eqref{bang-2} we have
$$ j' = j + \frac{s}{4} + O(m^{0.6}) + O( A^2 (1+p) ) = j + \frac{s}{4} + O(m^{0.6}).$$
From \eqref{ts} we have
$$ 0 \leq (j-j_\Delta) \log 9 \leq s_\Delta - s \log 2$$
and hence (since $s \geq \frac{m}{\log^2 m}$ and $\frac{1}{4} \log 9 < \log 2$)
$$ 0 \leq (j'-j_\Delta) \log 9 \leq s_\Delta$$
Thus $(j',l_\Delta) \in \Delta$.  Thus $\Delta$ and $\Delta'$ intersect, which by Lemma \ref{black} forces $\Delta=\Delta'$, which is absurd since $(j,l) + \v_{[1,\k+p]}$ lies in $\Delta'$ but not $\Delta$ (the $l$ coordinate is larger than $l_\Delta$).  We conclude that
$$ l_{\Delta'} - \frac{s_{\Delta'}}{\log 2} > l_\Delta - 10.$$
On the other hand, from \eqref{ts} we have
$$ l_{\Delta'} - \frac{s_{\Delta'}}{\log 2} \leq l + \l_{[1,\k+p]}$$
hence by \eqref{bang-1} we have
\begin{equation}\label{joel}
 l_{\Delta'} - \frac{s_{\Delta'}}{\log 2} = l_\Delta + O( A^2 (1+p) ).
\end{equation}
From \eqref{jjd}, \eqref{joel}, \eqref{bang-1} we then have
\begin{align*}
 0 \leq j + \j_{[1,\k+p]} - j_{\Delta'} &\leq \frac{1}{\log 9} s_{\Delta'} - \frac{\log 2}{\log 9} (l_{\Delta'} - l - \l_{[1,\k+p]})  \\
&= -\frac{\log 2}{\log 9} (l_\Delta - l - \l_{[1,\k+p]} + O(A^2(1+p)) ) \\
&= O( A^2 (1+p) ).
\end{align*}
so that
$$ j + \j_{[1,\k+p]} = j_{\Delta'} + O( A^2 (1+p) ).$$
Thus, outside the event $E'$, the event that $(j,l) + \v_{[1,\k+p]}$ lies in a triangle $\Delta'$ can only occur if $(j,l) + \v_{[1,\k+p]}$ lies within a distance $O(A^2(1+p))$ of the point $(j_{\Delta'}, l_\Delta)$.

Now suppose we have two distinct triangles $\Delta', \Delta''$ in ${\mathcal T}$ obeying \eqref{joel}, with $s_{\Delta'}, s_{\Delta''} \geq s'$ with $j_{\Delta'} \leq j_{\Delta''}$.  Set $l_* \coloneqq l_\Delta + \lfloor s'/2 \rfloor$, and observe from \eqref{ts} that $(j_*,l_*) \in \Delta'$ whenever $j_*$ lies in the interval
$$ j_{\Delta'} \leq j_* \leq j_{\Delta'} + \frac{1}{\log 9} s_{\Delta'} - \frac{\log 2}{\log 9} (l_{\Delta'} - l_*)$$
and similarly $(j_*,l_*) \in \Delta''$ whenever
$$ j_{\Delta''} \leq j_* \leq j_{\Delta''} + \frac{1}{\log 9} s_{\Delta''} - \frac{\log 2}{\log 9} (l_{\Delta''} - l_*).$$
By Lemma \ref{black}, these two intervals cannot have any integer point in common, thus
$$ j_{\Delta'} + \frac{1}{\log 9} s_{\Delta'} - \frac{\log 2}{\log 9} (l_{\Delta'} - l_*) \leq j_{\Delta''}.$$
Applying \eqref{joel} and the definition of $l_*$, we conclude that
$$ j_{\Delta'} + \frac{1}{2} \frac{\log 2}{\log 9} s' + O( A^2 (1+p) )\leq j_{\Delta''}$$
and hence by \eqref{sbig}
$$ j_{\Delta''} - j_{\Delta'} \gg s'.$$
We conclude that for the triangles $\Delta'$ in ${\mathcal T}$ obeying \eqref{joel} with $s_{\Delta'} \geq s'$, the points $(j_{\Delta'}, l_\Delta)$ are $\gg s'$-separated.  Let $\Sigma$ denote the collection of such points, thus $\Sigma$ is a $\gg s'$-separated set of points, and outside of the event $E'$, $(j,l) + \v_{[1,\k+p]}$ can only occur in a triangle $\Delta'$ with $s_{\Delta'} \geq s'$ if
$$ \dist( (j,l) + \v_{[1,\k+p]}, \Sigma ) \ll A^2(1+p).$$
We conclude that
$$ \P( E_{p,s'} \wedge \bar{E'} ) \ll
\P\left( \dist( (j,l) + \v_{[1,\k+p]}, \Sigma ) \ll A^2(1+p) \right).$$

From \eqref{ston} we see that
\begin{align*}
&\P\left( (j,l) + \v_{[1,\k+p]} = (j_{\Delta'},l_\Delta) + O( A^2(1+p) ) \right)\\
&\quad\quad\ll \frac{A^2(1+p)}{s^{1/2}}
G_{1+s}\left( c(j_{\Delta'}-j - \frac{s}{4}) \right)\\
&\quad\quad\ll \frac{A^2(1+p)}{s'} \sum_{j' = j_{\Delta'} + O(s')}
\frac{1}{s^{1/2}} G_{1+s}\left( c(j'-j - \frac{s}{4}) \right).
\end{align*}
Summing and using the $\gg s'$-separated nature of $\Sigma$, we conclude that
\begin{align*}
\P\left( \dist( (j,l) + \v_{[1,\k+p]}, \Sigma ) \ll A^2(1+p) \right)
&\ll \frac{A^2(1+p)}{s'} \sum_{j' \in \Z}
\frac{1}{s^{1/2}} G_{1+s}\left( c(j'-j - \frac{s}{4}) \right) \\
&\ll \frac{A^2(1+p)}{s'}
\end{align*}
and the claim \eqref{pang-2} follows.
\end{proof}

From Lemma \ref{77} we have
$$ \P( E_{p,4^A (1+p)^3} ) \ll A^2 \frac{1}{4^A (1+p)^2} + \exp( - c A^2 (1+p) )$$
whenever $0 \leq p \leq m^{0.1}$.  Thus by the union bound, if $E_*$ denotes the union of the $E_{p,4^A (1+p)^3}$ for $0 \leq p \leq m^{0.1}$, then
$$ \P( E_*) \ll A^2 4^{-A}.$$
Next, we apply Lemma \ref{rip} with $(j',l') \coloneqq (j,l) + \v_{[1,\k]}$ to conclude that
$$
 \E 1_{R \leq \r} \exp\left( - \sum_{p=1}^{\t_{\min(\r,R)}} 1_W((j,l) + \v_{[1,\k+p]} + \eps R \right) \leq \exp(\eps),$$
where now $\r = \r((j,l) + \v_{[1,\k]})$ and $\t_i = \t_i((j,l) + \v_{[1,\k]})$.  If we then let $F_*$ to be the event that
$$
 \exp\left( - \sum_{p=1}^{\t_{\min(\r,R)}} 1_W((j,l) + \v_{[1,\k+p]} + \eps \min(\r,R) \right) > 10^{A+2} \exp(\eps)$$
then by Markov's inequality we have
$$ \P( F_*) \leq 10^{-A-2}.$$
Outside of the event $F_*$, but assuming $\r \geq R$, we have
$$ \exp\left( - \sum_{p=1}^{\t_{\min(\r,R)}} 1_W((j,l) + \v_{[1,\k+p]} + \eps R \right)  \ll 10^A$$
which implies under these hypotheses that
$$ \sum_{p=1}^{\t_{\min(\r,R)}} 1_W((j,l) + \v_{[1,\k+p]})  \gg \eps R - O( A ).$$
In particular, if we set $R \coloneqq \lfloor A^2 / \eps^4\rfloor$, we have
\begin{equation}\label{pp1}
 \sum_{p=1}^{\t_{R}} 1_W((j,l) + \v_{[1,\k+p]})  \gg \frac{A^2}{\eps^3}
\end{equation}
whenever we lie outside of $F_*$ and $\r \geq R$.

Now suppose we lie outside of both $E_*$ and $F_*$.  To prove \eqref{cheap}, it will now suffice to show the deterministic claim
\begin{equation}\label{pp2}
 \sum_{p=0}^{P-1} 1_W((j,l) + \v_{[1,\k+p]} ) > \frac{10 A}{\eps^3}.
\end{equation}
We argue by contradiction.  Suppose that \eqref{pp2} fails, thus
$$ \sum_{p=0}^{P-1} 1_W((j,l) + \v_{[1,\k+p]} ) \leq \frac{10 A}{\eps^3}.$$
Then the point $(j,l) + \v_{[1,\k+p]}$ is white for at most $10 A/\eps^3$ values of $0 \leq p \leq P-1$, so in particular for $P$ large enough there is $0 \leq p \leq 10A/\eps^3+1 = O_{A,\eps}(1)$ such that $(j,l) + \v_{[1,\k+p]}$ is black. By Lemma \ref{black}, this point lies in a triangle $\Delta' \in {\mathcal T}$.  As we are outside $E_*$,  the event $E_{p,4^A(1+p^3)}$ fails, so we have
$$ s_{\Delta'} < 4^A (1+p)^3.$$
Thus by \eqref{ts}, for $p'$ in the range
$$p + 10 \times 4^A (1+p)^3 < p' \leq P-1,$$
we must have $l + \l_{[1,\k+p']} > l_{\Delta'}$, hence we exit $\Delta'$ (and increment the random variable $\r$).  In particular, if
$$ p + 10 \times 4^A (1+p)^3 + 10 A/\eps^3 + 1 \leq P-1,$$
then we can find
$$ p' \leq p + 10 \times 4^A (1+p)^3 + 10 A/\eps^3 + 1 = O_{p,A,\eps}(1) $$
such that $l + \l_{[1,\k+p']} > l_{\Delta'}$ and $(j,l) + \v_{[1,\k+p]}$ is black (and therefore lies in a new triangle $\Delta''$).  Iterating this $R$ times, we conclude (if $P$ is sufficiently large depending on $A,\eps$) that $\r \geq R$ and that $\t_R \leq P$.  Choosing $P$ large enough so that all the previous arguments are justified, the claim \eqref{pp2} now follows from \eqref{pp1}, giving the required contradiction.  This (finally!) concludes the proof of \eqref{haha}, and hence Proposition \ref{mono-prop}.  As discussed previously, this implies Propositions \ref{jeo-prop}, \ref{key}, \ref{f-decay} and Theorem \ref{main}.

\end{document}